\documentclass[11 pt]{amsart}

\usepackage[T1]{fontenc}
\usepackage[OT2, T1]{fontenc}
\usepackage[russian,english]{babel}
\usepackage{geometry}

\usepackage{setspace}
\usepackage{verbatim}
\usepackage{mathrsfs}
\makeindex
\usepackage{dsfont}
\usepackage{amsbsy}
\usepackage{amstext}
\usepackage{amsthm}
\usepackage{amssymb}
\usepackage[hidelinks]{hyperref}
\usepackage{xcolor}
\usepackage{cancel}
\usepackage{graphicx}
\usepackage[export]{adjustbox}

\usepackage{float}
\PassOptionsToPackage{normalem}{ulem}
\usepackage[toc,page]{appendix}
\usepackage{graphicx,tikz}
\usepackage[columns=1]{idxlayout}

\usepackage{adjustbox,graphicx,subcaption}
\makeatletter
\numberwithin{equation}{section}
\numberwithin{figure}{section}
\theoremstyle{plain}
\newtheorem{thm}{\protect\theoremname}[section]
\newtheorem{definition}[thm]{Definition}
\theoremstyle{plain}
\newtheorem{remark}[thm]{Remark}
\newtheorem{cor}[thm]{\protect Corollary}
\newtheorem{exmp}[thm]{Example}
\newtheorem{prop}[thm]{\protect Proposition}
\newtheorem{lem}[thm]{\protect\lemmaname}
\newtheorem{conj}[thm]{\protect Conjecture}
\newcommand{\xlkj}{\mathbf{x}_{j, \mathbf{k}, \ell, \p}}
\newtheorem{problem}[thm]{Problem}

\def\p{{\textcolor{black} {p}}}
\def\1{{\textcolor{black} {1}}}
\def\d1{{\textcolor{black} {d-1}}}

\def\r{{\textcolor{black} r}}

\def\om{{\textcolor{black} {\omega}}}
\def\bump{{\textcolor{black} {b}}}

\def \bN {\mathbb N}

\def \bQ {\mathbb Q}
\def \bR {\mathbb R}
\def \bRn {\mathbb R^{n-1}}
\def \bZ {\mathbb Z}
\def \bZn {\mathbb Z^{n-1}}

\def \sA {\mathscr A}
\def \sB {\mathscr B}
\def \sC {\mathscr C}
\def \sD {\mathscr D}

\def \sF {\mathscr F}
\def \sG {\mathscr G}
\def \sH {\mathscr H}

\def \sK {\mathscr K}

\def \sM {\mathscr M}

\def \sP {\mathscr P}

\def \sR {\mathscr R}
\def \sS {\mathscr S}

\def \sU {\mathscr U}

\def \sW {\mathscr W}
\def \sX {\mathscr X}

\def \sZ {\mathscr Z}

\def \rc {\mathrm c}
\def \rd {\mathrm d}

\def \ri {\mathrm i}

\def \rN {\mathrm N}

\def \ba {\mathbf a}

\def \bc {\mathbf c}

\def \be {\mathbf e}
\def \bf {\mathbf f}

\def \bk {\mathbf k}

\def \bn {\mathbf n}

\def \bp {\mathbf p}

\def \bs {\mathbf s}

\def \bx {\mathbf x}

\def \by {\mathbf y}
\def \bz {\mathbf z}

\def \bzero {\mathbf 0}

\def \fa {\mathfrak a}

\def \fk {\mathfrak k}

\def \fE {\mathfrak E}

\def \fL {\mathfrak L}

\def \fN {\mathfrak N}

\def \fX {\mathfrak X}

\def \cA {\mathcal A}
\def \cB {\mathcal B}

\def \cE {\mathcal E}

\def \cH {\mathcal H}

\def \cT {\mathcal T}

\def \dim {\mathrm{dim}}

\def \det {\mathrm{det}}

\def \supp {{\mathrm{supp}}}

\def \ellpcond {$0\leq  \ell \leq  \fL_p(\delta,Q)$}
\def \ellponecond {$0\leq  \ell \leq  \fL_\1(\delta,Q)$}
\def \ellpdcond {$0\leq  \ell \leq  \fL_\d1(\delta,Q)$}

\def \ds1 {\mathds{1}}

\def \ind {\mathds{1}}

\def\om{{\textcolor{black} {\omega}}}
\def\om{{\textcolor{black} {\omega}}}
\def\Lc{{\textcolor{black} {\fL_p }}}

\def\const{{\textcolor{black} {\fk}}}

\makeatother

\usepackage{babel}
\providecommand{\lemmaname}{Lemma}

\providecommand{\theoremname}{Theorem}
\usepackage{xifthen,ifthen}
\makeatletter
\newcounter{@ToDo}
\newcommand{\todo@helper}[1]{%
	({\color{blue}TODO~\arabic{@ToDo}: {#1\@addpunct{.}}})%
}
\newcommand{\todo}[1]{\stepcounter{@ToDo}%
	\relax\ifmmode\text{\todo@helper{#1}}%
	\else\todo@helper{#1}\fi%
}
\newcounter{@cdo}
\newcommand{\cdo@helper}[1]{%
	({\color{red}CITE~\arabic{@cdo}: {#1\@addpunct{.}}})%
}
\newcommand{\cdo}[1]{\stepcounter{@cdo}%
	\relax\ifmmode\text{\cdo@helper{#1}}%
	\else\cdo@helper{#1}\fi%
}

\begin{document}

\author{Rajula Srivastava, Niclas Technau}
\title{
Density of Rational Points Near Flat/Rough
Hypersurfaces
}
\begin{abstract}
For $n\geq 3$, let $\mathscr{M} 
\subseteq\bR^{n}$ be a compact hypersurface, parametrized by a 
 homogeneous function of degree 
$d\in \bR_{>1}$, with non-vanishing curvature away from the origin. Consider the number
$\rN_\sM(\delta,Q)$
of rationals
$\ba/q$, with denominator 
$q\in [Q,2Q)$ and $\ba \in \bZn$, lying at a distance at most
$\delta/q$ from $\sM$.
This manuscript provides
essentially sharp estimates
for $\rN_\sM(\delta,Q)$
throughout the range $\delta \in (Q^{\varepsilon-1},1/2)$ for $d>1+\tfrac{1}{2n-3}$.
Our result
is a first of its kind 
for hypersurfaces with vanishing 
Gaussian curvature ($d>2$) and 
those which are rough (meaning not even $C^2$ 
at the origin which happens when $d<2$). 
An interesting outcome of our investigation
is the understanding of a `geometric' term 
$(\delta/Q)^{(n-1)/d}Q^n$
(stemming from a so-called Knapp cap), 
arising in addition to the usual probabilistic term 
$\delta Q^n$; the sum of these terms 
determines the size of $\rN_\sM(\delta,Q)$
for $\delta\in(Q^{\varepsilon-1},1/2)$. Consequences of our result concern 
the metric theory of 
Diophantine approximation
on `rough' hypersurfaces
--- going beyond
the recent break-through of
Beresnevich and L. Yang. 
Further, we establish smooth 
extensions of
Serre's dimension growth conjecture.
 

\tableofcontents{}

\end{abstract}


\address{Rajula Srivastava; \newline Former Address: Department of Mathematics, 
University of Wisconsin 480 Lincoln Drive, Madison, WI, 53706, USA
\newline New Address: Mathematical Institute, University of Bonn, Endenicher Allee 60,
53115, Bonn, Germany, and
\newline Max Planck Institute for Mathematics, Vivatsgasse 7,
53111, Bonn,
Germany.}
\email{rsrivastava9@wisc.edu \quad rajulas@math.uni-bonn.de}

\address{Niclas Technau; \newline 
Former Address: The Division of Physics, Mathematics and Astronomy, 
Caltech, 1200 E California Blvd, Pasadena, CA, 91125
\newline 
New Address: Graz University of Technology, 
Institute of Analysis and Number Theory, Steyrergasse 30/II, 8010 Graz, Austria}
\email{ntechnau@caltech.edu \quad ntechnau@tugraz.at}

\maketitle

\section{Introduction}
We begin by introducing the 
necessary amount of notation 
to state the main theorem  
quickly and 
contextualize the result thereafter.
Let $\mathscr{M}\subseteq
\mathbb{R}^{n}$ \index{dimension@$n$: dimension of
the ambient space}
be a compact $C^{1}$-manifold
of dimension $m$, $\Vert
\cdot\Vert_{2}$\index{Euclideannorm@
$\Vert\cdot\Vert_2$: Euclidean norm} 
be the Euclidean norm 
on $\bR^n$. Throughout, $n$ denotes the dimension of the ambient Euclidean space.
The distance of a point 
$\mathbf{y}\in\mathbb{R}^{n}$
to $\mathscr{M}$ is defined by 
\[
\mathrm{dist}(\mathbf{y},\mathscr{M})
:=\inf_{\mathbf{x}\in\mathscr{M}}
\Vert\mathbf{x}-\mathbf{y}\Vert_{2}.
\]
For functions $A, B: \sX\to\mathbb{C}$
with $\sX \subseteq \mathbb{R}^k$, we shall use the Vinogradov notation $A\ll B$ to mean that $|A(\bx)|\leq C|B(\bx)|$ 
is true all $\bx \in \sX$ for some positive constant $C>1$. We shall write $A\asymp B$ to denote that $A\ll B$ and $B\ll A$. Further, the characteristic function of a set $\sX$ shall be denoted by $\mathds{1}_{\sX}$.

Given $Q\geq 1$ and $\delta\geq 0$,
we are interested in counting the number of rational
points with
denominator $q\in[Q,2Q)$, located in a $\delta/q$-neighbourhood
of $\mathscr{M}$. This is captured by the counting function
\begin{equation}
\mathrm{N}_{\mathscr{M}}
(\delta,Q)
:=
\#\big\{
(q,\bp) \in \bZ \times \bZ^{n}
:\,q\in[Q,2Q),\,
\mathrm{dist}
\Big(
\frac{\mathbf{p}}{q},
\mathscr{M}
\Big)
\leq
\frac{\delta}{q}
\big\}.\label{def: counting rational points near manifolds}
\end{equation}
\index{counting_function@$\mathrm{N}_{\mathscr{M}}
(\delta,Q)$: unsmoothed count of rational points near
the manifold $\sM\subseteq \bR^n$}
We make some standard reductions. 
Throughout, we denote the codimension
of $\mathscr{M}\subseteq \bR^n$ by
$c:=n-m$ . Using compactness, we can
cover $\mathscr{M}$ 
by a finite collection of open charts. 
Thus, using the implicit function theorem, 
we may assume without loss of generality
that $\mathscr{M}$ can be represented in the
Monge form as the graph 
$$ \{(\mathbf{x}, \mathbf{f}(\mathbf{x})):
\mathbf{x}\in\mathscr{B}_m\}$$
over an open ball $\sB_m\subseteq \bR^{m}$ of a $C^{1}$ vector valued function 
$\mathbf{f}:=(f_{1},\ldots,f_{c})$. 
Indeed by translation and re-scaling,
if necessary, 
we can arrange that 
\begin{equation}\label{def: normalised Monge}
    \sM = \{(\mathbf{x},
\mathbf{f}(\mathbf{x})):
\mathbf{x}\in\sU_m\},
\end{equation}
where 
\index{unitball@$\sU_m$: the $m$-dimensional (closed)
Euclidean unit ball}
$$
\sU_m:=
\{\bx \in \bR^m: 
\Vert\bx\Vert_2 <1
\}
$$
is the unit ball
in $\bR^m$ around the origin.
If \eqref{def: normalised Monge} holds,
then $\sM$ is said to be in 
\emph{normalised Monge form}. 
An standard application of Taylor's 
theorem shows that one can control the counting function
\eqref{def: counting rational points near manifolds}
in terms of 
\index{counting@$\mathrm{N}_{\mathscr{\mathbf{f}}}
(\delta,Q)$: unsmoothed count of rational points
near the graph of $\mathbf{f}$}
$$
\mathrm{N}_{\mathscr{\mathbf{f}}}
(\delta,Q):=\sum_{
(q,\mathbf{a})\in\bZ\times \mathbb{Z}^{m}}
\ind_{[1,2)}
\big(\frac{q}{Q}\big)
\ind_{\sU_m} 
\big(\frac{\mathbf{a}}{q}\big)
\mathds{1}_{[0,\delta]}
\big(
\Vert qf_{1}
\big(\frac{\mathbf{a}}{q}\big)
\Vert
\big)
\ldots\mathds{1}_{[0,\delta]}
\big(
\Vert qf_{c}
\big(\frac{\mathbf{a}}{q}\big)
\Vert
\big)
$$
where $\Vert x \Vert:=
\min\{\vert x -z\vert : z\in \bZ\}$ 
abbreviates the distance 
to the nearest integer. \index{nearestinteger@
$\Vert \cdot \Vert$: distance to the nearest integer}
How large is
$\mathrm{N}_{\mathscr{\mathbf{f}}}(\delta,Q)$ 
expect to be? Fixing $q\in [Q,2Q)$, 
we model each function
$\ba \mapsto \Vert qf_{i}(\ba /q)\Vert$ 
by a uniformly distributed 
random variable $R_{q,i}:\sU_m\rightarrow[0,1/2]$. 
Assume that the random variables are 
independent of each other.
There are about $Q^m$ many relevant choices 
for $\ba$, and $Q$ many choices for $q$.
Because of mutual independence, 
the constraint $$\max_{t\leq c}\Vert R_{q,t}
\Vert<\delta$$ holds with probability 
$\delta^{c}$. 
Thus it is reasonable to {\emph{expect}} that 
\begin{equation}\label{eq: heuristic guess}
 \mathrm{N}_{\mathscr{\mathbf{f}}}(\delta,Q)
 \asymp
\delta^{c}Q^{m+1}.
\end{equation}
However, this random model 
cannot predict
the size of $ \mathrm{N}_{\mathscr{\mathbf{f}}}(\delta,Q)$ correctly when $\delta$ is rather
small in terms of $Q$, 
because $\mathrm{N}_{\bf}(\delta,Q)$
counts, in particular, the
points \emph{lying} \emph{on} 
the manifold ---
there might be relatively 
many such points, 
or none at all!
Let us illustrate these
situations by well--known examples,
see Beresnevich \cite[p. 189]{Beresnevich Rational points near manifolds}.

\begin{exmp}
The cylinder 
$\sZ := \{(x_1,\ldots,x_n)\in \bR^n: 
x_1^2+x_2^2 = 3\}$
is readily seen to contain
no point of $\bQ^n$ at all. 
A short computation, in combination with the fact that
any non-zero integer is at least one 
in absolute value, shows that any rational 
point contributing to
$\rN_\sZ(\delta, Q)$
when $\delta =o(Q^{-1})$
must in fact be a point on $\sZ$.
Hence $\rN_\sZ(\delta, Q)=0$
in the range $\delta \in  
(0,(Q \log Q)^{-1})$,
once $Q$ is sufficiently large
---
violating 
\eqref{eq: heuristic guess}
dramatically.
\end{exmp}
On the other hand,  
we can easily generate a class of manifolds 
with a good amount of rational points
on them.
\begin{exmp}
Let     
$$
\sM_{m,n}:=\{(x_1,\ldots,
x_{m-1}, x_m,x_{m}^2,\ldots,
x_{m}^{n-m+1})\in \bR^n: 
x_1^2+ \ldots + x_m^2 \leq  1\}.
$$
Given a denominator 
$q \in [Q,2Q)$, 
the point $\bp/q$
with $$\bp:=(p_1,\ldots,
p_{m-1},0,\ldots,0)\in \bZ^n$$
is on $\sM$. 
Thus, $$
\rN_{\sM_{m,n}}(\delta,Q)\geq N(0,Q)
\geq Q^m
$$
for any $\delta\geq 0$.
So, $\rN_{\sM_{m,n}}
(\delta,Q)\asymp
\delta^{c}Q^{m+1}$
can only be true if $\delta^c Q^{m+1} \gg
Q^m$, that is
$\delta \gg Q^{-\frac{1}{c}}$.
\end{exmp}

Hypersurfaces
in $\bR^n$ (or manifolds of dimension $m=n-1$),
are the focal point of 
our manuscript. 
In this case, we can write
\begin{equation}\label{eq; counting function hypersurfaces}
    \rN_{f}(\delta,Q)
=\sum_{
(q,\mathbf{a})\in\bZ
\times \mathbb{Z}^{n-1}}
\ind_{[1,2)}(\frac{q}{Q})
\ind_{\sU_m} 
(\frac{\mathbf{a}}{q})
\mathds{1}_{[0,\delta]}
(\Vert qf(\frac{\mathbf{a}}{q})
\Vert).
\end{equation}
In light of the above examples,
one might hope that the following is true for hypersurfaces:
if there is no ``local obstruction'', 
then the probabilistic 
guess \eqref{eq: heuristic guess} 
is accurate up to a $Q^\varepsilon$
of room in the expected range for $\delta$, meaning that
\begin{equation}\label{eq: guess hypersurfaces}
\mathrm{N}_{f}(\delta,Q)
 \asymp
\delta Q^{n} \,\,
\mathrm{uniformly\, for\,any\,} 
Q\geq 1,\, \mathrm{and}\,
\delta\in (Q^{\varepsilon-1},1/2)
\end{equation}
where $\varepsilon>0$ is fixed.
The implied
constants are allowed to depend on $f$ 
and $\varepsilon$.

The term ``local obstruction''
is deliberately vague. While there is  
some understanding of geometric 
properties (described in due course)
which have to be excluded
for \eqref{eq: guess hypersurfaces} to be true,
it is currently not known 
what a precise characterisation 
all such local obstructions could be. 
Indeed, it is the goal of our
manuscript to investigate this. 
In a recent break-through,
J.-J. Huang showed the following:
\begin{thm}[\cite{Huang rational points}]\label{thm: Huang}
Let $n\geq 3$ and $\sM$ be 
immersed by $f: \sU_{n-1} \rightarrow \bR$. Suppose that $f$ is 
$\max(\lfloor \frac{n-1}{2}\rfloor +5, 
n+1)$-many times continuously differentiable. Let
\index{Hessian@$H_f$: the Hessian matrix of $f$}
$$
H_{f}(\bx):= 
\Big( \frac{\partial^2}
{\partial x_i \partial x_j}
f(\bx)
\Big)_{i,j\leq n-1}
$$ 
be the Hessian of $f$ at 
$\mathbf{x}\in\sU_{n-1}$.  
If the Gaussian curvature of $\sM$ is bounded away from zero; 
that is, if the determinant of $H_f(\bx)$
does not vanish for any $\bx \in \sU_{n-1}$,
then \eqref{eq: guess hypersurfaces}
holds.
\end{thm}
In other words, uniform curvature information (and no local flatness) rules 
out any local obstruction for hypersurfaces. 
Thus \eqref{eq: guess hypersurfaces} is true for a wide class of hypersurfaces in all dimensions $n\geq 3$. 
The main innovation in \cite{Huang rational points}
was an elegant bootstrap procedure relying on projective duality and the method of stationary phase from harmonic analysis.

On the other hand,
our main result (Theorem \ref{thm: main} 
below) establishes a \textit{new
heuristic} for a rich class of flat or rough
hypersurfaces.
Indeed we handle hypersurfaces whose Hessian determinant
(along with several other higher order derivatives) 
either vanishes or blows up at an isolated point, say the origin.
To handle this, we have to introduce
an additional new term to account for this 
local flatness/roughness. We shall see
that this term arises naturally 
from the underlying harmonic analysis 
but also from purely geometric considerations.

\subsection{Main Results} In order to describe 
the class of hypersurfaces we consider, 
some more notation is needed.
Let $\mathscr{H}_d(\bRn)$\index{classhom@$\mathscr{H}_d(\bRn)$: set of homogeneous functions 
$f: \bRn \rightarrow \bR$ 
of degree $d\in \bR$}
denote the collection of  
functions 
$f:\bRn \rightarrow \bR$
which are homogeneous 
of degree $d$, that is \index{degree@$d$: the degree of homogeneity of a function $f\in \mathscr{H}_d(\bRn)$}
$f(\lambda \bx)
= \lambda^d 
f(\bx)$ for all 
$\bx \in \bRn$
and all $\lambda \geq 0$.

The word {\emph{smooth}} 
means, as usual, `infinitely many times
continuously differentiable'. 
\begin{definition}
For a real number 
$d\neq 0 $, 
denote by
$\mathscr{H}_d^{
\mathbf{0}}(\bRn)$ 
the set of all 
smooth functions 
$f\in \mathscr{H}_d(\bRn)$ 
whose Hessian $H_{f}(\bx)$ remains invertible
away from the origin; that is
\index{classhomzero@$\mathscr{H}_d^{\bzero}(\bRn)$: set of smooth  $f\in \mathscr{H}_d^{\bzero}(\bRn)$ 
whose is non-zero Hessian away from $\bzero$}
\begin{align*}
 \mathscr{H}_d^{\mathbf{0}}(\bRn)
    := \{&f\in \mathscr{H}_d(\bRn): 
    f \mathrm{\,\,is \,\, smooth\,\,on\,\,} 
    \bRn\setminus\{\bzero \},\,
    \mathrm{and}\,\\
    &\det\, H_{f}(\mathbf{x})
    \neq 0 \mathrm{\,\,if\,\,} 
    \mathbf{x} \neq  \bzero
    \}.    
\end{align*}
\end{definition}

The simplest example of a
function in
$\sH_d^{\mathbf{0}}
(\bRn)$ to keep in mind
is $f(\bx):= 
\Vert \bx \Vert_2^d$. We are now ready to state our main result.

\begin{thm}[Main Theorem]
\label{thm: main}Let $n\in \bZ_{\geq 3}$. 
Let $d>\tfrac{2(n-1)}{2n-3}$ be a real number.
Fix $\varepsilon>0$,  
and $f\in \mathscr{H}_d^{\mathbf{0}}(\bR^{n-1})$. 
Then 
\begin{equation}\label{eq: main bounds 1}
\rN_f(\delta,Q) \asymp 
\delta Q^n
+ 
\Big(\frac{\delta}{Q}\Big)^{
\frac{n-1}{d}}Q^{n}
\,\,\mathrm{for\,any}\,\, Q\geq 1 \,\,
\mathrm{and}\,\, \delta \in 
(Q^{\varepsilon-1}, 1/2),
\end{equation}
provided $\tfrac{2(n-1)}{2n-3}<d <n-1$.

If $d \geq n-1$, then
\begin{equation}\label{eq: main bounds 2}
\delta Q^n
+ 
\Big(\frac{\delta}{Q}\Big)^{
\frac{n-1}{d}}Q^{n}
\ll
\rN_f(\delta,Q) \ll
\delta Q^n
+ 
\Big(\frac{\delta}{Q}\Big)^{
\frac{n-1}{d}}Q^{n+k \varepsilon} 
\end{equation}
for any $Q\geq 1$ and $\delta \in 
(Q^{\varepsilon-1}, 1/2)$, with $k=k(n)>0$ depending only 
on $n$.
The other implied constants 
in this theorem depend 
on $\varepsilon,f,n$, and $d$,
but not on $\delta$ or $Q$.
\end{thm}
\begin{remark}
In fact, we prove a 
more general (smoothed) version of the above theorem 
which also includes the range 
$1<d\leq  \tfrac{2(n-1)}{2n-3}$, see 
Theorem \ref{thm: main smoothed}.
However, the upper bound we obtain in this range 
is not sharp, cf. Remark \ref{rem exclusion}.
\end{remark}
An interesting consequence of the above theorem is that the original heuristic \eqref{eq: guess hypersurfaces} is indeed true for a large subclass of the locally flat 
hypersurfaces we consider!  Thus having 
non-vanishing Gaussian curvature 
is a sufficient 
but, by far, not a necessary condition 
for \eqref{eq: guess hypersurfaces} 
to be valid. On the other hand, the `flat term' 
that we introduce is indeed required 
to count rational points near 
manifolds corresponding to functions 
of large enough homogeneity, 
as a straightforward reformulation of our 
theorem shows.
\begin{cor}\label{cor: counting explicit}
Keep 
the notation and assumptions as in Theorem 
\ref{thm: main}. Let $\varepsilon>0$
be small in terms of $d$
and $n$.\\
(i) 
If $ \tfrac{2(n-1)}{2n-3}<d\leq 2(n-1)$, then
$$
\rN_{f}(\delta,Q) \asymp \delta Q^n
\,\,\mathrm{for}\,\, Q\geq 1 \,\,
\mathrm{and}\,\, \delta \in 
(Q^{\varepsilon-1}, 1/2).
$$
(ii)
If $d> 2(n-1)$, then
$$ 
\Big(\frac{\delta}{Q}\Big)^{
\frac{n-1}{d}}Q^{n}
\ll
\rN_{f}(\delta,Q) \ll 
\Big(\frac{\delta}{Q}\Big)^{
\frac{n-1}{d}}Q^{n} Q^{k\varepsilon}
\,\,  \mathrm{for}\,\, Q\geq 1 \,\,
\mathrm{and}\,\, \delta \in 
(Q^{\varepsilon-1}, 
Q^{- \frac{n-1-k\varepsilon d}{d-(n-1)}}) 
$$
while
$$
\rN_{f}(\delta,Q) \asymp \delta Q^n
\,\, 
\mathrm{for}\,\, Q\geq 1 \,\,
\mathrm{and}\,\, \delta \in 
[Q^{- \frac{n-1-k\varepsilon d}{d-(n-1)}},1/2).
$$
All the implied constants may depend 
on $\varepsilon,f,n,d$ but not on $\delta$, or $Q$.
\end{cor}
\begin{proof}
The claims follow directly from 
Theorem \ref{thm: main}
once we know 
where 
$\cA(\delta):= \delta Q^n $
dominates 
the second term $\cB(\delta):= 
(\frac{\delta}{Q})^{
\frac{n-1}{d}}Q^{n + k\varepsilon}$
and thus dictates the size
of the right hand side of 
\eqref{eq: main bounds 1}.
It is readily seen that 
the set of $\delta>0$ for which
$\cA(\delta)=\cB(\delta)$
is a singleton and 
the unique $\delta$
realising this equality is $$
\delta_{n,d} := 
Q^{- \frac{n-1-k\varepsilon d}{d-(n-1)}}.
$$
Upon observing 
$\cA(1)=Q^n>Q^{n-\frac{n-1}{d}
+k\varepsilon}=\cB(1)$,
we conclude that $\cA(\delta)$ dominates over $\cB(\delta)$
when $0<\delta<\delta_{n,d}$.
Further $\cB(\delta)>\cA(\delta)$ 
on the interval $\delta >\delta_{n,d}$.
A straightforward case by case
analysis completes the proof.
\end{proof}

To illustrate 
the underlying geometry,
we sketch the hypersurfaces
immersed into $\bR^3$
by 
$$
F_d(x_1,x_2):=
    (x_1^2+x_2^2)^{
    \frac{d}{2}
    }\in  
    \mathscr{H}_d^{
    \mathbf{0}}(\bR^2)
$$
for various values of $d$.
First we consider $d=2$
which is the 
only value of $d$ 
leading to a hypersurface
with (everywhere) non-vanishing 
Gaussian curvature, 
see Figure \ref{fig:d2}.
\begin{figure}[ht]
\captionsetup{justification=centering}
\includegraphics[width=.2
\linewidth,height=.2
\textheight,keepaspectratio]
{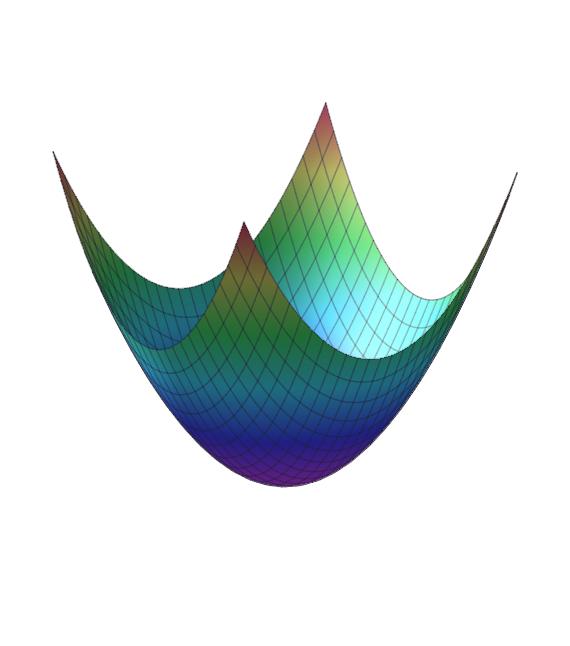}
    \caption{The graph of 
    $F_2$.}
    \label{fig:d2}
\end{figure}
Indeed, values of $d>2$ lead 
to hypersurfaces 
whose Gaussian curvature
vanishes at the origin,
see Figure \ref{fig:sfig1}.
On the other hand, 
$d<2$
leads to unbounded 
Gaussian curvature 
near the origin
(and indeed the hypersurface 
is not even $C^2$), 
unless $d=0$,
compare Figure \ref{fig:sfig2}.
\begin{figure}[ht]
\begin{subfigure}{.5\textwidth}
  \centering
  \includegraphics[width=.5\linewidth]{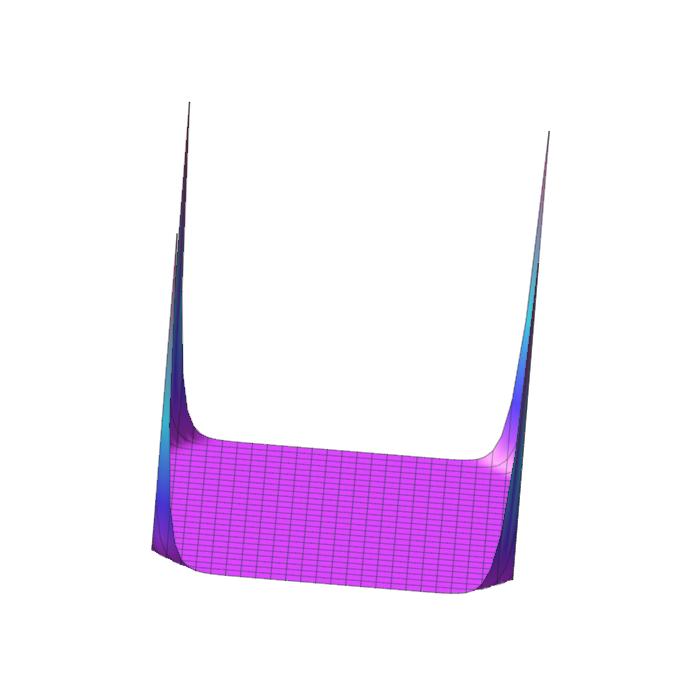}
  \caption{The graph of 
    $F_{40}$.}
  \label{fig:sfig1}
\end{subfigure}%
\begin{subfigure}{.5\textwidth}
  \centering
  \includegraphics[width=.5\linewidth]{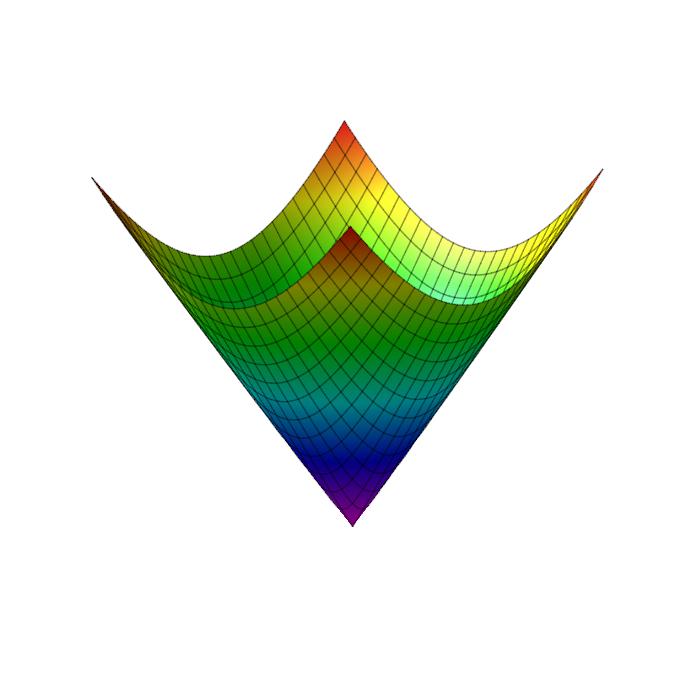}
  \caption{The graph of 
    $F_{\frac{40}{39}}$.}
  \label{fig:sfig2}
\end{subfigure}
\caption{Figure (A) depicts $F_d$ when
$d=40$ is `large' in terms of
the ambient dimension $3$. 
Let $d'$ be such that
$\frac{1}{d}+\frac{1}{d'}=1$.
As $d$ is large, the graph of 
$F_{d'}$ --- as seen in Figure (B)
--- is rough (note that $d'$ is close to $1$). 
This phenomenon 
is useful to observe in what follows.
}
\end{figure}

\subsection{State of the Art}
The probabilistic guess 
\eqref{eq: heuristic guess} 
has long manifested itself into
different folklore conjectures.
In order to state these
in a precise form,
it is crucial to observe 
the following.
As alluded to earlier,
the random behaviour 
leading to the guess
\eqref{eq: heuristic guess} might not
emerge if the geometry (curvature) 
of the manifold obstructs it. For example, 
the manifold might locally 
behave like a rational hyper-plane.
So, a basic (open-ended) conjecture 
is the following.
\begin{conj}[Main Conjecture,
see 
\cite{Huang rational points}]\label{conj: Huang main}
Let $\mathscr{M}\subseteq \mathbb{R}^n$ 
be a compact manifold
of dimension $m$ and
co-dimension $c=n-m$.
If $\sM$ satisfies
proper curvature conditions, 
then there exists a constant
$c_{\mathscr{M}}>0$ so that
\begin{equation}
\mathrm{N}_{\mathscr{M}}
(\delta,Q)\sim c_{\mathscr{M}}
\delta^{c}Q^{m+1}
\qquad(Q\rightarrow\infty)\label{eq:  rational  main conjectur}
\end{equation}
holds, 
for every fixed $\epsilon>0$,
uniformly in the range
\begin{equation}
\delta\in(Q^{\epsilon-\frac{1}{c}},1/2).\label{eq: conjectured range delta}
\end{equation}
\end{conj}
The above conjecture appears also 
in the work of Beresnevich and Kleinbock, 
see \cite[Problem 4]{BK}.
Of course there are guesses around
what exactly those 
`proper curvature conditions'
mentioned in 
Conjecture \ref{conj: Huang main} 
should be.
In order to state these precisely,
we need more notation.
Denote by $\be_i:=(0,\ldots, 0,
1,0\ldots, 0)\in\bR^n$
the $i^{th}$ standard unit vector
in $\bR^n$. We say $\mathscr{M}$
is \emph{$l$-nondegenerate} \index{nondegenerate manifold}
\emph{at $\mathbf{x}_0$} 
if the derivatives of order at most $l$
of the map $\mathbf{x}\mapsto(\mathbf{x},
\mathbf{f}(\mathbf{x}))$ 
in the directions $\be_1,\ldots,\be_n$ 
span $\mathbb{R}^{n}$. Geometrically this means that
the order of contact of $\mathscr{M}$ 
with any hyperplane is at
most $l$. Now we can state: 
\begin{conj}[\cite{Huang Subspaces}]
\label{conj: Rational Main Conjecture }
If $\mathscr{M}$
is $m+1$ 
non-degenerate everywhere, 
then 
\eqref{eq:  rational  main conjectur}
holds in the range 
\eqref{eq: conjectured range delta}.
\end{conj}
\begin{remark}
Being non-degenerate (of some order) 
does not ensure a manifold satisfies
\eqref{eq:  rational  main conjectur}
throughout  
\eqref{eq: conjectured range delta}.
Example 4 in
\cite{Huang rational points}
points out that the Fermat curve 
$\sF_d$ in $\bR^2$,
defined by $x_1^d + x_2^d=1,$
has the property that if
$\delta \in (0, Q^{-\frac{1}{d-1}})$
then $N_{\sF_d}(\delta,Q)
\asymp \delta Q^2 $
cannot be true. 
The threshold $Q^{-\frac{1}{d-1}}$
for the counting function 
$N_{\sF_d}(\delta,Q)$
to change its behaviour
is expected.
One may compare this threshold with 
the one from 
Corollary \ref{cor: counting explicit}
and take $n=2$ there, 
which is however not allowed 
because of the requirement $n\geq 3$.\\
Further, our Theorem \ref{thm: main}
establishes that
various non-degenerate (homogeneous)
hypersurfaces satisfy the conclusion 
of Conjecture \ref{conj: Huang main}
if and only if the threshold $2(n-1)$ 
is not crossed by the degree $d$. 
Roughly speaking, Theorem \ref{thm: main}
thus verifies Conjecture 
\ref{conj: Rational Main Conjecture }
for a class of hypersurfaces 
which are allowed to be twice as degenerate 
as originally asked for.
\end{remark} 

In the following the implied constants
are, without further mention, 
only allowed to depend on the ambient 
manifold (and the dimension) and $\varepsilon$ where applicable.
To ease the exposition, 
we shall usually not state 
the smoothness 
assumptions on the manifolds
which are needed for the various 
results to hold. The reader is referred to
the relevant papers.

\subsubsection*{Planar Curves}

Huxley \cite{Huxley} was the first to
establish general upper bound for \\ $\mathrm{N}_{\mathscr{M}}(\delta,Q)$.
Let $\mathscr{M}$ be a planar curve with curvature bounded between
positive constants. Then Huxley showed 
$\mathrm{N}_{\mathscr{M}}
(\delta,Q)\ll
\delta^{1-\varepsilon} Q^{2}$
for all $\delta\in(Q^{-1},1/2)$ 
and for any $\varepsilon>0$. 
The sub-optimal factor 
$\delta^{-\varepsilon}$ 
was removed by 
Vaughan and Velani 
\cite{Vaughan Velani} 
in a seminal work. 
On the other hand, Beresnevich, Dickinson, and Velani 
\cite{BDV limsup sets} obtained
the complementary bound $\mathrm{N}_{\mathscr{M}}(\delta,Q)\gg\delta Q^{2}$ uniformly in
$\delta\in(Q^{-1},1/2)$. 
These results require 
at least $C^1$-smoothness.
A notable exception
is the work of Beresnevich and Zorin
\cite{BZ}. In general, 
the theory of counting rational points 
near rough manifolds, 
i.e. which are not $C^2$,
is barely existent. 
For counting rationals in lowest terms,
see S. Chow \cite{Chow}.
For further reading, see \cite{Huang: Planar curves}.

\subsubsection*{Higher Dimensions}

Beresnevich's 
breakthrough 
\cite{Beresnevich Rational points near manifolds} 
provided the first 
general lower bounds:
if $\mathscr{M}$ 
is any analytic 
(and non-degenerate) manifold, 
then the lower bound implicit in
\eqref{eq:  rational  main conjectur}
holds. 
The theory of general upper 
bounds is far murkier. 
In a nutshell,
only sufficiently smooth and sufficiently curved hypersurfaces are
completely understood --- save for some rather special exceptions. 

J.J. Huang \cite{Huang rational points} advanced the subject substantially 
by showing Theorem \ref{thm: Huang}.
For higher co-dimensions, 
Schindler and Yamagishi 
\cite{ Schindler Yamagishi} showed
recently that \eqref{eq: heuristic guess} holds true
for a manifold $\mathscr{M}$ 
of dimension $m$ (and codimension $c$)
provided the following
curvature condition holds:
\begin{equation}
\det \,H_{t_{1}f_{1}+\ldots+t_{c}f_{c}}(\bx)= 0 \label{eq: Schindler-Yamagishi curvature condition}
\end{equation}
for $\bx \in \sU_m $
implies that $(t_{1},\ldots,
t_{c})\in\mathbb{R}^{c}$ 
has to be the zero vector 
in $\mathbb{R}^{c}$. 
The condition \eqref{eq: Schindler-Yamagishi curvature condition}
restricts the co-dimension $c$
significantly, in particular, 
it is known that 
$c\leq\rho_{\mathrm{RH}}(m)$ 
where $\rho_{\mathrm{RH}}$
is the Radon--Hurwitz 
function (e.g. 
$\rho_{\mathrm{RH}}(2t)=2$
for odd $t$). We refer to 
Roos, A. Seeger, 
and R. Srivastava 
\cite[Page 4]{RSS Heisenberg} for
further discussion.

A remarkable feature of Schindler and Yamagishi's result
is that \eqref{eq: conjectured range delta} 
holds in a wider range than
\eqref{eq:  rational  main conjectur}! The condition (\ref{eq: Schindler-Yamagishi curvature condition})
was recently weakened by 
Munkelt \cite{Munkelt} 
to just
requiring that $\mathrm{rank}
(H_{t_{1}f_{1}+\ldots+
t_{c}f_{c}}(\bx))
 \geq n-1$.
 

\section{Applications}
\label{sec: apps}
This section gathers two applications 
of Theorem \ref{thm: main}. 
Rather than working with the functions 
from the class 
$ \sH_d^{\mathbf{0}}(\bRn)$
it is more natural at present 
to work with their graphs, i.e. the manifolds 
they parameterise (in their normalised
Monge form). 
The next definition formalises this. 
\begin{definition}
Let $d\neq0$ and $n\geq 2$. We define
$\sG( \sH_d^{\mathbf{0}}(\bRn))$ 
to be the set of hypersurfaces 
$\sM \subset \mathbb{R}^{n-1}$ with
normalised Monge parametrisation given by
\eqref{def: normalised Monge}
for some 
$f\in \sH_d^{\mathbf{0}}(\bRn)$.
\end{definition}
The subsequent section 
details applications of our main results:
firstly to the metric theory of 
Diophantine approximation 
(Khintchine's Theorem on 
hypersurfaces,
including Hausdorff dimension
refinements)
and secondly to arithmetic geometry 
(smooth variants of Serre's
Dimension Growth Conjecture).
Thirdly, we 
indicate applications 
to estimating exponential sums 
(via the large sieve).

\subsection{Metric Diophantine Approximation on Manifolds}

At its ancient origin, Diophantine approximation
is concerned with how well a given
real numbers can be approximated by simpler numbers, for instance
by rational numbers with bounded denominator. The word `metric'
here abbreviates `Lebesgue measure theoretic', indicating that one is
only interested in statements which are true for Lebesgue almost all
parameters. 

A foundational result in the area is Khintchine's theorem. This elegant theorem
quantifies how well a generic number $x\in[0,1)$ can be approximated
by a fraction $p/q$.\footnote{Notice that approximating by rationals is a $1$-periodic operation.
Hence $[0,1)$ is the natural probability space to work with.} We express, as one traditionally does, the upper bound on the approximation
error as a function $\psi(q)$ 
of the denominator $q$.  
We say $\psi$ is an approximation function, 
if $\psi:\bZ_{\geq 1} \rightarrow (0,1/2)$
is non-increasing. 
Define the set of $\psi$-approximable numbers 
as $$ 
\mathscr{W}(\psi):=\{x\in[0,1]:\,\Vert qx\Vert<\psi(q)\,\mathrm{for\,infinitely\,many\,}q\geq1\}.
$$
Denoting by $\mathrm{meas}_n$ \index{meas@$\mathrm{meas}_n$: the $n$-dimensional
Lebesgue measure}
the $n$-dimensional 
Lebesgue measure on $[0,1]^n$, we can now state:
\begin{thm}[Khintchine's theorem; 1926]
Let $\psi$ be an approximation function. Then 
\[
\mathrm{meas}_1(\mathscr{W}(\psi))
=\begin{cases}
1, & \mathrm{if}\,\sum_{q\geq1}\psi(q)\,\mathrm{diverges},\\
0, & \mathrm{if}\,\sum_{q\geq1}\psi(q)\,\mathrm{converges}.
\end{cases}
\]
\end{thm}

\begin{remark}
In particular, Khintchine's theorem states that the Lebesgue measure
of $\mathscr{W}(\psi)$ satisfies a zero--one--dichotomy. In view
of the $1^{st}$ and $2^{nd}$ 
Borel--Cantelli lemma, from basic
probability theory, 
the previous theorem 
is essentially a consequence
of the statement that 
the family of functions 
$x\mapsto\Vert qx\Vert$
is `quasi-independent' 
on the interval $[0,1]$. 
This perspective
explains at once 
why we should, indeed, 
expect $\mathscr{W}(\psi)$ to have measure either zero or one,
and why the convergence of 
$\sum_{q\geq1}\psi(q)$ is relevant.
\end{remark}
Leaving aside how well 
a generic number is approximable, 
it is of interest to understand 
the set of numbers which 
are \emph{very well approximable}, 
that is the set $\sW (\psi_{z})$ where
$\psi_z(y):=y^{-z}$ and $z>1$. 
An application of the pigeonhole
principle,
which is often called Dirichlet's 
approximation theorem, reveals that
$\mathscr{W}(\psi_{z})=[0,1]$ for 
$z\in(0,1]$. So we confine our attention to $z>1$.
In light of Khintchine's theorem,
$\mathscr{W}(\psi_{z})$ is a null set. 
The foundational result of Jarn\'{i}k 
quantifies its size by computing
its Hausdorff dimension $\dim_{\mathrm{Haus}}$.
\begin{thm}[Jarn\'{i}k's theorem; 1931]
If $z>1$, then 
$\dim_{\mathrm{Haus}}
(\mathscr{W}(\psi_{z}))
=\frac{2}{1+z}$.
\end{thm}

Higher dimensional versions of Khintchine's theorem and Jarn\'{i}k's theorem
concerning the set 
\[
\mathscr{W}_{n}(\psi):=\{\mathbf{x}\in[0,1]^n:\,\max_{1\leq i\leq n}\Vert qx_{i}\Vert<\psi(q)\,\mathrm{for\,infinitely\,many\,}q\geq1\}
\]
of \emph{simultaneously $\psi$-approximable vectors} 
are known, see the survey article \cite{B Recent}.
The word simultaneous comes from
the fact that one tries to 
approximate real vectors
by rational vectors with the very same denominator $q$.
\begin{remark}
Another natural generalisation of 
$\mathscr{W}(\psi)$ is the set
of the dual $\psi$-approximable vectors, 
that is
\[
\mathscr{W}_{n}^{\perp}(\psi):=\{\mathbf{x}\in[0,1]^n:
\,\Vert\left\langle \mathbf{a},
\mathbf{x}\right\rangle \Vert
<\psi(\Vert\mathbf{a}
\Vert_{2})\,
\mathrm{for\,infinitely
\,many\,}\mathbf{a}\in
\mathbb{Z}^{n}\}.
\]
This notion quantifies 
how close a vectors 
$\mathbf{x}$ comes to
a rational higher-plane, 
in terms of the canonical 
`height' associated
to the plane. Typically, 
dual $\psi$-approximable 
vectors are easier
to understand than simultaneously $\psi$-approximable vectors. We
shall not be concerned with the former (see \cite{B Recent}
for further reading).
\end{remark}
In contrast to this, deciding if either of these basic results holds
for a proper sub-manifold $\mathscr{M}\subset
\mathbb{R}^{n}$ is a
challenging question 
which is under active investigation. 

Using a standard framework 
in Diophantine approximation 
(called `ubiquity')
one usually proves 
Khintchine and Jarn\'{i}k 
type theorems by counting rational
points near manifolds.
However, there is an 
important  difference:
for such a statement one
\emph{is allowed} to remove any 
proportion of the manifolds 
from the counting problem
as long as one can provide a suitable
measure bound for what was removed. 
Hence, counting rational points near the 
\emph{entire} manifold
is a strictly more difficult problem. 
We refer to Beresnevich, Dickinson, and Velani 
\cite{BDV limsup sets}
for how the ubiquity framework works. 

\subsubsection*{Khintchine's theorem for manifolds}

Unlike the classical Khintchine's theorem, where only the divergence
statement requires some effort, both the divergence and the convergence
parts contained in Khintchine's theorem are difficult for manifolds.
Indeed the two theories
are developed to slightly differently stages.
To state the results, we say $\mathscr{M}$ is 
non-degenerate if it
is $l$ non-degenerate 
at every point for 
some uniform $l>0$.
Since counting rational points
near non-degenerate planar curves is sufficiently
well understood, we restrict attention 
to $n\geq3$ in the following.
Further, 
we will not repeat to mention classes in which Conjecture
\ref{conj: Rational Main Conjecture } has been established as in
these cases Khintchine's theorem is known.

\textit{Convergence theory of Khintchine's theorem on manifolds:}
There is a collection of fine results which can be seen as partial
progress towards a coherent Khintchine convergence theory. The resolution
of the Baker--Sprinzuk problem by Kleinbock and Margulis \cite{KM} is
such an instance. However, in a recent breakthrough the convergence
theory of non-degenerate $C^{2}$-manifolds was proved by Beresnevich
and L. Yang \cite{Beresnevich Yang}. This was followed by the work of Schindler and the authors \cite{DRN} which proved better quantitative estimates for smooth manifolds. 
We refer to \cite{BBDV} for the history on the matter. 

\textit{Divergence theory of Khintchine's theorem on manifolds:}
The divergence theory for \emph{analytic}
non-degenerate manifolds was established
by Beresnevich \cite{Beresnevich Rational points near manifolds}. 
We also refer to \cite{BBDV} for further reading. 
Beresnevich \cite[Section 8]{Beresnevich Rational points near manifolds} explicitly asks for removing 
the strong smoothness assumptions from
his results by writing in the final section of that paper: 
\begin{quotation}
`More precisely, the analyticity assumption is used to verify condition
(i) of Theorem KM. A natural challenging question is then to what
extent the analyticity assumption can be relaxed within Theorem 5.2
and consequently within all the main results of this paper.'
\end{quotation}
Further, there is interesting work 
of J.-J. Huang 
in the case that the manifold is a 
sub-space \cite{Huang Subspaces}. 
Here the curvature assumptions to
count rational points take the form 
of assumptions on certain
Diophantine exponents.\\
\\
In terms of smoothness assumptions, 
the theory of Diophantine approximation
on rough manifolds is essentially 
non-existent. The present paper,
as a consequence,
establishes for the first time 
counting results for general 
classes of such hypersurfaces.

\subsubsection*{Diophantine Applications}

Now we turn our attention to Jarn\'{i}k's theorem for manifolds. Here
the situation is quite different. Not all that much is known. For
instance, Jarn\'{i}k's theorem for the moment curve $(x,x^{2},x^{3})$
in $\mathbb{R}^{3}$ is only partially known --- we refer to Beresnevich
and L. Yang \cite{Beresnevich Yang}, and Schindler-Srivastava-Technau \cite{DRN} for further discussion. 

Our main applications here
are the following two theorems. 
Denote by $\cH^s$ the Hausdorff $s$-measure
in $\mathbb{R}^n$.\index{Hausdorff_measure@$\cH^s$:
the Hausdorff $s$-measure in $\bRn$}
\begin{thm}\label{thm: dio app low degree}
Let $n\geq 3$ be an integer.
Let $ \tfrac{2(n-1)}{2n-3} < d\leq 2(n-1)$ 
be a real number. 
For any approximation function
$\psi$, $s> \frac{n-1}{2}$ 
and 
$\sM \in \sG(\sH_d^{\mathbf{0}}(\bRn))$ 
we have 
$$
\cH^s(\sS_n(\psi)\cap \sM)=0 
\quad \mathrm{if}
\quad \sum_{q\geq 1} 
\Big(\frac{\psi(q)}{q}\Big)^{s+1} q^n <\infty.
$$
\end{thm}
In the case of large degree $d$
we have:

\begin{thm}\label{thm: dio app high degree}
Let $n\geq 3$ be an integer, 
$d>2(n-1)$
be a real number, and $\varepsilon>0$.
For any approximation function
$\psi$, $s>\frac{n-1}{2}$
and $\sM \in 
\sG(\sH_d^{\mathbf{0}}(\bRn))$ 
we have 
$$
\cH^s(\sS_n(\psi)\cap \sM)=0 
\,\, \mathrm{if}
\,\,\sum_{q\geq 1} 
\Big(\frac{\psi(q)}{q}\Big)^{s+1} q^n <\infty,
\,\, \mathrm{and} \,\,
\sum_{q\geq 1} 
\Big(\frac{\psi(q)}{q}\Big)^{s+\frac{n-1}{d}
+\varepsilon} 
q^{n-1} <\infty.
$$
\end{thm}
The proof of these results,
once a device like 
Theorem \ref{thm: main}
(or Theorem \ref{thm: main smoothed}) is at hand,
is well--known to experts in Diophantine 
approximation on manifolds.
The appendix, see Section
\ref{app: Dio}, provides the details
for the interested reader.

\subsection{A smooth variant of Serre's Dimension Growth Conjecture}

Let $\mathbb{P}^{n}(\mathbb{Q})$ 
denote the 
$n$-dimensional projective
space over the rationals $\mathbb{Q}$. 
Consider an irreducible variety
$\mathscr{V}\subset
\mathbb{P}^{n}(\mathbb{Q})$ of 
degree $D\geq1$.
For a given point 
$\mathbf{p}=(p_{0},\ldots,p_{n})
\in\mathbb{P}^{n}(\mathbb{Q})$,
take
$\mathbf{p}\in\mathbb{Z}^{n+1}$
and so that 
$\mathrm{GCD}(p_{0},\ldots,p_{n})=1$,
define its height via\index{naiveheigt@$H$: the naive height on projective $n$-space}
\[
H(\mathbf{p}):=\max_{0\leq i\leq d}\vert p_{i}\vert.
\]
A fundamental problem in number theory is to count 
the number of points in 
$\mathscr{V}$ with bounded height,
as the height bound tends to infinity. 
More formally, one is interested in
the asymptotics of
\[
N_{\mathscr{V}}(B):=\#\{\mathbf{p}\in\mathscr{V}:\,H(\mathbf{p})\leq B\},
\]
which is the subject
of a deep conjecture of Manin 
and his collaborators,
see \cite{Franke Manin Tschinkel}
and \cite{Batyrev Manin}. 
In this context,
Serre \cite{Serre Topics in Galois theory}
proposed the following upper-bound:
\begin{conj}[Serre's Dimension Growth Conjecture]
If $\mathscr{V}\subset\mathbb{P}^{n}(\mathbb{Q})$
is an irreducible variety of degree at least two, 
then $N_{\mathscr{V}}(B)\ll_{\mathscr{V}}
B^{\dim\mathscr{V}}(\log B)^{\mathfrak{z}}$
for some constant $\mathfrak{z}
=\mathfrak{z}(\mathscr{V})>0$.
\end{conj}
For a historical account of the (now essentially solved)
Dimension Growth Conjecture, 
we refer the reader to the 
book of Browning 
\cite{Browning Quantitative arithmetic geometry}.
Indeed, substantial progress 
was achieved by the work of Browning,
of Heath-Brown, and
of Salberger.
An important technique 
that grew out of this body of work 
is Heath-Brown's
version \cite{Heath Brown} 
of the determinant method 
of Bombieri 
and Pila \cite{Bombieri Pila}.

A natural generalisation 
of counting 
rational points 
on projective varieties,
is counting rational points
on (smooth) projective manifolds. 
In this spirit, we consider the
following problem: 
\begin{problem}[Smooth variant of Serre's Dimension Growth Conjecture]
\label{prob: smooth Serre}
Suppose $\mathscr{V}\subset\mathbb{P}^{d}(\mathbb{Q})$
is a smooth manifold. 
Under which conditions on $\mathscr{V}$ is
it true that $N_{\mathscr{V}}(B)\ll_{\mathscr{V},
\varepsilon}B^{\dim\mathscr{V}+\varepsilon}$
holds each $\varepsilon>0$?
\end{problem}

In \cite{Huang rational points}, 
Problem \ref{prob: smooth Serre}
was solved
for a rich class 
of smooth manifolds
defined by a mild 
curvature condition. 
This curvature condition
requires that there 
exists for every 
$\mathbf{x}_{0}\in[-1,1]^{m}$ 
a unit vector 
$\mathbf{u}\in\mathbb{R}^{c}$ 
such that 
\[
\det
\Big(\frac{\partial^{2}
\left\langle \mathbf{u},\mathbf{f}
\right\rangle (\mathbf{x}_{0})}
{\partial x_{i}\partial x_{j}}
\Big)_{
1\leq i,j\leq m}
\neq 0.
\]
For further discussion, see
Schindler and Yamagishi 
\cite{Schindler Yamagishi} 
and Munkelt \cite{Munkelt}.
We use the following convention,
which is the analog of the above condition
for hypersurfaces from
the class $\sG( \sH_d^{\mathbf{0}}(\bRn))$.
\begin{definition}\label{def: class of addmissible manifolds}
Denote by $\sP_c(\sH_d^{\mathbf{0}}(\bRn))$ 
the set of $\sM \in \sG( \sH_d^{\mathbf{0}}(\bRn))$
of co-dimension 
$\rc\geq 1$
which exhibit the subsequent property.
For every small open ball 
$\sB(\bx_0,r)$ with centre $\bx_0$ and radius $r>0$
and any local chart $(\bx,\bf(\bx))$ of $\sM$, 
where  
$\bx \in \sB(\bx_0,r)$,
there exists $\bs \in \mathbb{Z}^{\rc}$
and $Q_0\in \bZ_{\geq 1}$
satisfying
$$
\bx \mapsto \langle \bf(\bx + \bx_0), 
\frac{\bs}{Q_0} \rangle \in \sH_d^{\mathbf{0}}(\bRn).
$$
\end{definition}
We can resolve Problem \ref{prob: smooth Serre}
for the class $\sP_c(\sH_d^{\mathbf{0}}(\bRn))$
as soon as $c>1$. The case $c=1$
has a rather different character, 
we elaborate more on this
in Section \ref{sec: final remarks}.
\begin{thm}\label{thm: dimension growth app}
Let $n\geq 3$, and $c\in [2,n-1]$ be integers.
If $\sM \in \sP_c(\sH_d^{\mathbf{0}}(\bRn))$ 
and $ \tfrac{2(n-1)}{2n-3} < d\leq 2(n-1)$, then 
$$
N_{\sM}(B) \ll_{\sM,\varepsilon} 
B^{\dim \sM + \varepsilon}
$$
for any fixed $\varepsilon>0$.
\end{thm}
\begin{remark}
Suppose $\sM$ is as in Theorem 
\ref{thm: dimension growth app}.
Assume further that $\sM$ is a non-singular 
projective variety (defined
by a homogeneous polynomial).
Then Theorem \ref{thm: dimension growth app}
can be used to count integral points 
on $\sM$ which are located in
dilations $t\sR$ for any fixed region $\sR$
when $t\rightarrow \infty$,
compare 
\cite[Remark 7]{Huang rational points}. 
\end{remark}

\section{Outline of the Proof and Novelties}
\label{sec: outline}
Let $f\in \sH_d^{\mathbf{0}}(\bRn)$. Our analysis of the counting 
function in \eqref{eq; counting function hypersurfaces}
involves eight steps. 
The first step of introducing smooth 
cut-offs is rather standard. However, steps two through five introduce multiple ideas from harmonic analysis which are new in the context of counting rational points.
Steps two and three involve 
Knapp cap type considerations. These are inspired from examples 
often used in Fourier analysis 
to analyse the effect of curvature 
(or lack thereof) on estimates 
for restriction operators associated 
to frequencies arising from 
submanifolds; and for Radon-type operators associated to averages over such manifolds. 
The idea of rescaling in step four 
allows us to `zoom in' around the region where the curvature
vanishes (or blows up), 
by exploiting the homogeneity of the 
parametrizing function $f$. 
The next step is to employ a more 
elaborate version of stationary phase expansion involving higher order terms.
Steps six and seven build on the 
iterative method of 
\cite{Huang rational points} 
but involve several additional 
ideas to adapt it to the the present context. 
A key difference is that while in 
\cite{Huang rational points}, 
the class of surfaces is essentially 
preserved under duality, 
we need to simultaneously 
count rational points near 
two different kinds of hypersurfaces 
(the `flat' case and the `rough' case) 
for the induction to work. We also need to run the bootstrapping argument for each dyadic scale individually.

\subsubsection*{Step 1: Localisation and Smoothing} 
Since the curvature of $\sM$ vanishes (or blows up) only at the origin in $\bR^{n-1}$, we first localise to a small region around it. To this effect, for a small enough $\upsilon>0$ (to be specified later), let $\sB(\mathbf{0}, \upsilon)$ 
be the (Euclidean) ball of radius $\upsilon$ centred at $\bzero\in \bRn$. 
While $\nabla f$ 
might not be invertible on $\sB(\mathbf{0}, \upsilon)$, 
our assumption on the Hessian of $f$, 
combined with the homogeneity of $f$ 
implies that $\sB(\mathbf{0}, \upsilon)\setminus\{\bzero\}$ can be covered 
by a union of a finite number of 
domains (depending on $f$) such that 

1. Each domain can be expressed as the union of dyadic dilations of a sufficiently small ball contained in the annulus $\{\bx: \tfrac{\upsilon}{2}\leq \|\bx\|_2\leq \upsilon\}$ (see \S\ref{subsec local} for a detailed description).

2. $\nabla f$ is invertible 
when restricted to each of the aforementioned balls bounded away from the origin, and thus on each of the domains generated by them (using homogeneity).

We fix one such domain $\mathscr{D}_1$, generated by the dyadic dilations of a ball $\mathscr{B}$ of sufficiently small radius contained in the annulus $\{\bx: \upsilon/2\leq \|\bx\|_2\leq \upsilon \}$. Fixing a smooth bump function $\varrho$ supported in the said ball, we approximate the indicator function of $\mathscr{D}_1$
by the dyadic sum \index{varrho@$\varrho$: smooth and compactly supported function used to construct a 
suitable (dyadic) partition of unity of a conical region and designed so that the inveribility of 
$\nabla f: \bRn \rightarrow \bRn$
holds in a slight enlargement of
$\supp(\varrho)$ }
$\sum_{\ell\geq 1}\varrho (2^\ell\cdot)$.\index{ell@$\ell$: the dyadic power $2^{\ell}$
is used to partition $\bRn$ into the annuli 
$ \sA(2^{-(\ell+1)},2^{-\ell})$ 
} Further, let $b,\om: \bR \rightarrow [0,1]$\index{bump@$\omega$: smooth and compactly supported function used to
track $q\in [Q,2Q)$}
\index{bump2@$b$: smooth and compactly supported function used to detect $\Vert q f(\ba/q)\Vert \leq \delta$}
be smooth and compactly supported functions approximating the indicator functions of the intervals $[1,2)$ and $[0,1/2]$ in
\eqref{eq; counting function hypersurfaces}. 
Our focus will be on obtaining estimates for
the smoothed counting function\index{smooth_count@$N_{f}^{\om,\varrho,b}(\delta,Q)$: smooth count of rational points 
near the graph of $f\in \sH_{d}^{d}(\bRn)$}
\begin{equation}\label{def: counting function}
N_{f}^{\om,\varrho,b}(\delta,Q)
:= \sum_{(q,\ba) \in \mathbb{Z}
\times \bZn}
\om\Big(\frac{q}{Q}\Big)
\sum_{\ell\geq 0}\varrho \Big(2^{\ell}
\frac{\mathbf{a}}{q}\Big)
b\Big(\frac{\Vert q 
f(\mathbf{a}/q)\Vert}{\delta}\Big).
\end{equation}
To switch seamlessly between smooth and sharp-cut
off functions, we tacitly require henceforth that
\begin{equation}\label{def: om partion of unity}
\ind_{[1,2^R]} (\vert x\vert )
\leq 
\sum_{0\leq r\leq R} \om\Big(\frac{x}{2^r}\Big)
\leq 
\ind_{[1/10,2^{R+1}]} (\vert x\vert)
\end{equation}
for all non-zero $x\in \bR$
and $R\geq 1$.
Such functions $\om$ are well-known, see Lemma \ref{lem: partition of unity}
in the appendix.

It is a routine task to deduce
Theorem \ref{thm: main} 
from the following estimate 
on the smooth counting function 
\eqref{def: counting function} (see Section \ref{app: smoothing} of the appendix for details).
We have
\begin{equation}
\label{eq: main smooth}
\Big(\delta +
\Big(\frac{\delta}{Q}\Big)^{
\frac{n-1}{d}}\Big)Q^{n}
\ll
N_{f}^{\om,\varrho,b}(\delta,Q) \ll
\delta Q^{n} +
\Big(\frac{\delta}{Q}\Big)^{
\frac{n-1}{d}}
Q^{n+k\varepsilon}
\,\, \mathrm{for}\,\, \delta \in 
(Q^{\varepsilon-1}, 1/2)  
\end{equation}
where $k>0$ is a constant only
depending on the dimension $n$.
The remaining implied constants 
depend exclusively
on $\varepsilon,f,n,d,\om,\varrho,b$ 
but not on $\delta$, or $Q$.
A more precise statement shall follow in Theorem \ref{thm: main smoothed}. For now, we suppress the dependence on the parameters $\om, \varrho$ and $b$.

The proof of the above estimate is
based on harmonic analysis
and driven 
by a curvature-informed decomposition.
Let
$$
\mathscr{D}_{d-1}:=\nabla f(\mathscr{D}_1),
$$ 
and let $$\widetilde{f}: \nabla f(\mathscr{D}_1)\to \mathscr{D}_1$$\index{dual_f@$\widetilde{f}$: the (Legendre) dual of $f$}
be the Legendre transform of $f$ (see \S\ref{subsec hom func duality} for details). It is not very hard to see that $\widetilde{f}$ is homogeneous of degree\index{Hoelderdual@$d'$: the H\"{o}lder dual of $d>1$ which is the degree of homogeneity of $\widetilde{f}$}
\begin{equation*}
    d':=(1-d^{-1})^{-1}=\frac{d}{d-1}.
\end{equation*}
In fact, we have
$$
\widetilde{f} \in 
\mathscr{H}_{d'}^{\mathbf{0}}(\mathscr{D}_{d-1}).
$$
Since $d'$ is the H\"{o}lder conjugate of $d$, exactly one of the relations
$$1<d'\leq 2\leq d<\infty$$ or $$1<d<2<d'<\infty$$ holds true. Without loss of generality, let us assume that
$$d\geq 2$$
and $d'\in (1,2]$.
As will be clear shortly, the main induction argument rests on a delicate connection between the rational point count for the
hypersurface immersed by $f$ and that for the (Legendre) dual
hypersurface immersed by $\widetilde{f}$.

The steps 3 through 7
need to be executed for both $f$
as well as its dual $\widetilde{f}$,
with various approximate replacements
of the weight functions $\varrho$.
(Here, and throughout, by a `weight function'
we mean a smooth, non-negative,
and compactly supported function.)
However, for ease of exposition,
we confine attention in the present
outline to the case
$f\in \mathscr{H}_{d}^{\mathbf{0}}(\bRn)$
with $d\geq 2$ (the `flat case') and 
keep $\varrho$ artificially fixed.
Roughly speaking
we (smoothly) decompose the hypersurface 
into a part which 
is ``too flat'' 
for harmonic analysis to be profitably used, 
and a part which is ``appropriately curved''.
The next step deals with 
the former.

\subsubsection*{Step 2: Pruning out Spatial Knapp Caps}
By our assumption on $f$,
the curvature of
the hypersurface immersed by 
$f$ vanishes only at the 
origin $\bzero \in \bRn$. 
In what neighborhood of $\bzero$ 
is it approximately flat at
the scale $\delta/Q$ 
within we are to locate the
rational points?
The uncertainty principle 
provides the following heuristic answer: it is only possible to detect curvature outside the box

\begin{equation*}
\sC(\delta,Q):=
\Big[-\Big(
\frac{\delta}{Q}\Big)^\frac{1}{d},
\Big(\frac{\delta}{Q}
\Big)^\frac{1}{d}\Big]^{n-1}
\end{equation*}
centred at the origin.
Indeed, for any $\bx\in \sC(\delta,Q)$, we have
$\Vert f(\bx)\Vert 
\leq \vert f(\bx)\vert \ll \delta/q$.
Thus the Knapp
cap given by $\{(\bx,f(\bx)): \bx\in \sC(\delta,Q)\}$
is fully contained in the box 
$ \sC(\delta,Q)\times [-\frac{\delta}{Q},
\frac{\delta}{Q}]$. We shall remove 
a tiny enlargement 
of the box $\sC(\delta,Q)$
to avoid the obstruction due to local flatness (the so-called Knapp obstruction). 
How many rationals 
do we find in this set? 
Thanks to the geometry 
of numbers,
see Section \ref{sec: lower bounds},
is not difficult to see 
that for each $q\asymp Q$,
the contribution is of size
$$ 
\#(\sC(\delta,Q) \cap \frac{1}{q}\bZn)
\asymp 
\frac{\mathrm{vol}
(\sC(\delta,Q))}{q^{-(n-1)}}
\asymp \Big(\frac{\delta}{Q}
\Big)^{\frac{n-1}{d}}Q^{n-1}
$$
provided the re-normalised volume 
$\sC(\delta,Q)/q^{-(n-1)}$ 
tends to infinity with $Q$, 
which essentially means that
$\delta \geq 
Q^{1-d+\varepsilon}$.
Summing up the 
re-normalised volumes over $q\asymp Q$
gives one of the two main terms, namely
$(\frac{\delta}{Q})^{\frac{n-1}{d}}Q^{n}$. This term accounts for the contribution to $N_f(\delta, Q)$ (as defined in \eqref{def: counting function}) coming from terms corresponding to sufficiently large values of $\ell$ (in terms of $\delta$ and $Q$).
The problem is now reduced to dealing with terms of the form\index{counting@$N_f(\delta,Q,\ell )$: 
smooth, spatially localised count}
$$
N_f(\delta,Q,\ell )
:=\sum_{(q,\ba) \in \mathbb{Z}
\times \bZn}
\om\Big(\frac{q}{Q}\Big)
\varrho \Big(2^{\ell}
\frac{\mathbf{a}}{q}\Big)
b\Big(\frac{\Vert q 
f(\mathbf{a}/q)\Vert}{\delta}\Big)
$$
when $\ell$ is not too large 
in terms of $\delta$ and $Q$.
\subsubsection*{Step 3: Knapp Caps 
in Frequency Domain}
Let $\widehat{b}$ 
be the Fourier transform of $b$.
By Poisson summation 
and partial integration, 
(see the proof of 
Lemma \ref{lem: truncated Poisson}), 
the function
$N_{f}(\delta,Q,\ell)$ 
is equal to the exponential sum
\begin{equation}\label{eq: heuristic expon sum shape}
\sum_{(q,\mathbf{a})\in \bZ\times \mathbb{Z}^{n-1}}
\om(\frac{q}{Q})
\varrho(2^\ell\frac{\mathbf{a}}{q}) 
\sum_{0\leq \vert j\vert \leq 
\frac{Q^\varepsilon}{\delta}} 
\delta
\widehat{b}(j\delta)
e(jqf(\mathbf{a}/q))
\end{equation}
aside from a tiny error.
The zero mode, $j=0$, 
gives rise to the 
``probabilistic main term''
of size $\delta Q^n$.

It might seem that since 
we already extracted the 
more subtle ``geometric main term''
$(\frac{\delta}{Q})^{\frac{n-1}{d}}Q^{n}$
in the previous step, 
all that is left to show is that the remaining terms
only contribute a negligible error term.
However, this initial impression 
turns out to be incorrect, for there exists
a `twin' of the Knapp cap
in Fourier space which needs to be removed 
as well. To this end,
we introduce another partition 
of unity in $j$
by summing up 
dyadic powers $\om(x/2^r)$
over $r\geq 0$ and
then remove contributions from small $j$ 
by another volume based argument. 
Only now can we actually hope to 
exploit significant cancellation 
in the exponential sums.
Via this removal process,
the term 
$(\frac{\delta}{Q})^{\frac{n-1}{d}}Q^{n}$
emerges once again.

At this stage, we have reduced matters 
to bounding\index{counting@$N_f(\delta,Q,\ell,r)$:
smooth count localised in physical and frequency space}
$$
N_f(\delta,Q,\ell,r):=
\sum_{(q,\mathbf{a})\in \bZ\times \mathbb{Z}^{n-1}}
\om(\frac{q}{Q})
\varrho(2^\ell\frac{\mathbf{a}}{q}) 
\sum_{ j \in \bZ}
\om(\frac{j}{2^r})
\delta
\widehat{b}(j\delta)
e(jqf(\frac{\mathbf{a}}{q}))
$$
when $\ell \geq 0$ is not too large
and $r$ is positive and not too small.

\subsubsection*{Step 4: Re-scaling}
For now we fix $q,j$.
By Poisson summation, the 
sum in $\ba$ in
\eqref{eq: heuristic expon sum shape} is converted into a sum of oscillatory integrals of the form
$$
\sum_{\bk\in\bZn}
\int_{\bRn}
\varrho\Big(\frac{2^\ell}{q}\bx\Big) 
e(jqf(\mathbf{x}/q) 
- \langle \bx ,\bk \rangle )
\rd \bx
$$
By a change of variables,
the inner integral above equals
$$
\Big(\frac{q}{2^{\ell}}\Big)^{n-1}
\int_{\bRn}
\varrho(\bx) 
e(qj[f(2^{-\ell} \mathbf{x}) 
- \langle 2^{-\ell}\bx ,j^{-1}\bk \rangle] )
\rd \bx.
$$
As $f$ is homogeneous of
degree $d$, we can rewrite the phase as
$$ qj[f(2^{-\ell} \mathbf{x}) 
- \langle 2^{-\ell}\bx ,j^{-1}\bk \rangle]=qj[2^{-\ell d} f(\mathbf{x}) 
- \langle 2^{-\ell}\bx ,j^{-1}\bk \rangle]=qj 2^{-\ell d}[ f(\mathbf{x}) 
- \langle 2^{\ell(d-\1)}\bx ,j^{-1}\bk \rangle].$$
Thus we have reduced matters to oscillatory integrals of the form
\begin{equation*}
I(\lambda,\varrho,\phi):=\int_{\mathbb{R}^{n-1}}\,\varrho(\mathbf{x})\,e(\lambda\phi(\mathbf{x}))\,\mathrm{d}\mathbf{x}
\end{equation*}
where 
$$\lambda:=qj2^{-\ell d},\, \qquad \phi(\bx):= 
\phi(\bx,\bk,j,\ell):=
f(\mathbf{x}) 
- \langle \bx ,
2^{(\d1)\ell}j^{-1}\bk \rangle,$$
and the amplitude $\varrho$ is supported in a bounded set \textit{away} from the origin.
The re-scaling
and the previous pruning step ensure 
that we can access the theory
of oscillatory integrals since:\\
1.) the oscillatory parameter
$\lambda$
is at least a little 
large (in our case that means 
$\lambda:=jq2^{-\ell d} 
\geq Q^\varepsilon$), 
and \\
2.) there is 
uniformly bounded curvature,
i.e. the determinant of 
$H_\phi$ is uniformly bounded from above and below, and the phase function $\phi$ and its derivatives are uniformly bounded from above, in the support of $\varrho$.\\
Indeed, either
such an integral
$I(\lambda,\varrho,\phi)$
has a negligible contribution
(by the non-stationary phase principle)
or can be computed
(by the stationary 
phase principle) with a substantial decay in $\lambda$. In any case,
we get information 
that can be summed
over both $\bk$ and $j$.

\subsubsection*{Step 5: Multi-term stationary phase expansion and Enveloping}
A significant barrier that we need to work around
is that the 
oscillatory parameter $\lambda$ 
is allowed to be of size 
as small as $Q^\varepsilon$.
(For comparison, in \cite{Huang rational points}, 
it was essentially always true that $\lambda \geq Q$.)
As a result, in the worst case, we obtain very small decay and hence the size of the error term arising out of the 
standard one term stationary phase expansion (employed to great success in \cite{Huang rational points})
is still too large. To obtain an acceptable error term,
we employ a multi-term stationary phase expansion. 
As expected, the latter is significantly more technical. Further, since such an  expansion involves a family of differential operators acting on the weight $\varrho$, the resulting expression is no longer positive.
More precisely, up to a small error, $I(\lambda,\varrho,\phi)$
evaluates to a sum of
complex valued functions
$w_1,\ldots,w_t$,
arising from certain
complex-valued differential 
operators applied to $\varrho$,
times (in our case) the oscillatory
factor $e(qj\widetilde{f}(\bk/j))$.

As it happens, the positivity of the weight functions
is crucial for making certain enlargement
arguments, pertaining 
to sets of critical points 
being either empty 
or only having bounded many elements.
To restore 
the said positivity, we engineer
a technical device 
we call enveloping. This is the method of majorizing 
the absolute value of a smooth complex 
valued function by a positive smooth function on a slight
enlargement of its support.
It is important for us 
to have a good control over a finite number of multi-index
derivatives of the envelope 
in order to successfully execute 
Step 6. In principle, after each inductive step, we need to replace
the implicit weight function in the counting function \eqref{def: counting function}
by a new envelope and keep track of the sup-norm 
of each envelope in terms of the powers of the factor $Q^{-\varepsilon}$. For simplicity of the exposition
we ignore this point in what follows, postponing a more detailed and careful description to \S\ref{sec: envelope}.

\subsubsection*{Step 6: A Duality Principle}
In this step, by means of the stationary phase expansion,
we connect the rational point count
$N_f(\delta,Q,\ell,r)$ associated to a dyadic piece of the given hypersurface to the weighted 
count $N_{\widetilde{f}}(Q^{-1},
\delta^{-1},\ell,r)$ associated with the corresponding dyadic piece of the dual hypersurface.
Notice that the roles of $\delta$ and $Q$
have now been interchanged!
This kind of observation
was explicitly made
by Vaughan  and Velani 
\cite[p.105]{Vaughan Velani},
and also underpins
\cite{Huang rational points}. One may
think of this as a strong and special form 
of the uncertainty principle.
Unlike applying partial summation, as
for instance Vaughan and Velani did, 
we shall use 
Poisson summation over $q$,
capitalising on
the presence of 
our smooth weight function $\om(q/Q)$. 
Thus we save a logarithmic factor over 
the partial summation argument
and more importantly,
we can readily discard any error terms
(which would be more difficult 
without the smoothing).
In a nutshell,
the $q$-summation, 
after taking absolute values,
is roughly just
$$
Q \cdot \ind_{[0,2^iQ^{-1}]}(\Vert 
j\widetilde{f}(\bk/j)\Vert)
$$
summed over those $i$ which are small 
in terms of $Q$.
We estimate this quantity in terms of
$N_{\widetilde{f}}(Q^{-1},
\delta^{-1},\ell,r)$
by summing over $r$.
The rapid decay of 
$\widehat{b}(\delta j)$
ensures that the frequencies 
$j\asymp \delta^{-1}$ contribute
the lion's share. Similarly,
$i=0$ contributes the major part
because of the rapid decay of 
$\widehat \om$.
So, heuristically one may take
$i=0$ and $j = \delta^{-1}$.

A non-trivial complication
arises because of contributions to
$N_{\widetilde{f}}(Q^{-1},
\delta^{-1},\ell,r)$ arising out of local Knapp-cap obstructions in certain ranges of the dyadic 
parameters $(r,\ell,i)$.
In order to handle 
these contributions, 
we need to use the geometry of the relevant
regions to relate $Q$ and $\delta$
with $(r,\ell,i)$.
It takes a few pages of estimations before
these ranges can be handled.

\subsubsection*{Step 7: Bootstrapping at Several Dyadic Scales:}
After deriving a useful 
relation between $N_f(\delta,Q,\ell,r)$ 
and $N_{\widetilde{f}}(Q^{-1},
\delta^{-1},\ell,r)$ in the last step,
we start to develop a  
bootstrapping procedure, inspired by the
seminal work \cite{Huang rational points}. 
One of the novelties of our argument is a delicate interplay between rational point counts for dual manifolds \textit{of different nature}. For 
instance, in \cite{Huang rational points}, the functions $f$ and $\widetilde{f}$
are both in the same class. This happens in our set-up only when $d=2$. For all other values of $d$, we need to keep alternating between estimating rational points for the flat manifold and its rough dual throughout our induction, since these estimates feed on each other. 

Moreover, the presence of the additional geometric term (arising from the Knapp-cap) makes the induction on scales of $\delta$ more intricate. Further, the entire bootstrapping procedure needs to run at several spatial dyadic scales (corresponding to different values of $\ell$) in parallel. Indeed, the range of $\ell$ we consider has a delicate dependence on both $\delta$ and $Q$. Consequently, any induction on scales type argument involving $\delta$ has an effect on the range of $\ell$ for which the induction hypothesis remains valid. For the spatial scales where we cannot use the hypothesis, we can do no better than a trivial volume estimate which then needs to be handled efficiently. 

Finally, our arguments here are sensitive 
to the relative size of $d$ (the degree of homogeneity) and $n-1$ (the dimension of the hypersurface). In particular, for $d\geq 2(n-1)$, our bootstrapping argument cannot improve the error term
beyond the geometric main term evaluated at the scale 
$\delta \approx Q^{-1}$.
The reason is that
there are genuinely that many rational 
points close to the hypersurface in this regime,
as can be shown by a separate argument;
see Section \ref{sec: lower bounds}.

\subsubsection*{Step 8: Obtaining Upper Bounds}
Once we obtain sharp bounds on 
each dyadic piece $N_{f}(\delta,
Q,\ell)$ (and $N_{\widetilde{f}}(\delta,
Q,\ell)$), all that remains is to sum up the terms
over $\ell$ and add the various bounds (such as those obtained from pruning) together. 
This is rather straightforward.

\subsection{Structure of the Manuscript}
In the following, the symbol $\S$
abbreviates the word section/subsection.
\begin{itemize}
\item Section \ref{sec: preliminaries} contains a detailed treatment of Step 1 (\S\S\ref{subsec local}-\ref{subsec smooth}).  This is followed by the statement of Theorem \ref{thm: main smoothed} which is a smoothed and more general version of Theorem \ref{thm: main}. This section also contains facts about oscillatory integrals (\S\ref{subsec: oscillatory}) and geometric summation (\S \ref{subsec: geometric sum}) which play a crucial role in the arguments to follow.
\item Section \ref{sec: pruning} contains the pruning Steps 2 and 3: in spatial (\S \ref{subsec: prun phys}) and frequency (\S \ref{subsec: prun freq}) domains.
\item Section \ref{sec: counting rational points to oscillatory} expands on Step 4, 
reducing the counting problem to obtaining estimates on
certain oscillatory integrals. The latter are handled by dividing the analysis into three regimes: non-stationary phase analysis, an intermediate regime and the (most critical) stationary phase regime. The last two regimes employ a multi-term stationary phase asymptotic expansion.
\item Section \ref{sec: envelope} carries out Steps 5 and 6. The device of enveloping for weights is introduced in \S\ref{subsec: envelope}. This is combined with the stationary phase analysis (from before) to deduce a duality principle in \S\ref{subsec: duality} (connecting the oscillatory integrals to counting rational 
points for a dual hypersurface). \S\ref{subsec strange terms} contains a few technical estimates to handle terms arising out of local Knapp-cap obstructions in certain ranges of the dyadic parameters.
\item Section \ref{sec: bootstrapping} carries out Step 7 to establish a bootstrapping argument utilising the duality between the flat hypersurface and its rough dual (or vice versa). The cases when $d\leq n-1$ and $d\geq n-1$ need to be handled separately, in \S\ref{subsec: d small} and \S\ref{subsec: d large} respectively.
\item Section \ref{sec: upper bounds} utilises the arguments above to carry out the induction and obtain the upper bounds in Theorem \ref{thm: main smoothed} (Step 8).
\item Section \ref{sec: lower bounds} proves the lower bounds in Theorem \ref{thm: main smoothed} by using lattice point counting techniques.
\item Section \ref{sec: final remarks} contains some closing remarks, discussion of related open problems and acknowledgement.
\item The appendices contain a brief discussion about standard arguments used throughout the manuscript: from harmonic analysis (\S\ref{app: smoothing}), Diophantine approximation (\S\ref{app: Dio}) and arithmetic geometry (\S\ref{app: dimension growth}).
\end{itemize}

\subsection{Notation}
Generally, a bold-faced letter 
denotes a vector and sets 
(aside from specific standard sets like $\mathbb{R}$) are written in a script-font. As explained before, 
$\ind$ denotes the indicator function of 
the set specified in its subscript.

Bachmann--Landau notation is used
in the usual way. 
For functions $A, B: \sX\to\mathbb{C}$
with $\sX \subseteq \mathbb{R}^{n-1}$, we shall use the Vinogradov notation $A\ll B$ to mean that $|A(\bx)|\leq C|B(\bx)|$ 
is true all $\bx \in \sX$ for some positive constant $C>1$ which might depend on $n, d$, $\varepsilon$ and $f$  from Theorem \ref{thm: main}. We shall write $A\asymp B$ to denote that $A\ll B$ and $B\ll A$.
(In harmonic analysis, the symbols
$\lesssim$ and $\gtrsim$ are used 
in place of the Vinogradov symbols.)
Any other dependency 
will be indicated 
by an appropriate subscript. 

Given $\mathscr{X}\subseteq\mathbb{R}^{n-1}$
and $\mathscr{Y}\subseteq\mathbb{R}^{m}$, 
we denote by $$ 
\mathscr{C}_{c}^{\infty}
(\mathscr{X},\mathscr{Y})
$$
the collection of 
all smooth functions 
$s:\mathscr{X}\rightarrow\mathscr{Y}$
whose support is compact.  
For a smooth function 
$g:\mathscr{X}\rightarrow\mathbb{R}$,
with $\mathscr{X}\subseteq\mathbb{R}^{n-1}$, 
we define the canonical
$C^{t}(\mathscr{X})$-norm by
\begin{equation}\label{def: C^k norm}
\Vert g\Vert_{C^{t}(\mathscr{X})}:=\max_{0\leq\vert\boldsymbol{\alpha}\vert\leq t}\sup_{\mathbf{x}\in\mathscr{X}}\vert g^{(\boldsymbol{\alpha})}(\mathbf{x})\vert
\end{equation}
where the derivative $g^{(\boldsymbol{\alpha})}$ is with respect to
the multi-index derivative $\boldsymbol{\alpha}:=(\alpha_{1},\ldots,\alpha_{m})\in\mathbb{Z}_{\geq0}^{m}$
of length 
\[
\vert\boldsymbol{\alpha}\vert:=\sum_{1\leq i\leq m}\alpha_{i}.
\]
We follow the convention that 
$g^{(\bzero)}=g$.
When $\sX = \bR^{n-1}$ we simplify notation
by letting
$$ \Vert g\Vert_{C^{t}}
:= \Vert g\Vert_{C^{t}(\mathscr{X})}.
$$

The Euclidean norm 
$$\Vert(x_{1},\ldots,x_{m})
\Vert_{2}:=(\vert x_{1}\vert^{2}
+\ldots+\vert 
x_{m}\vert^{2})^{1/2}
$$
should not be confused 
with the distance to 
the nearest integer 
$\Vert\cdot\Vert
=\mathrm{dist}(\cdot,\bZ)$.

An open (Euclidean) ball 
in $\bRn$ with radius $r>0$
and centre $\bx$ will be denoted by $\sB(\bx,r)$. 
The unit ball $\sB(\bzero,1)$ shall be denoted by $\sU$. For $0<r<R$, we define
\index{annulus@$\sA(r,R)$: the open (Euclidean) annulus
with inner radius $r$ and outer radius $R$}
$$
\sA(r,R):=\{\by\in \bRn: 
r<\Vert \by \Vert_2 < R \}
$$
to be the open (Euclidean) annulus
in $\bRn$ centred at the origin, with inner radius $r$ and outer radius $R$. Further, the closures of these sets shall be denoted by $\overline{\sB}(\bx,r), \overline{\sU}$ and $\overline{\sA}(r,R)$ respectively.

For a set 
$\mathscr{S}\subseteq\mathbb{R}^{n}$, $\partial\mathscr{S}$ 
shall denote its (Euclidean) boundary. 
For a compact set 
$\mathscr{S}\subseteq\mathbb{R}^{n-1}$ and a point $\mathbf{x}\in \mathbb{R}^{n-1}$, we define
\[
\mathrm{dist}(\mathbf{x},\mathscr{S}):=\min\{\left\Vert \mathbf{x}-\mathbf{s}\right\Vert _{2}:\mathbf{s}\in\mathscr{S}\}
\]
to be the (Euclidean) distance of 
$\mathbf{x}$ 
from $\mathscr{S}$.

Given a smooth function $w$ on $\bR^{n-1}$ and $\ell\in\mathbb{Z}_{\geq 0}$, we define\index{localisation@$w_{\ell}(\cdot)=  w(2^{\ell}\cdot)$: localisation 
of a smooth function $w$ to the scale $2^{\ell}$}
\[
w_{\ell}(\mathbf{x}):= w(2^{\ell}\mathbf{x}).
\]

\section{Preliminaries}\label{sec: preliminaries}
This section gathers 
reduction steps and facts
which are sometimes considered standard 
in the higher echelon. 
However, it stands to reason
that not every such step and fact 
is trivial for a general reader.
Therefore it is worthwhile 
to lay them out in the present section. 
\subsection{Localisation}
\label{subsec local}
Let
$f\in\sH_d^{\mathbf{0}}(\bRn)$ with $d>1$. 
As the determinant 
of $H_f$ vanishes or blows up only at the origin, there exists a $\upsilon>0$ and
constants $c_f, C_f>0$ such that
\begin{equation*}
    \label{eq pre Hess cond}
    2{c_f}<|\det\, H_{f}(\mathbf{x})|<\frac{C_f}{2}
\end{equation*}
for all $\mathbf{x}$ 
contained in the closure of the annulus 
$\sA(\frac{\upsilon}{2},\upsilon)$.
By the Inverse Function Theorem, 
we conclude that for every such $\bx$, 
there exists a ball 
$\sB(\bx,\epsilon_\bx) 
\subset \sA(\frac{\upsilon}{3},
\frac{4}{3}\upsilon)$ 
such that $$
\nabla f: \sB(\bx,\epsilon_\bx) 
\to \nabla f(\sB(\bx,\epsilon_\bx))
$$ is a diffeomorphism, and
\begin{equation}
    \label{eq Hess cond}
    c_f<|\det\, H_{f}(\mathbf{z})|<C_f
\end{equation}
for all $\mathbf{z}\in \sB(\bx,\epsilon_\bx)$. Naturally $\epsilon_{\bx} < \upsilon$. 
The collection 
$$\{\mathscr{B}(\bx, \tfrac{\epsilon_{\bx}}8): \tfrac{\upsilon}{2}\leq \|\bx\|_2\leq \upsilon\}$$
forms an open cover of $\overline{\mathscr{A}}(\tfrac{\upsilon}{2}, \upsilon)$. Using compactness, we can extract a finite subcollection of balls, associated to points $\bx_1, \ldots, \bx_I$ such that 
$${\mathscr{A}(\tfrac{\upsilon}{2}, \upsilon)}\subseteq \bigcup_{i=1}^I\mathscr{B}(\bx_i, \tfrac{\epsilon_{\bx_i}}8).$$
By scaling, we also have
$$
\sB(\bzero, \upsilon)\setminus\{{\bzero}\}
\subseteq \bigcup_{i=1}^I\bigcup_{\ell \geq 0}
2^{-\ell} \mathscr{B}(\bx_i, 
\tfrac{\epsilon_{\bx_i}}{8}).
$$
The number $I$ above depends only on $f$. Hence, without loss of generality, we shall fix one such ball, say $\mathscr{B}(\bx_1, \tfrac{\epsilon_{\bx_1}}8)$ and focus our attention on the domain
$$\mathscr{D}_{\1}:=\bigcup_{\ell \geq 0}2^{-\ell}\mathscr{B}(\bx_\1, \tfrac{\epsilon_{\bx_\1}}8).$$

Due to the fact that 
$\nabla f$ is homogeneous of degree 
$d-1$, it follows that 
$\nabla f: \mathscr{D}_{\1}
\to \nabla f(\mathscr{D}_{\1})$ 
is a diffeomorphism. 
It can be readily verified 
that the inverse map 
$(\nabla f)^{-1}$ is a 
homogeneous function 
of degree $\frac{1}{d-1}$ 
on $$\mathscr{D}_{\d1}:=\nabla f(\mathscr{D}_{\1}).$$
Using homogeneity, 
this set can be expressed as a union 
of dyadic scalings of \\
$\nabla f(\sB(\bx_{\1},\epsilon_{\bx_{\1}}) )$. Indeed
$$\mathscr{D}_{\d1}=
\bigcup_{\ell \geq 0}
\nabla f(2^{-\ell}\mathscr{B}(\bx_{\1}, 
\tfrac{\epsilon_{\bx_{\1}}}8))=
\bigcup_{\ell \geq 0}
2^{-\ell(\d1)}\nabla 
f(\mathscr{B}(\bx_{\1}, \tfrac{\epsilon_{\bx_{\1}}}8)).$$
 
Next, we use a smooth dyadic sum 
of bump functions to approximate 
the characteristic function of $\mathscr{D}_1$. 
For technical reasons, 
we need to be careful
with the support of these functions. At
every step outlined in the introduction, we need 
to ensure that there is a 
small amount of room 
for the support of the said 
functions to be expanded later-on, 
while still being contained 
inside a region where
$\nabla f$ is invertible. 
\begin{definition}
    Let $\eta>0$. We say that a smooth non-negative function $w$ is an \emph{$(f,\eta)$ admissible weight function} if $\nabla f$ is invertible 
on the $\eta$-thickening of the support of $w$, that is, on the set given by
$$\{\bx\in \bRn: 
\mathrm{dist}(\bx,\supp(w))\leq \eta \}.$$

\end{definition}
We introduce 
a bump function 
$\varrho: \bR^{n-1}\to [0,1]$ such that
\begin{equation}
    \label{eq varrho prop}
    \supp\,{\varrho}\subseteq \mathscr{B}(\bx_{\1},\tfrac{\epsilon_{\bx_{\1}}}4),\, \qquad \varrho\equiv 1 \,\text{ on }\, \mathscr{B}(\bx_{\1},\tfrac{\epsilon_{\bx_{\1}}}8).
\end{equation}

Clearly $\varrho$ is $(f, \tfrac{3\epsilon_{\bx_{\1}}}{4})$ admissible. Moreover, if $\epsilon_{\bx_{\1}}$ (or $\upsilon$) is chosen sufficiently small, we have

\begin{equation}
    \label{eq varrho pou}
    \sum_{\ell\geq 0}\varrho (2^{\ell}\cdot\bz)\geq 1
\,\, \mathrm{for\, all}
\,\,
\bz\in \mathscr{D}_1.
\end{equation}

Observe that the $\sum_{\ell\geq 0}
\varrho(2^\ell\cdot)$ is supported in the set
$$\bigcup_{\ell \geq 0}
2^{-\ell}\mathscr{B}(\bx_1, \tfrac{\epsilon_{\bx_{\1}}}4).$$


\subsection{Homogeneity and Legendre Duality}
\label{subsec hom func duality} 

As discussed in the last subsection,
$\nabla f: \mathscr{D}_1\to 
\mathscr{D}_{d-1}=\nabla f(\mathscr{D}_1)$ 
is a diffeomorphism. 
We can thus define the Legendre dual of 
$f$ on $\mathscr{D}_{d-1}$ as follows 
\begin{equation}
\label{def Legendre dual}
\widetilde{f}(\mathbf{y}):=\mathbf{y}\cdot(\nabla f)^{-1}(\mathbf{y})-(f\circ(\nabla f)^{-1})(\mathbf{y}).
\end{equation}
 It can be checked that $\tilde{f}$ is homogeneous of degree $\frac{d}{d-1}$. Its gradient $\nabla \tilde{f}$ is smooth away from the origin, and is homogeneous of degree $\frac{1}{d-1}$. In fact, we have
\[\nabla \tilde{f}(\mathbf{y})=(\nabla f)^{-1}(\mathbf{y})\]
for all $\mathbf{y}\in\mathscr{D}_{d-1}$ and 
\begin{equation}
    \label{eq dual Hess}
    H_{\tilde{f}}(\by)=H_f({\bx})^{-1}, \quad \text{whenever } \by=\nabla f(\bx)\neq \mathbf{0}.
\end{equation}
In fact, substituting $\by=\nabla f(\bx)$ in \eqref{def Legendre dual}, we get
\begin{equation}
    \label{eq Ldual2}
    \tilde{f}(\by)=\by\cdot\bx-f(\bx).
\end{equation}
On account of \eqref{eq Hess cond} and \eqref{eq dual Hess}, we conclude that
\begin{equation*}
    C_f^{-1}<|\det\, H_{\tilde{f}}(\mathbf{y})|<c_f^{-1}
\end{equation*}
for all $\by\in \nabla 
f(\mathscr{B}(\bx_1, \tfrac{\epsilon_{\bx_{\1}}}8)).$
We are now ready to define the so-called dual manifold $$\mathscr{M}(\Tilde{f}):=\{(\mathbf{y}, \tilde{f}(\mathbf{y})): \mathbf{y}\in \mathscr{D}_{d-1}\}.$$
\begin{remark}
The basis of the proof of Theorem \ref{thm: main} is an induction argument, in which we keep switching between dyadic pieces of manifolds described by $f$ and its dual $\tilde{f}$. We need to establish the corresponding estimates for both the functions individually.
In order to avoid ambiguity and emphasize the role of the degree of homogeneity in our arguments, we shall denote $f$ by $f_\1$ and $\tilde{f}$ by $f_{\d1}$, so that $f_\p$ is homogeneous of degree $\frac{d}{\p}$.
\end{remark}
\begin{remark}
\label{rem self duality}   
For $\by=\nabla f_\1(\bx)$,
Equation \eqref{eq Ldual2} reads
$$
    f_{\d1}(\by)=\tilde{f_\1}(\by)=\by\cdot\bx-f_\1(\bx).
$$
One can repeat the entire argument above for $f_{\d1}$ instead of $f_{\1}$ and thus see that
$$\widetilde{(f_{\d1})}=\tilde{\tilde{f_{\1}}}=f_{\1}.$$
In other words, the operation of taking Legendre dual is an involution. Moreover, since $f_{\1}$ is smooth away from the origin, so is $f_{\d1}$. 
\end{remark}

Henceforth, we shall assume that $f_1$ is homogeneous of degree
\begin{equation}
    \label{eq d geq 2}
    d\geq 2\,,
\end{equation}
and that $f_{d-1}$ is homogeneous of degree $\frac{d}{d-1}\in (1, 2]$. The complementary case can be handled by switching the roles of $f_1$ and its Legendre dual.

\subsection{Duality of weights} 
\label{subsec dual weight} 
We need to divide the respective domains into dyadic regions based on the ``size of the curvature'' of $f_\1$ and $f_{\d1}$. We have already introduced a family of smooth bump functions associated to such a partition on $\mathscr{D}_{\1}$ in \S\ref{subsec local}. The pushforward of these bump functions gives rise to a corresponding partition on $\mathscr{D}_{\d1}$. An interesting point is that the sizes of these dyadic regions on $\mathscr{D}_{\d1}$ are determined by the homogeneity of $f_{\d1}$. We now present the details.

Let $\varrho$ be the smooth bump function introduced in \S\ref{subsec local}, which satisfies \eqref{eq varrho prop}-- \eqref{eq varrho pou} and is $(f_1,\tfrac{3\epsilon_{\bx_{\1}}}4)$ admissible. 
We define its dual $\varrho^{[\d1]}:\mathbb{R}^{n-1}\to \bR_{\geq 0}$
to be 
\begin{equation*}
    \varrho^{[\d1]}(\by)=\varrho((\nabla f_\1)^{-1}(\by))=\varrho((\nabla f_\d1)(\by)).
\end{equation*}
The function $\varrho^{[\d1]}$ is $(f_{d-1}, C\epsilon_{\mathbf{x}_1})$ admissible (where $C$ depends only on $f_\1$), and satisfies the properties
\begin{equation*}
    \supp\,{\varrho^{[\d1]}}\subseteq \nabla 
f_\1(\mathscr{B}(\bx_{\1},\tfrac{\epsilon_{\bx_{\1}}}4)),\, \qquad \varrho^{[\d1]}\equiv 1 \,\text{ on }\, \nabla 
f_\1(\mathscr{B}(\bx_{\1},\tfrac{\epsilon_{\bx_{\1}}}8)).
\end{equation*}
Observe that the map $$\by\mapsto \sum_{\ell\geq 0}\varrho^{[\d1]}(2^{(d-1)\ell}\by)$$ is supported in the set
$$
\bigcup_{\ell \geq 0}
2^{-\ell(\d1)}\nabla 
f_\1(\mathscr{B}(\bx_{\1}, \tfrac{\epsilon_{\bx_{\1}}}4))=\bigcup_{\ell \geq 0}
\nabla f_{\1}(2^{-\ell}\mathscr{B}(\bx_{\1}, 
\tfrac{\epsilon_{\bx_{\1}}}4)).$$
Moreover, for all $$\by\in \sD_{\d1}:=
\bigcup_{\ell \geq 0}
\nabla f_\1(2^{-\ell}\mathscr{B}(\bx_{\1}, 
\tfrac{\epsilon_{\bx_{\1}}}8))=
\bigcup_{\ell \geq 0}
2^{-\ell(\d1)}\nabla 
f_\1(\mathscr{B}(\bx_{\1}, \tfrac{\epsilon_{\bx_{\1}}}8)),$$ we have
\begin{equation*}
    \sum_{\ell\geq 0}\varrho^{[\d1]} (2^{(\d1)\ell}\cdot\by)
=\sum_{\ell\geq 0}\varrho((\nabla f_1)^{-1}(2^{(\d1)\ell}\by))=
\sum_{\ell \geq 0}\varrho(2^{\ell}(\nabla f_1)^{-1}(\by))\geq 1.
\end{equation*}

\subsection{Smooth Counting Functions}
\label{subsec smooth}
Let $b,\om: \bR \rightarrow [0,1]$
be smooth and compactly supported functions, with $\om$ satisfying \eqref{def: om partion of unity}. Let $$b_{\delta}(t):=b(t/\delta).$$ Further, for $\ell\in\mathbb{Z}_{\geq 0}$, we recall our notation
\[
\varrho_{\ell}^{[1]}(\mathbf{z}):= \varrho^{[1]}(2^{\ell}\mathbf{z}),\qquad \varrho_{\ell}^{[\d1]}(\mathbf{z}):= \varrho^{[1]}(2^{\ell(d-1)}\mathbf{z}).
\]
Using the weights above, we define the smoothed counting function 
\begin{equation*}
N_{f_1}^{\om,\varrho^{[\1]},b}(\delta,Q)
:= \sum_{(q,\ba) \in \mathbb{Z}
\times \bZn}
\om\Big(\frac{q}{Q}\Big)
\sum_{\ell\geq 0}\varrho^{[\1]}_\ell \Big(
\frac{\mathbf{a}}{q}\Big)
b_{\delta}\Big({\Vert q 
f_\1(\mathbf{a}/q)\Vert}\Big),
\end{equation*}
and its dual
\begin{equation*}
N_{f_{\d1}}^{\om,\varrho^{[\d1]},b}(\delta,Q)
:=\sum_{(q,\ba) \in \mathbb{Z}
\times \bZn}
\om\Big(\frac{q}{Q}\Big)
\sum_{\ell\geq 0}\varrho^{[\d1]}_{(\d1)\ell} \Big(
\frac{\mathbf{a}}{q}\Big)
b_{\delta}\Big({\Vert q 
f_{\d1}(\mathbf{a}/q)\Vert}\Big).  
\end{equation*}

We are now poised to state the smooth version of Theorem \ref{thm: main}. 
\begin{thm}
\label{thm: main smoothed}
Let $\varepsilon>0$ and $d\geq 2$
be real numbers, $n\geq 3$ 
be an integer and $b,\om: \bR \rightarrow [0,1]$ be smooth and compactly supported functions, with $\supp(b) \subseteq (-1/3,4/3)$ and $\om$ satisfying \eqref{def: om partion of unity}. Let $f_{\1}\in \mathscr{H}_d^{\mathbf{0}}(\bR^{n-1})$ and suppose $\varrho^{[\1]}: \bR^{n-1} \rightarrow [0,1]$ is a smooth compactly supported function satisfying \eqref{eq varrho prop} and \eqref{eq varrho pou}. 
Finally, let $f_{\d1}$ be the Legendre dual of $f_\1$ (defined locally as above) and $\varrho^{[\d1]}$ be the dual of $\varrho^{[\1]}$ under $\nabla f_{\1}$. Then
$$
\delta Q^{n} +
\Big(\frac{\delta}{Q}\Big)^{
\frac{n-1}{d}}
Q^{n}
\ll
N_{f_\1}^{\om,\varrho^{[\1]},b}(\delta,Q) 
\ll \delta Q^{n} +
\Big(\frac{\delta}{Q}\Big)^{
\frac{n-1}{d}}
Q^{n+k\varepsilon}
\,\, \mathrm{for}\,\, \delta \in 
(Q^{\varepsilon-1}, 1/2),
$$
where $k>0$ is a constant depending 
on $n$ only, and
$$
\delta Q^{n} 
\ll
N_{f_\d1}^{\om,\varrho^{[\d1]},b}(\delta,Q) \ll 
\delta Q^{n} +
Q^{n - \frac{2(n-1)}{d} +k\varepsilon}
\,\, \mathrm{for}\,\, \delta \in 
(Q^{\varepsilon-1}, 1/2).
$$
The implied constants depend exclusively
on $\varepsilon,f_1,n,d,\om,\varrho,b$ 
but not on $\delta$, or $Q$.
\end{thm}

\begin{remark}\label{rem exclusion}
    Note that the upper and lower bounds in the second estimate above do not match precisely when $n-\tfrac{2(n-1)}{d}\geq n-1$. This is the case when $d\geq 2(n-1)$, or equivalently, when 
    $$d'=\tfrac{d}{d-1}\in (1,\tfrac{2(n-1)}{2n-3}].$$
\end{remark}
The rest of the paper is devoted to the proof of Theorem \ref{thm: main smoothed}. The fact that it implies Theorem \ref{thm: main} is quite standard. For the convenience of the reader, we include a proof in Section \ref{app: smoothing} of the appendix.
From now on, we fix a choice of $d\geq 2$, $f_{\1}\in \mathscr{H}_d^{\mathbf{0}}(\bR^{n-1})$, $b, \omega$ and $\varrho^{[\1]}$ satisfying the conditions prescribed above, and define
\begin{equation*}
    N_{f_\1}(\delta,Q):=N_{f}^{\om,\varrho^{[\1]},b}(\delta,Q),\, \qquad N_{f_\d1}(\delta,Q):=N_{f_{\d1}}^{\om,\varrho^{[\d1]},b}(\delta,Q).
\end{equation*}

For technical reasons, it will be useful to consider not just weights supported in these domains, but also those supported in slight ``thickenings'' of these sets, which are still $(f_\1, \eta)$ or $(f_{\d1}, \eta)$ admissible for small $\eta$. To this effect, let $\rho$ be an $(f_\1, \tfrac{\epsilon_{\bx_{\1}}}{2})$ admissible weight supported in the ball 
\begin{equation}
\label{eq W1def}
    \mathscr{W}_{\1}:=\mathscr{B}(\bx_\1, \tfrac{\epsilon_{\bx_{\1}}}{2}).
\end{equation}
Then its pushforward $\rho^{[\d1]}$ under $\nabla f_{\1}$ is a weight supported in the set 
\begin{equation}
\label{eq Wd1def}
    \mathscr{W}_{\d1}:=\nabla f_{\1}(\mathscr{B}(\bx_{\1}, \tfrac{\epsilon_{\bx_{\1}}}{2})).
\end{equation} More precisely, we define $\rho^{[\d1]}:\mathscr{W}_{\d1}\to \bR_{\geq 0}$
to be 
\begin{equation*}
    \rho^{[\d1]}(\by)=\rho((\nabla f_\1)^{-1}(\by))=\rho((\nabla f_\d1)(\by)).
\end{equation*}
This dual weight $\rho^{[\d1]}$ is an $(f_{\d1}, C\epsilon_{\bx_{\1}})$ admissible weight, where the constant $C$ depends only on $f_\1$. To keep our notation uniform, 
we shall denote an $(f_\1,\eta)$ 
admissible weight $\rho$ 
supported in a thickening of $\mathscr{D}_1$ 
by $\rho^{[\1]}$. 

As generalisations of the smoothed counting functions associated to the weight $\varrho^{[\1]}$ and $\varrho^{[\d1]}$, we define 
\begin{equation*}
N_{f_1}^{\om,\rho^{[\1]},b}(\delta,Q)
:= \sum_{(q,\ba) \in \mathbb{Z}
\times \bZn}
\om\Big(\frac{q}{Q}\Big)
\sum_{\ell\geq 0}\rho^{[\1]}_\ell \Big(
\frac{\mathbf{a}}{q}\Big)
b_{\delta}\Big({\Vert q 
f_{\1}(\mathbf{a}/q)\Vert}\Big),
\end{equation*}
and
\begin{equation*}
N_{f_{\d1}}^{\om,\rho^{[\d1]},b}(\delta,Q)
:=\sum_{(q,\ba) \in \mathbb{Z}
\times \bZn}
\om\Big(\frac{q}{Q}\Big)
\sum_{\ell\geq 0}\rho^{[\d1]}_{(\d1)\ell} \Big(
\frac{\mathbf{a}}{q}\Big)
b_{\delta}\Big({\Vert q 
f_{\d1}(\mathbf{a}/q)\Vert}\Big).
\end{equation*}
To clean up notation, for $\p\in \{1,d-1\}$ and $\ell\in \mathbb{Z}_{\geq 0}$,
we also let \index{envel@$\mathfrak{N}^{\rho}(\delta,Q, \ell,\p, \r)$: smooth count, frequency and spatially localised which is associated to the envelope $\rho$}
\begin{equation}
     \label{eqn: frak nl}
     \mathfrak{N}^{\rho}(\delta,Q, \ell,\p):=
\sum_{(q,\ba) \in \mathbb{Z}
\times \bZn}
\om\Big(\frac{q}{Q}\Big)
\rho^{[\p]}_{\p\ell} \Big(
\frac{\mathbf{a}}{q}\Big)
b_{\delta}\Big({\Vert q 
f_{\p}(\mathbf{a}/q)\Vert}\Big).
\end{equation}

\begin{remark}
From now on until the end of Section \ref{sec: counting rational points to oscillatory}, we shall work with the smoothed counting functions defined above, associated to the weights $\rho^{[1]}$ and $\rho^{[\d1]}$. In Section \ref{sec: envelope} we shall construct a family of weights, starting from $\varrho$, supported in an increasing sequence of sets contained in $\mathscr{W}_1$ (or $\mathscr{W}_{\d1}$). The weight $\rho$ works as a placeholder for members of this family. Working with a more general weight will facilitate and ease our inductive argument later-on.    
\end{remark}

\subsection{Oscillatory Integrals}\label{subsec: oscillatory}
Let $u\in C^{\infty}_0(\mathbb{R}^{n-1})$ and $\phi\in C^{\infty}(\mathbb{R}^{n-1})$. We need several
facts about oscillatory integrals of the form\index{oscillatory_integal@$I(\lambda,u,\phi)$:
oscillatory integral with phase function 
$\lambda \phi$ and amplitude function $u$}
\begin{equation*}\label{def: oscillatory integral}
I(\lambda,u,\phi):=\int_{\mathbb{R}^{n-1}}\,u(\mathbf{x})\,e(\lambda\phi(\mathbf{x}))\,\mathrm{d}\mathbf{x}.
\end{equation*}
The first lemma,
the Non-Stationary Phase Principle, 
quantifies that if an oscillatory integral
$I(\lambda,u,\phi)$ is such that $\nabla\phi$ is somewhat
large (throughout the support of the amplitude function $u$), 
then $\lvert I(\lambda,u,\phi)\rvert$ is very small in terms of $\lambda$.
\begin{lem}[Non-Stationary Phase Principle]
\label{lem: non-stationary phase}
Let $m\geq 1$ be an integer.
Let $\mathscr{K}\subseteq\mathbb{R}^{m}$
be a compact set and 
$\mathscr{X}\subseteq\mathbb{R}^{m}$ 
be an open
neighbourhood of $\mathscr{K}$. 
If $\lambda\geq1$ is a
real number and $A\geq1$ is an integer, 
then
\[
\vert I(\lambda,u,\phi)\vert\leq \lambda^{-A} 
C
\sum_{\vert\boldsymbol{\alpha}\vert\leq A}
\sup_{\mathbf{x}\in\mathbb{R}^{m}}
(\Vert u^{(\boldsymbol{\alpha})}
(\mathbf{x})\Vert_{2}\,
\Vert\nabla\phi(\mathbf{x})
\Vert_{2}^{\vert\boldsymbol{\alpha}\vert-2A}).
\]
Here the constant $C>0$ depends 
only on an upper bound on 
$\Vert\phi\Vert_{C^{A+1}(\mathscr{X})}$.
\end{lem}

\begin{proof}
This follows by repeated partial integration, see
 Theorem 7.7.1 in H\"{o}rmander's book \cite{Hoermander: The Analysis}.
\end{proof}
We also need to deal with the cases that $\nabla\phi(\mathbf{x})$
can be very small (somewhere on the support of the amplitude function $u$). 
The analysis of the oscillatory integral in this case, 
provided the Gaussian curvature of the hypersurface parametrized by $\phi$ does not vanish,
is the content of the classical stationary phase principle. Recall that the Hessian of $\phi$ at a point $\mathbf{x}_0\in \mathbb{R}^{n-1}$ is the matrix $$H_{\phi}(\bx_0):= 
\Big( \frac{\partial^2}
{\partial x_i \partial x_j}
\phi(\bx)
\Big)_{1\leq i,j\leq n-1}.$$
The following version of the stationary phase principle suffices for us:

\begin{thm}[Stationary Phase Principle]
\label{thm: stationary phase higher order terms}
Let $m\geq 1$ be an integer,
$\mathscr{X}\subseteq\mathbb{R}^{m}$
be a compact set, 
and $\mathscr{\widetilde{X}}$ 
be an open neighbourhood
 of $\mathscr{X}$. Let $u\in C_{0}^{\infty}(\mathscr{\widetilde{X}})$ and $\phi\in C^{\infty}(\mathbb{R}^{n-1})$. 
Suppose
$\phi$ has at most one critical point 
$\mathbf{x}_{0}$ in the support of $u$
where $\nabla\phi(\mathbf{x}_{0})=\mathbf{0}$. 
Further, assume 
$\det\, H_{\phi}(\mathbf{x}_{0})\neq 0$ 
and that the signature $\sigma$ 
of $H_{\phi}(\mathbf{x})$
does not change on $\mathscr{\widetilde{X}}$. 
Let 
\[
g_{\mathbf{x}_{0}}(\mathbf{x}):=\phi(\mathbf{x})
-\phi(\mathbf{x}_{0})-\frac{1}{2}\left\langle H_{\phi}(\mathbf{x}_{0})(\mathbf{x}-\mathbf{x}_{0}),
\mathbf{x}-\mathbf{x}_{0}\right\rangle, 
\]
which vanishes of order three at $\mathbf{x}_{0}$.
Set $\mathrm{D} :=-\ri(\frac{\partial}{\partial x_1},\ldots, \frac{\partial}{\partial x_m}) $.
For $\tau\in\mathbb{Z}_{\geq0}$,
we define the differential operator \index{stationary_phase@$\mathfrak{D}_{\tau,\phi}$:
the differential operator of order $3\tau$,
of index $\tau\geq 0$, attached to the decay factor
$\lambda^{-\tau}$ which occurs
in the stationary phase expansion 
of $I(\lambda,u,\phi)$}
\begin{equation}
(\mathfrak{D}_{\tau,\phi}u)(\mathbf{x}):=
\frac{1}{(2 \mathrm{i})^{\tau}}
\sum_{0 \leq \mu \leq 2\tau }
\frac{1}{2^{\mu}\mu! (\mu + \tau )!}
\left\langle (H_{\phi}(\mathbf{x}_{0}))^{-1}\mathrm{D},\mathrm{D}
\right\rangle ^{\mu + \tau}(g_{\mathbf{x}_{0}}^{\mu} \cdot u)(\mathbf{x})
\label{def: differential operator with phase function dependence}
\end{equation}
for $\tau\geq1$ and for $\tau=0$ we let
$(\mathfrak{D}_{0,\phi}u)(\mathbf{x}):=u(\mathbf{x})$.
If $\lambda\geq1$, then the finite sum\index{stationary_phase@$I(\lambda,u,\phi,t)$: the
$t$-term stationary phase expansion 
of $I(\lambda,u,\phi)$}
\[
I(\lambda,u,\phi,t):=\frac{e(\lambda\phi(\mathbf{x}_{0})+\sigma)}{(\vert\det(\lambda H_{\phi}(\mathbf{x}_{0}))\vert)^{\frac{1}{2}}}\sum_{0\leq\tau\leq t}\frac{(\mathfrak{D}_{\tau,\phi}u)(\mathbf{x}_{0})}{\lambda^{\tau}}
\]
approximates $I(\lambda,u,\phi)$ with the inequality
\[
\vert I(\lambda,u,\phi)-I(\lambda,u,\phi,t)\vert\leq C\Vert u\Vert_{C^{2t}(\mathscr{X})}\lambda^{-t}.
\]
Here $C$ exclusively depends on $\Vert\phi\Vert_{C^{3t+1}(\mathscr{X})}$
and $\Vert\widetilde{\phi}\Vert_{C(\mathscr{X})}$ where $$
\widetilde{\phi}(\mathbf{x}):=\Vert\mathbf{x}-\mathbf{x}_{0}\Vert_{2}/\Vert\nabla\phi(\mathbf{x})\Vert_{2}.
$$
\end{thm}

\begin{proof}
This is a special case of \cite[Theorem 7.7.5]{Hoermander: The Analysis}. 
\end{proof}
\subsection{Geometric Summation}\label{subsec: geometric sum} We require the following 
smoothed version of geometric 
summation,
which is slightly superior 
to the un-smoothed 
geometric summation 
(the latter tends to lose a 
logarithmic factor).

\begin{lem}[Geometric Summation]
\label{lem: smooth geometric summation}
Let $\mu \in \mathbb{R}$,
$g:\mathbb{R} \rightarrow\mathbb{C}$
be smooth, and $\mathrm{supp}(g)\subseteq (1/10, 10)$.
Then for any integer $A\geq 1$, \index{A@$A\geq 1$: 
arbitrarily large integer
quantifying rapid decay of Fourier transforms} we have
\begin{equation}\label{eq: smooth geo sum}
\sum_{q\in \mathbb{Z}} g(\frac{q}{Q}) e(\mu q) 
= 
Q\widehat{g}(Q\Vert \mu \Vert) + 
O_A( \Vert g \Vert_{C^A} Q^{-A}).   
\end{equation}

\end{lem}

\begin{proof}
By Poisson summation 
(and by a change of variables), 
we see that 
the left hand side of 
\eqref{eq: smooth geo sum}
equals
$$
\sum_{m\in \mathbb{Z}} 
\int_{\mathbb{R}} g(\frac{y}{Q}) e(\mu y)  e(-m y) \mathrm{d}y
=
\sum_{m\in \mathbb{Z}} Q \widehat{g} (Q (\mu  - m).
$$
We decompose the right hand side into dyadic ranges
$2^{i}< \vert \mu  - m \vert \leq 2^{i+1}$ for  $ i \geq -1$, and the special regime 
$ 0\leq \vert \mu  - m \vert \leq 2^{-1} $.
If $\Vert \mu \Vert < 1/2$, then
the latter regime contains exactly one integer $m$. 
Thus we infer
$$
\sum_{0\leq \vert \mu  - m \vert \leq 2^{-1}}
Q \widehat{g} (Q (\mu  - m))
= 
(1 + \ds1 _{\{0,  1/2 \}}(\Vert \mu \Vert))
Q \widehat{g} (Q \Vert \mu \Vert ).
$$
For any given constant $A>1$, partial integration produces the estimate
$ \widehat{g} (t) \ll_A \Vert g \Vert_{C^A} (1+\vert t\vert)^{-A}$.
For each $i\in \mathbb{Z}_{\geq -1}$,
the regime of all $m$ with $2^{i}< \vert m - \mu  \vert \leq 2^{i+1}$
contributes
$$
\sum_{2^{i}< \vert \mu  - m \vert  \leq 2^{i+1}} 
\widehat{g} (Q (\mu  - m)) \ll_A \Vert g \Vert_{C^A} 2^i (Q 2^i)^{-A}.
$$
Thus
$$
\sum_{ \vert \mu  - m \vert > 1/2} 
 \widehat{g} (Q (\mu  - m)) \ll_A \Vert g \Vert_{C^A} Q^{-A}.
$$
From this \eqref{eq: smooth geo sum} follows at once
by observing that in the case when $\Vert \mu  \Vert = 1/2$,
we also have $\widehat{g} (Q \Vert\mu\Vert)=O_A(\Vert g \Vert_{C^A} Q^{-A})$.
\end{proof}

\section{Pruning}\label{sec: pruning}
Let $\varepsilon>0$ be a fixed small constant 
(from Theorem \ref{thm: main smoothed}). 
For $\p\in\{\1,\d1\}$\index{p@$p\in \{\1,\d1\}$:
a parameter used to analyse $f\in \sH_{d/1}^{\bzero}(\bRn)$ 
and $\widetilde{f} \in \sH_{d/d-1}^{\bzero}(\bRn)$ at the same time by using that the H\"{o}lder
dual $d'= d/(d-1)$},  set \index{spatial_cut_off@$L(\delta, Q),L_{\p}(Q), \Lc(\delta, Q)$: dyadic cut-offs 
to remove Knapp caps near $\bzero$}
\begin{equation}
\label{def Ls}
L(\delta, Q):=
\frac{\log(\delta^{-1}Q)}{d\log2},
\,
L_{\p}(Q) := \frac{\log Q^{1-\varepsilon}}{p \log 2}, \,
\Lc(\delta, Q) := \min \left(L(\delta, Q),L_\p(Q)\right).
\end{equation}

In this section, we fix an $(f_\1, \tfrac{\epsilon_{\bx_{\1}}}{2})$ admissible weight $\rho$ supported in the the ball 
$\sB(\bx_1, \tfrac{\epsilon_{\bx_1}}2)$. 
All the implicit constants 
are allowed to depend on $\rho$.
To avoid pathological situations, 
we assume without loss of 
generality from now on that 
$$
\varepsilon \in \Big(0,\frac{1}{100}\Big).
$$
Next, recall the definition of 
$\mathfrak{N}^{\rho}
(\delta, Q, \ell,\p)$ 
from \eqref{eqn: frak nl}. 
We decompose
\begin{equation*}
N_{f_{\p}}^{\om,\rho^{[\p]},b}(\delta,Q)
=\sum_{0\leq\ell\leq \Lc(\delta, Q)}
\mathfrak{N}^{\rho}
(\delta, Q, \ell,\p)+
\sum_{\Lc(\delta, Q)<\ell}
\mathfrak{N}^{\rho}
(\delta,Q, \ell,\p). 
\end{equation*}

Note that $$\fL_{\d1}(\delta, Q)
\leq \fL_{\1}(\delta, Q)$$ since 
$d-1\geq 1$, recall \eqref{eq d geq 2}.
\begin{remark}
When $\delta$ and $Q$ are fixed and
clear from the context,  
we shall suppress 
the dependence on these parameters and 
refer to the quantities 
above simply by $L, L_{\p}$ and $\Lc$.    
\end{remark}

\subsection{Pruning in Physical Space}
\label{subsec: prun phys}
For small $\ell$, the sub-manifold 
$$\{(\mathbf{x},f_\p(\mathbf{x})):\,\mathbf{x}\in\mathrm{supp}(\rho^{[\p]}_{\ell\p})\}$$
is amenable to a stationary phase analysis because, as it will turn
out, it still possesses some curvature. However, the pieces 
$\{(\mathbf{x},f(\mathbf{x})):\,\mathbf{x}\in\mathrm{supp}(\rho^{[\p]}_{\ell\p})\}$
corresponding to $\ell>L(\delta, Q)$ are located too close to the origin, where the curvature either vanishes (for $\p=\1$), or blows up ($\p=\d1$).
Hence, we utilize the localisation instead. 
The parameter\index{frequency_cut_off@$J=\delta^{-1} Q^{\varepsilon}$: frequency truncation to detect $\Vert q f(\ba/q)\Vert \leq \delta$}
\begin{equation}
J:=
Q^{\varepsilon}\delta^{-1} \label{def: J}
\end{equation}
is crucial in 
this endeavor.

It is useful to record 
the following relations
\begin{equation}
2^{L(\delta, Q)}=
(\delta^{-1}Q)^{1/d}
\ll (QJ)^{1/d},
\label{eq: 2^L}
\end{equation}
and 
\begin{equation}
2^{\p{L_\p(Q)}}=Q^{1-\varepsilon}.
\label{eq: 2^Lp}
\end{equation}

\begin{lem}\label{lem: tail terms}
Let $\delta\in (0,1/2)$ and $Q\geq 1$. 
Then 
\[
\sum_{L(\delta, Q)<\ell}
\mathfrak{N}^{\rho}(\delta,Q, \ell,\p)
\ll
\left(\frac{\delta}{Q}\right)^{\frac{\p(n-1)}{d}}Q^{n}+Q,
\]
and
\[
\sum_{L_\p(Q)<\ell}
\mathfrak{N}^{\rho}(\delta,Q, \ell,\p)
\ll
Q^{1+(n-1) \varepsilon}
\]
for $\p\in\{\1, \d1\}$, with the implied constants independent of $\delta$ and $Q$.
\end{lem}

\begin{proof}
Recall the notation
$ \overline{\sB} (\bzero,R)= \{ \bx \in \bR^{n-1}: 
\Vert \bx \Vert_2 \leq  R\}$.
Observe that 
\[
\sum_{L<\ell}\rho^{[\p]}_{\ell\p}(\bz)\leq\mathds{1}_{\overline{\sB} (\bzero,2^{-\p L(\delta, Q) +2})}(\bz),\qquad \bz\in \mathbb{R}^{n-1}.
\]
Therefore, we can estimate
\begin{align*}
\sum_{L(\delta, Q)<\ell}\mathfrak{N}^{\rho}(\delta,Q, \ell,\p)&=\sum_{L(\delta, Q)<\ell}\,\sum_{(q,\ba) \in \mathbb{Z}
\times \bZn}
\om\Big(\frac{q}{Q}\Big)
\rho^{[\p]}_{\p\ell} \Big(
\frac{\mathbf{a}}{q}\Big)
b_{\delta}\Big({\Vert q 
f_{\p}(\mathbf{a}/q)\Vert}\Big)\\
&\leq
\Vert \rho^{[\p]}\Vert_{\infty}
\Vert b\Vert_{\infty}
\sum_{q\in \mathbb{Z}}
\om(\frac{q}{Q})
\#\{\mathbf{a}\in\mathbb{Z}^{n-1}:\Vert\mathbf{a}\Vert_{2}
\leq q2^{-\p L(\delta, Q)+2}\}\\
 & \ll
\sum_{q\in \mathbb{Z}}
\om(\frac{q}{Q})\big(q2^{-\p L(\delta, Q)+2}+1\big)^{n-1}
 \\
 & \ll Q\max_{0\leq t\le n-1}(Q2^{-\p L(\delta, Q)})^{t}=\max(Q,Q^{n}2^{-\p L(\delta, Q)(n-1)}).
\end{align*}
Since
$2^{-\p L(\delta, Q)}
= 
(\frac{\delta}{Q})^{\frac{\p}{d}}$,
we get 
\[
\sum_{L(\delta, Q)<\ell}\,\sum_{(q,\ba) \in \mathbb{Z}
\times \bZn}
\om\Big(\frac{q}{Q}\Big)
\rho^{[\p]}_{\p\ell} \Big(
\frac{\mathbf{a}}{q}\Big)
b_{\delta}\Big({\Vert q 
f_{\p}(\mathbf{a}/q)\Vert}\Big)
\ll\left(\frac{\delta}{Q}\right)^{\frac{\p(n-1)}{d}}Q^{n}+Q.
\]
Similarly, with $L_\p$ in place of $L$, we deduce
\begin{align*}
\sum_{L(\delta, Q)<\ell}\,\sum_{(q,\ba) \in \mathbb{Z}
\times \bZn}
\om\Big(\frac{q}{Q}\Big)
\rho^{[\p]}_{\p\ell} \Big(
\frac{\mathbf{a}}{q}\Big)
b_{\delta}\Big({\Vert q 
f_{\p}(\mathbf{a}/q)\Vert}\Big)
\ll
\sum_{q\in \mathbb{Z}}
\om(\frac{q}{Q})(Q2^{-\p L_\p+2}+1)^{n-1}.
\end{align*}
Since $Q 2^{-\p L_\p} \asymp Q^{\varepsilon}$, the right hand side is 
$O(Q \cdot Q^{(n-1)\varepsilon})$.
\end{proof}

In view of Lemma \ref{lem: tail terms}, it remains to analyse 
\[
\sum_{0\leq\ell\leq 
\Lc(\delta, Q)}
\mathfrak{N}^{\rho}(\delta,Q, \ell, \p).
\]

For each of the remaining 
$\ell$ a straightforward application 
of Poisson summation,
see Lemma \ref{lem: truncated Poisson}, 
can be used to (approximately) 
Fourier expand each 
$\mathfrak{N}^{\rho}(\delta,Q, \ell, \p)$
into the exponential sum
\[
\sum_{(q,\mathbf{a})\in
\bZ \times \mathbb{Z}^{n-1}}
\om(\frac{q}{Q})
\rho^{[\p]}_{\ell\p}
\left(\frac{\mathbf{a}}{q}\right)
\sum_{1\leq j\leq J}\delta 
\widehat{b}(j\delta)
e(jqf_{\p}(\mathbf{a}/q)).
\]
Here $J$ is as in \eqref{def: J}.
To analyse such sums, 
it is beneficial to introduce 
another smooth partition of unity.
This time in the $j$-aspect,
allowing us to decouple the $j$ 
and $q$ variables in a clean way 
--- which would not be possible otherwise.
\subsection{Partitioning Unity
and Pruning in Frequency Domain}\label{subsec: prun freq}
Recall
\begin{equation*}
\mathfrak{N}^{\rho}(\delta,Q, \ell,\p, \r) 
= 
\sum_{(j,q,\mathbf{a})\in\bZ^2 
\times \mathbb{Z}^{n-1}}
\om(\frac{q}{Q})
\rho^{[\p]}_{\ell\p}
\left(\frac{\mathbf{a}}{q}\right)
\om (\frac{j}{2^r}) 
\delta \widehat{b}(j\delta)
e(jqf_{\p}(\mathbf{a}/q)).
\end{equation*}
The cut-off parameter\index{R@$R= \log (J)/\log 2$:
frequency cut-off stemming from Knapp cap considerations}
\begin{equation}\label{def: R}
    R:= \frac{\log J}{\log 2}
\end{equation}
will make a frequent appearance 
in what follows. 
An integration by parts argument
shows that
$\widehat{\om}$
decays rapidly
in the sense that 
$\widehat{\om}(x)= 
O_A(1+\vert x\vert^A)^{-1}$
for any $A>1$.
Consequently,
we have the (approximate) decomposition 
\begin{equation*}
\mathfrak{N}^{\rho}(\delta,Q, \ell,\p)
= \delta \widehat{b}(0)
\sum_{(q,\mathbf{a})\in\bZ \times \mathbb{Z}^{n-1}}
\om(\frac{q}{Q})
\rho^{[\p]}_{\ell\p}\left(\frac{\mathbf{a}}{q}\right)
+\sum_{0 \leq r \leq R} 
\mathfrak{N}^{\rho}(\delta,Q, \ell,\p, \r)
+ O_A(Q^{-A})
\end{equation*}
for any $A>1$. Here we have used the fact that $\om$
satisfies \eqref{def: om partion of unity}.

Naturally, the first 
term on the right hand side 
stems from the zero mode and plays a special 
role. Let us evaluate this term 
before proceeding further. 

\begin{lem}
Let $A>1$. Then, 
uniformly in $0\leq \ell \leq L_p(Q)$, we have
that
\[\delta\widehat{b}(0)
\sum_{(q,\mathbf{a}) \in \bZ \times \mathbb{Z}^{n-1}}
     \om(\frac{q}{Q})
 \rho^{[\p]}_{\ell\p}\left(\frac{\mathbf{a}}{q}\right)
 =
c_{\rho,\om, b, \p} 2^{-\ell \p(n-1)}\delta Q^{n}+O_A(Q^{-A}),\]
where
\[
c_{\rho,\om, b, \p}
= 
\int_{-\infty}^{\infty} b (t) \, \rd t
\int_{\bRn} \rho^{[\p]} (\bx)  \, \rd \bx
\int_{-\infty}^{\infty} \om (t) t^{n-1} \rd t.
\]

\end{lem}

\begin{proof}
By Poisson summation and a change of variables,
\begin{equation}\label{eq: sum a}
\sum_{\mathbf{a}\in\mathbb{Z}^{n-1}}
\rho^{[\p]}_{\ell\p}\left(\frac{\mathbf{a}}{q}\right)
= q^{n-1}\sum_{\mathbf{a}\in\mathbb{Z}^{n-1}}
\widehat{\rho^{[\p]}_{\ell\p}}\left(q\mathbf{a}\right).
\end{equation}
The zero mode $\ba = \bzero$ contributes 
$q^{n-1} \widehat{\rho^{[\p]}_{\ell\p}}(\mathbf{0})$.
To estimate the contribution of $\ba \not = \bzero$,
we make two observations. 
Firstly, as before, a change of variables implies
$ \widehat{\rho^{[\p]}_{\ell\p}}(q \ba)
=
2^{-\ell \p(n-1)}
\widehat{\rho^{[\p]}}(2^{-\ell \p} q\ba).
$
Secondly, $0\leq \ell \leq L_p(Q)$
ensures $2^{-\p \ell} Q \gg Q^{\varepsilon}$.
Hence, by utilizing the rapid decay of $\widehat{\rho^{[\p]}}$,
we deduce that the contribution of 
$\ba \neq \bzero $ to \eqref{eq: sum a}
is $O_A(Q^{-A})$ for any $A>1$.

Therefore
$$
\sum_{(q,\mathbf{a}) \in \bZ \times \mathbb{Z}^{n-1}}
     \om(\frac{q}{Q})
 \rho^{[\p]}_{\ell\p}\left(\frac{\mathbf{a}}{q}\right)
= 
\sum_{q\in \bZ}
\om(\frac{q}{Q}) \big(q^{n-1} 
2^{-\ell \p(n-1)} \widehat{\rho^{[\p]}}
(\mathbf{0}) +O_A(Q^{-A})\big).
$$
Let $\Omega (t):= \om (t) t^{n-1}$
and observe $\om(\frac{q}{Q}) q^{n-1} = Q^{n-1} \Omega(\frac{q}{Q}) $.
By Poisson summation,
$$
\sum_{q\in \bZ}
\om(\frac{q}{Q}) q^{n-1} =
Q^{n-1}
\sum_{q\in \bZ}
\Omega(\frac{q}{Q}) 
=
Q^{n-1}
\sum_{r\in \bZ} Q
\widehat{\Omega}(Qr).
$$
The right hand side equals $Q^n \widehat{\Omega}(0)$ 
up to an error of $O_A(Q^{-A})$. Since 
$\widehat{\Omega}(0) = 
\int_{-\infty}^{\infty} \om (t) t^{n-1} \rd t,
$
we infer

\begin{align*}
\sum_{q\in \bZ}
\om(\frac{q}{Q}) 
\sum_{\mathbf{a}\in\mathbb{Z}^{n-1}}
\rho^{[\p]}_{\ell\p}\left(\frac{\mathbf{a}}{q}\right)
& = 
2^{-\ell\p (n-1)}\int_{\bRn} 
\rho^{[\p]} (\bx) \, \rd \bx
\sum_{q\in \bZ}
\om(\frac{q}{Q}) q^{n-1}+O_A(Q^{-A})\\
& = 
2^{-\ell\p (n-1)}Q^n
\int_{\bRn} 
\rho^{[\p]} (\bx) \, \rd \bx
\int_{-\infty}^{\infty} \om (t) t^{n-1} \rd t+O_A(Q^{-A})
\end{align*}
which produces the required relation.
\end{proof}

To proceed, we need to dispose of bad regimes where 
$r$ is small in terms of $\ell$ and $q$. 
We define\index{frequencycut@$R_{-}$: 
(dyadic) cut-off parameter to remove
Knapp caps in frequency space}
\begin{equation}
    \label{eq: R Rminus}
    R_{-}:=\frac{\log (2^{d\ell} Q^{\varepsilon -1})}{\log 2},\qquad R:=\frac{\log J}{\log 2}=\frac{\log (Q^{\varepsilon} \delta^{-1})}{\log 2}.
\end{equation}

For each $\delta\in (0, 1/2), Q\geq 1$ and $\ell\in \mathbb{Z}_{\geq 0}$, we gather the set of small $r$ into a `bad set' $\mathscr{B}(\delta, Q, \ell)$
given by 
\index{frequencycut2@$\mathscr{B}(\delta, Q, \ell)$: 
bad dyadic ranges in frequency space 
(there is no oscillation)}
\[
\mathscr{B}(\delta, Q, \ell):=
\{ r\in \bZ \cap [0,R]:\,Q 2^r \leq2^{d\ell}Q^{\varepsilon}\}
\]
and the remaining good set 
\begin{equation}\label{def: good set}
\mathscr{G}(\delta, Q, \ell):=
\{r\in \bZ_{\geq 0}:
\,2^{d\ell}Q^{\varepsilon-1}<2^r<Q^{\varepsilon}\delta^{-1}\}=\{r\in \bZ_{\geq 0}:\,R_{-}<r<R\}.
\end{equation}
The key point is that for $\ell \leq L_{\p}(Q)$ and all \index{frequencycut3@$\sG(\delta, Q, \ell)$: 
good dyadic ranges in frequency space (there is
oscillation)}$r \in \sG(\delta, Q, \ell)$, 
any $(q,j)\in Q \cdot \supp{(\om)} \times 2^r \cdot \supp{(\om)}$
satisfies the inequalities
\begin{equation*}
qj>2^{d\ell-2}Q^{\varepsilon},
\,q>2^{\ell\p-2}Q^{\varepsilon}.
\end{equation*}
\begin{remark}
If $\ell$ is so small that $2^{d \ell-3} <Q^{1-\varepsilon}$,
then the bad set $\mathscr{B}(\delta, Q, \ell)$
is empty.
\end{remark}
The next lemma bounds the contribution of
the bad set $\mathscr{B}(\delta, Q, \ell)$. Recall that $d\geq 2$ and $n\geq 3$.

\begin{lem}\label{lem: contr bad set 1}
 (i) For all $\delta\in (0, 1/2)$, $Q\geq 1$ and $0\leq \ell \leq \fL_\1(\delta, Q)$, we have
\[
\sum_{r\in\sB(\delta, Q, \ell)}\mathfrak{N}^{\rho}(\delta,Q, \ell,\1, \r)\ll 
\begin{cases}
2^{-\ell((n-1)-d)}\delta Q^{n-1+\varepsilon}, & \mathrm{if}\,2
\leq d\leq n-1,\\
\left(\frac{\delta}{Q}\right)^{\frac{(n-1)}{d}}Q^{n+\varepsilon}, 
& \mathrm{if}\,n-1\leq d.
\end{cases}
\]

(ii) For all $\delta\in (0, 1/2)$, $Q\geq 1$ and $0\leq \ell \leq \fL_{\d1}(\delta, Q)$, we have \[
\sum_{r\in\sB(\delta, Q, \ell)}\mathfrak{N}^{\rho}(\delta,Q, \ell,\d1, \r)\ll 2^{-\ell((\d1)(n-1)-d)}\delta Q^{n-1+\varepsilon}.
\]
\end{lem}
\begin{proof}
If $Q2^r > 2^{d\ell} Q^\varepsilon$,
then $\sB(\delta, Q, \ell)$ is empty and there is nothing to show.
For $2^r\leq 2^{d\ell} Q^{-1}Q^\varepsilon$ and $\p\in\{\1, \d1\}$, we have
\[
\mathfrak{N}^{\rho}(\delta,Q, \ell,\p, \r)=
\sum_{(j,q,\mathbf{a})\in\bZ^2 \times \mathbb{Z}^{n-1}}
\om(\frac{q}{Q})
\rho^{[\p]}_{\ell\p}\left(\frac{\mathbf{a}}{q}\right)
\om (\frac{j}{2^r}) \delta \widehat{b}(j\delta)
e(jqf_{\p}(\mathbf{a}/q))\]\[\ll 
\delta
\sum_{(q,\ba) \in \bZ \times \mathbb{Z}^{n-1}}
\om(\frac{q}{Q}) 
\rho^{[\p]}_{\ell\p}\left(\frac{\mathbf{a}}{q}\right)
\sum_{j\asymp 2^r}\vert \widehat{b}(\delta j) \vert .
\]
Observe
\[
\sum_{\mathbf{a}\in\mathbb{Z}^{n-1}}
\rho^{[\p]}_{\ell\p}\left(\frac{\mathbf{a}}{q}\right)\ll
(2^{-\ell\p}q +1)^{n-1}.
\]
Thus we infer
\begin{align*}
\mathfrak{N}^{\rho}
(\delta,Q, \ell,\p, \r) 
\ll
\delta 2^r\sum_{q \in \bZ}
\om(\frac{q}{Q}) (2^{-\ell\p}q +1)^{n-1} 
 \ll 
\delta 2^r Q (2^{-\ell\p}Q +1)^{n-1}.
\end{align*}
Since currently $Q\leq Q^{\varepsilon}2^{d\ell-r}$, summing up in $r\in \sB(\delta, Q, \ell)$, we get
\begin{equation}
    \label{eq bl}
    \sum_{r\in\sB(\delta, Q, \ell)}\mathfrak{N}^{\rho}(\delta,Q, \ell,\p, \r)\ll 2^{d\ell}Q^{\varepsilon-1}\delta Q(2^{-\ell\p}Q+1)^{n-1}.
\end{equation}
Recall that when $\p=1$, then 
$2^{\ell}\leq 2^{\fL_{\1}}\leq Q^{1-\varepsilon}$. Thus, 
$$\sum_{r\in\sB(\delta, Q, \ell)}\mathfrak{N}^{\rho}(\delta,Q, \ell,\1, \r)\ll 2^{\ell(d-(n-1))}\delta Q^{n-1+\varepsilon}.$$
Using the fact that $2^{\ell}\leq 2^{\fL_1}\leq 2^{L}\leq \left(\delta^{-1}Q\right)^{\frac{1}{d}}$, we see that if $n-1 \leq d$, then 
\begin{align*}
\sum_{r\in\sB(\delta, Q, \ell)}\mathfrak{N}^{\rho}(\delta,Q, \ell,\1, \r)&\ll
2^{L(d-(n-1))}\delta Q^{n-1+\varepsilon} 
\ll
(\delta^{-1}Q)^{\frac{d-(n-1)}{d}}\delta Q^{n-1+\varepsilon} 
\\&=\left(\frac{\delta}{Q}\right)^{\frac{(n-1)}{d}}Q^{n+\varepsilon}.    
\end{align*}
This proves (i). To prove (ii), we set $\p=\d1$ in \eqref{eq bl} and use the fact that $2^{\ell(\d1)}\leq 2^{(\d1)\fL_{\d1}}\leq Q^{1-\varepsilon}$ (recall \eqref{eq: 2^Lp}), to infer that
$$\sum_{r\in\sB(\delta, Q, \ell)}\mathfrak{N}^{\rho}(\delta,Q, \ell,\d1, \r)\ll 2^{\ell(d-(\d1)(n-1))}\delta Q^{n-1+\varepsilon}.$$
This establishes part (ii). Note that since $\frac{d}{\d1}\leq 2\leq n-1$ for $n\geq 3$, the expression above has a non-positive power in $2^{\ell}$.
\end{proof}

\section{From Counting Rational Points to Oscillatory Integrals} \label{sec: counting rational points to oscillatory}
Throughout this section, the weight function 
can either be 
$\rho^{[\1]}$ or $\rho^{[\d1]}$.
As before, we analyze both cases simultaneously 
by considering $\rho^{[\p]}$ with $\p\in \{1, \d1\}$. 
The implicit constants 
occurring in this section are allowed to depend on 
$\rho$ without further mention. 
We recall
\begin{equation*}
    \mathfrak{N}^{\rho}(\delta,Q, \ell,\p, \r) =
    \sum_{(j, q,\ba)\in \bZ^2 \times \bZn}\om(\frac{q}{Q})
    \om(\frac{j}{2^r})
    \rho^{[\p]}_{\ell\p}\left(\frac{\mathbf{a}}{q}\right)\delta  \widehat{b}(j\delta)
    e(jqf_{\p}(\mathbf{a}/q)).
\end{equation*}
Define 
\begin{equation*}
    \fE^{\rho}(\delta,Q, \ell, \p)
    := \sum_{r\in \sG(\delta, Q, \ell)} \fN^{\rho}(\delta,Q, \ell, \p, \r)
\end{equation*}
for $\ell \in \{0,\ldots, \fL_p(\delta,Q)\}$.
Lemma \ref{lem: contr bad set 1} 
ensures that for all $Q\geq 1$, $\delta\in (0, 1/2)$ and \ellponecond, we have
\begin{align}
    \label{eq frakN p1}
    \mathfrak{N}^{\rho}(\delta,Q, \ell, \1) & =
    c_{\rho,\om, b, \1}
    2^{-\ell(n-1)}\delta Q^{n}
    + 
    \fE^{\rho}(\delta,Q, \ell, \1)
    \\ + & 
    \begin{cases}
    O(2^{-\ell(n-1-d)}\delta Q^{n-1+\varepsilon}), &\text{ for } 2\leq d\leq n-1.\\
    O\bigg( 
    \left(\frac{\delta}{Q}\right)^{\frac{n-1}{d}}
    Q^{n(1+\varepsilon)}\bigg) &\text{ for } d \geq n-1.
    \end{cases}\nonumber
\end{align}
Similarly, for $Q\geq 1$, $\delta\in (0, 1/2)$ and \ellpdcond, we have
\begin{align}
    \label{eq frakN d1}
    \mathfrak{N}^{\rho}(\delta,Q, \ell, \d1) & =
    c_{\rho,\om, b, \d1}
    2^{-\ell(\d1)(n-1)}\delta Q^{n}\\ & +
    \fE^{\rho}(\delta,Q, \ell, \d1)
    +O(2^{-\ell((\d1)(n-1)-d)}\delta 
    Q^{n-1+\varepsilon})\nonumber.
\end{align}
Hence we now focus on bounding 
$\fN^{\rho}(\delta,Q, \ell, \p, \r)$ for $\p\in \{1, d-1\}$, 
\ellpcond, and $r \in \sG(\delta, Q, \ell)$.
The Poisson summation formula for $\bZn$ yields
\[
\sum_{\mathbf{a}\in\mathbb{Z}^{n-1}}\rho^{[\p]}_{\ell\p}\left(\frac{\mathbf{a}}{q}\right)e(jqf_{\p}(\mathbf{a}/q))=\sum_{\mathbf{k}\in\mathbb{Z}^{n-1}}\int_{\mathbb{R}^{n-1}}\rho^{[\p]}_{\ell\p}(\frac{\mathbf{x}}{q})e(qjf_{\p}(\mathbf{x}/q)-\left\langle \mathbf{k},\mathbf{x}\right\rangle )\,\mathrm{d}\mathbf{x}
\]
where $\langle \cdot,\cdot \rangle$ denotes the 
canonical Euclidean inner-product in $\bR^{n-1}$.
After a change of variables, $\mathbf{x}\mapsto q\mathbf{x}$, the
integral equals 
\[
q^{n-1}\int_{\mathbb{R}^{n-1}}\rho^{[\p]}_{\ell\p}(\mathbf{x})e(qj(f_{\p}(\mathbf{x})-\left\langle \mathbf{k}/j,\mathbf{x}\right\rangle ))\,\mathrm{d}\mathbf{x}.
\]
Changing variables again via $\mathbf{x}\mapsto2^{\ell\p}\mathbf{x}$
to normalise the amplitude function to $\rho^{[\p]}$, we conclude that 
\[
\sum_{\mathbf{a}\in\mathbb{Z}^{n-1}}
\rho^{[\p]}_{\ell\p}\left(\frac{\mathbf{a}}{q}\right)
e(jqf_{\p}(\mathbf{a}/q))=q^{n-1}2^{-\ell\p(n-1)}
\sum_{\mathbf{k}\in\mathbb{Z}^{n-1}}I(\ell,\p,q,j,\mathbf{k})
\]
where 
\[
I(\ell,\p,q,j,\mathbf{k}):=
\int_{\mathbb{R}^{n-1}}
\rho^{[\p]}(\mathbf{x})\,
e(qj(f_{\p}(2^{-\ell\p}\mathbf{x})-
2^{-\ell\p}\left\langle \mathbf{k}/j,\mathbf{x}\right\rangle ))\,
\mathrm{d}\mathbf{x}.
\]
Because $f_\p$ is homogeneous
of degree $\frac{d}{\p}$, we have 
\[
qj(f_{\p}(2^{-\ell\p}\mathbf{x})-2^{-\ell\p}\left\langle 
\mathbf{k}/j,\mathbf{x}\right\rangle )=
qj2^{-d\ell}(f_{\p}(\mathbf{x})-
2^{(d-\p)\ell}\left\langle 
\mathbf{k}/j,\mathbf{x}\right\rangle ).
\]
Thus,
\begin{equation}
I(\ell,\p,q,j,\mathbf{k})=\int_{\mathbb{R}^{n-1}}\rho^{[\p]}(\mathbf{x})\,e(qj2^{-d\ell}(f_{\p}(\mathbf{x})-2^{(d-\p)\ell}\left\langle \mathbf{k},\mathbf{x}\right\rangle ))\,\mathrm{d}\mathbf{x}.
\label{eq: osc int}
\end{equation}
The upshot is that
\begin{align*}
&\mathfrak{N}^{\rho}(\delta,Q, \ell, \p) \leq \\
& c_{\rho, \om, b, \p}
2^{-\ell\p(n-1)}
\Big( \delta Q^{n}+\sum_{r\in \mathscr{G}(\delta, Q, \ell)}
\sum_{(j, q)\in \bZ^2 }\om(\frac{q}{Q})
    \om(\frac{j}{2^r})q^{n-1}
    \delta\widehat{b}(j\delta)
    \sum_{\mathbf{k}\in\mathbb{Z}^{n-1}}I(\ell,\p,q,j,\mathbf{k})
\Big).
\end{align*}
Recall from \eqref{eq W1def} and \eqref{eq Wd1def} that we consider $\rho^{[\1]}=\rho$, resp. $\rho^{[\d1]}$,
 supported in 
$$\mathscr{W}_{\1}:=
\sB(\bx_1, \tfrac{\epsilon_{\bx_1}}2),
\qquad \mathrm{resp.}\,\,\,
\mathscr{W}_{\d1}:=\nabla f_{\1}(\mathscr{W}_{\1}).
$$
Further, we set
$$\mathscr{V}_{\1}:=\sB(\bx_1, {\epsilon_{\bx_1}}),\, \qquad \mathscr{V}_{\d1}:=\nabla f_{\1}(\mathscr{V}_{\1}).$$
Then for $\p\in \{\1, \d1\}$, using the fact that $\nabla f_{\1}: \mathscr{V}_{\1}\to \mathscr{V}_{\d1}$ is a diffeomorphism with inverse $\nabla f_{\d1}$, we get
$$\mathscr{W}_{\p}\subseteq \mathscr{V}_{\p},\, \qquad\mathscr{W}_{d-\p}=\nabla f_{\p}(\mathscr{W}_{\p}),\, \qquad \mathscr{V}_{d-\p}=\nabla f_{\p}(\mathscr{V}_{\p}).$$
Finally, we define
\[
\gamma_{\ell, \p}:=2^{-(d-\p)\ell}
\frac{\mathrm{dist}(\partial\mathscr{V}_{d-\p},
\partial\mathscr{W}_{d-\p})}{2^{10}}.
\]
For a fixed choice of $0\leq \ell\leq \Lc(\delta, Q),$ $q\leq Q$, and $j\leq J$, we
denote by $\mathscr{K}_{1}(\ell,\p,j)$ the set of critical $2^{(d-\p)\ell}\mathbf{k}/j$
located on the gradient manifold $\mathscr{W}_{\p}$, by $\mathscr{K}_{2}(\ell,\p,j)$
the set of critical $2^{(d-\p)\ell}\mathbf{k}/j$ located near (but
not on) the gradient manifold $\mathscr{W}_{\p}$, and finally by $\mathscr{K}_{3}(\ell,\p,j)$
the set of critical $2^{(d-\p)\ell}\mathbf{k}/j$
which are somewhat distant from
the gradient manifold $\mathscr{W}_{\p}$; more precisely, we let 
\begin{align*}
\mathscr{K}_{1}(\ell,\p,j) & :=\{\mathbf{k}\in\mathbb{Z}^{n-1}:\,
\mathrm{dist}(\mathbf{k}/j,2^{-(d-\p)\ell}\mathscr{W}_{d-\p}) =0\},
\\
\mathscr{K}_{2}(\ell,\p,j) & :=\{\mathbf{k}\in\mathbb{Z}^{n-1}:\,
\mathrm{dist}(\mathbf{k}/j,2^{-(d-\p)\ell}\mathscr{W}_{d-\p})\in(0,\gamma_{\ell, \p})\},\\
\mathscr{K}_{3}(\ell,\p,j) & :=\{\mathbf{k}\in\mathbb{Z}^{n-1}:\,
\mathrm{dist}(\mathbf{k}/j,2^{-(d-\p)\ell}\mathscr{W}_{d-\p})\geq\gamma_{\ell, \p}\}.
\end{align*}
Each $\mathscr{K}_{i}(\ell,\p,j)$ contributes to 
$\fN^{\rho}(\delta,Q, \ell, \p, \r)$ the term
\begin{equation}\label{def Ni}
N^{\rho}_i(\delta,Q, \ell, \p, \r):=2^{-\ell\p(n-1)}
\sum_{q,j \in \bZ} 
\om(\frac{q}{Q}) 
\om(\frac{j}{2^r})
q^{n-1}
\delta 
\widehat{b}(j\delta)
\sum_{\mathbf{k}\in\mathscr{K}_{i}(\ell,\p,j)}
I(\ell,\p,q,j,\mathbf{k})
\end{equation}
where $1\leq i\leq 3$.
We show now that $N_{3}^{\rho}(\delta,Q, \ell, \p, \r)$
(the ``non-stationary phase regime'') is not contributing all that
much. After that we demonstrate that also $N_{2}^{\rho}(\delta,Q, \ell, \p, \r)$
(the ``intermediate stationary phase regime'') is readily controllable.
This then leaves us with understanding $N_{1}^{\rho}(\delta,Q, \ell, \p, \r)$
(the ``proper stationary phase regime'') at which point we utilize a duality
argument like in \cite{Huang rational points} but use a multi-term stationary expansion.

\subsection{Non-stationary Phase Contribution}

\begin{lem}[Non-stationary phase]
\label{lem: non stationary phase regime}
Uniformly for $0\leq \ell\leq \Lc(\delta,Q)$,
$r\in \sG(\delta,Q,\ell)$,
and $\p\in \{1,d-1\}$,
we have 
\[
N^{\rho}_3(\delta,Q, \ell, \p, \r)\ll2^{-\ell\p(n-1)}.
\]
\end{lem}

\begin{proof}
Recall that,
as soon as $Q$ is sufficiently, 
we have $qj>2^{\ell d}$ for each $j\asymp 2^r$ 
because $r\in\sG(\delta,Q,\ell)$. We also recall 
the definition of $I(\ell, \p, q, j, \bk)$ from \eqref{eq: osc int}. The phase function of this integral can be re-expressed as
$$q2^{-\ell\p}[2^{-(d-\p)\ell}jf_\p(\bx)-\langle\bx, \bk\rangle].$$
Let
\[
\phi_{3}^{j,\mathbf{k}, \ell, \p}(\mathbf{x}):=\frac{2^{-(d-\p)\ell}jf_{\p}(\mathbf{x})-\left\langle \mathbf{x},\mathbf{k}\right\rangle }{\mathrm{dist}(\mathbf{k},j2^{-(d-\p)\ell}\mathscr{W}_{d-\p})},
\,\,\mathrm{and}\,\,
\lambda_{3}(j,\mathbf{k}, \ell, \p):=
q2^{-\ell\p}\mathrm{dist}(\mathbf{k},
j2^{-(d-\p)\ell}\mathscr{W}_{d-\p}).
\]
Notice 
\[
\Vert\nabla\phi_{3}^{j,\mathbf{k}, 
\ell, \p}(\mathbf{x})\Vert_2
=\left\Vert \frac{2^{-(d-\p)\ell}j
\nabla f_{\p}(\mathbf{x})
-\mathbf{k}}{\mathrm{dist}(\mathbf{k},
j2^{-(d-\p)\ell}
\mathscr{W}_{d-\p})}\right\Vert_2
\geq1.
\]
Fix a large integer $l \geq 1$.
Now we verify that 
$\Vert \phi_{3}^{j,\mathbf{k}, \ell, \p}
\Vert _{C^{l}}$
is 
uniformly bounded in $j, \ell$ 
and $\mathbf{k}\in\mathscr{K}_{3}(\ell, \p, j)$.

Take a large constant $C>1$. First suppose
$\Vert\mathbf{k}\Vert_{2}\geq 
Cj2^{-(d-\p)\ell}$. 
Since $\mathscr{W}_{d-\p}$ (and hence $2^{-(d-\p)\ell}\mathscr{W}_{d-\p}$) is 
contained in a ball of bounded radius around the origin, we have

\[
\mathrm{dist}(\frac{\mathbf{k}}{\Vert\mathbf{k}\Vert_{2}},\frac{j}{\Vert\mathbf{k}\Vert_{2}}2^{-(d-\p)\ell}\mathscr{W}_{d-\p})\geq\mathrm{dist}(\frac{\mathbf{k}}{\Vert\mathbf{k}\Vert_{2}},\mathbf{0})-\mathrm{dist}(\mathbf{0},\frac{j}{\Vert\mathbf{k}\Vert_{2}}2^{-(d-\p)\ell}\mathscr{W}_{d-\p})\geq\frac{1}{2}.
\]
Thus for $\mathbf{x}\in\mathrm{supp}(\rho^{[\p]})$, we infer
\[
|\phi_{3}^{j,\mathbf{k}, \ell, \p}(\mathbf{x})|\ll \Bigl\Vert\frac{j2^{-(d-\p)\ell}}{\Vert\mathbf{k}\Vert_{2}}f_{\p}(\mathbf{x})-\left\langle \mathbf{k}/\Vert\mathbf{k}\Vert_{2},\mathbf{x}\right\rangle \Bigr\Vert_{2}\leq\frac{1}{C}\left\Vert f_{\p}(\mathbf{x})\right\Vert _{2}+\left\Vert \mathbf{x}\right\Vert _{2}\ll1.
\]
So $|\phi_{3}^{j, \mathbf{k}, \ell, \p}(\bx)|$
is uniformly bounded for
all $\mathbf{k}, j$, 
and $\ell$ in this regime.

On the other hand, if $\Vert\mathbf{k}\Vert_{2}<Cj2^{-(d-\p)\ell}$
then 
\[|\phi_{3}^{j,\mathbf{k}, \ell, \p}(\mathbf{x})|=
\frac{\vert2^{-(d-\p)\ell}jf_{\p}(\mathbf{x})-\left\langle \mathbf{k},\mathbf{x}\right\rangle \vert}{\mathrm{dist}(\mathbf{k},j2^{-(d-\p)\ell}\mathscr{W}_{d-\p})}\ll\frac{\vert2^{-(d-\p)\ell}j\vert+\Vert\mathbf{k}\Vert_{2}}{j2^{-(d-\p)\ell}}\ll1+C\ll1.
\]
The exact argument, used for $\nabla \phi_3^{j, \bk, \ell, \p}$ instead of the phase function itself, also shows that the gradient is uniformly bounded from above as well.  
Further, any derivative of $\phi_{3}^{j,\mathbf{k}, \ell, \p}$ of order greater than one is of the form
\[
\frac{j2^{-(d-\p)\ell}}{\mathrm{dist}(\mathbf{k},j2^{-(d-\p)\ell}\mathscr{W}_{d-\p})}f_{\p}^{(\boldsymbol{\alpha})}(\mathbf{x})
\]
where $\boldsymbol{\alpha}\in\mathbb{N}^{n-1}$ 
is a multi-index of length at least $2$. 
By the construction of $\mathscr{K}_{3,\ell,j}$,
we have $$
2^{-(d-\p)\ell}/\mathrm{dist}(\mathbf{k}/j,2^{-(d-\p)\ell}
\mathscr{W}_{d-\p})\leq2^{-(d-\p)\ell}\gamma_{\ell, \p}^{-1}
$$
and since $\gamma_{\ell, \p}\gg
2^{-(d-\p)\ell}$, we deduce that 
\[
\Vert\phi_{3}^{j,\mathbf{k}, \ell, \p}\Vert_{C^{l}}\ll1
\]
uniformly in $\mathbf{k}\in
\mathscr{K}_{3}(\ell,\p,j)$ and $j$. 
Thus
we can use Lemma \ref{lem: non-stationary phase} with 
$\lambda_{3}=q2^{-\ell\p} \mathrm{dist}(\mathbf{k},j2^{-(d-\p)\ell}\mathscr{W}_{d-\p})$
and $A\geq n+1$ to obtain
\[
I(\ell,\p,q,j,\mathbf{k})\ll\lambda_{3}^{-A}=(2^{-\ell\p}q\mathrm{dist}(\mathbf{k},j2^{-(d-\p)\ell}\mathscr{W}_{d-\p}))^{-A}.
\]
Defining $\mathfrak{K}_{3,\ell,j}
:=\{\mathbf{y}\in\mathbb{R}^{n-1}:
\,\mathrm{dist}(\mathbf{y}, j
2^{-(d-\p)\ell}\mathscr{W}_{d-\p})
\geq j \gamma_{\ell, \p}\}$,
the above implies that
\[
\sum_{\mathbf{k}\in \mathscr{K}_{3}(\ell,\p,j)}
I(\ell,\p,q,j,\mathbf{k})
\ll\int_{\mathfrak{K}_{3,\ell,j}}\,(2^{-\ell\p}
q\mathrm{dist}(\mathbf{y},
j2^{-(d-\p)\ell}\mathscr{W}_{d-\p}))^{-A}\,
\mathrm{d}\mathbf{y}\ll(2^{-\ell\p}q)^{-A}
\]
where the implied constant is independent of $\ell$ and $j$. Thus,
\[
N^{\rho}_3(\delta,Q, \ell, \p, \r)\ll2^{-\ell\p(n-1)}
\sum_{q\ll Q}\sum_{j\ll J}q^{n-1}\delta 
(2^{-\ell\p}q)^{-A}.
\]
Because $r\in\mathscr{G}(\delta, Q, \ell)$ we have $2^{-\ell\p}q\gg Q^{\varepsilon}$.
Thus we can take $A$ large so that $(2^{-\ell\p}q)^{-A}\leq Q^{-n-1}$.
Hence 
\[
N^{\rho}_3(\delta,Q, \ell, \p, \r)
\ll2^{-\ell\p(n-1)}\sum_{q\ll Q}
\sum_{j\ll J}q^{-2}
\delta
\ll2^{-\ell\p(n-1)}.
\]
\end{proof}
Next we bound $N^{\rho}_2(\delta,Q, \ell, \p, \r)$. 

\subsection{Intermediate Regime Contribution}
We set
\begin{equation}
\phi^{j, \mathbf{k}, \ell, \p}(\mathbf{x}):=
f_{\p}(\mathbf{x})-2^{(d-\p)\ell}\left\langle \mathbf{k}/j,\mathbf{x}
\right\rangle,
\,\, \mathrm{and}
\,\,\lambda:=2^{-d\ell}qj.\label{def: critical points phase function-1}
\end{equation}
The unique stationary point of 
$\phi^{j, \mathbf{k}, \ell, \p}$ is 
\begin{equation}
    \label{eq: def xlkj}
    \mathbf{x}_{j,\mathbf{k},\ell, \p}:=(\nabla f_{\p})^{-1}(2^{(d-\p)\ell}\mathbf{k}/j).
\end{equation}
Because $(\nabla f_{\p})^{-1}$ 
is homogeneous of degree $\frac{\p}{d-\p}$, 
we see that
\begin{equation*}
\mathbf{x}_{j,\mathbf{k},\ell, \p}=2^{\ell\p}
(\nabla f_{\p})^{-1}(\mathbf{k}/j).
\end{equation*}
Hence, 
\[
\rho^{[\p]}(2^{-\ell\p}\mathbf{x}_{j,\mathbf{k},\ell, \p})=\rho^{[\p]}(2^{-\ell\p}2^{\ell\p}(\nabla f_{\p})^{-1}(\mathbf{k}/j))=\rho^{[d-\p]}(\mathbf{k}/j).
\]
\begin{lem}[Intermediate regime]
\label{lem: intermediate regime}
Uniformly for $0\leq \ell\leq \fL_{\p}$
and $r\in \sG(\delta, Q, \ell)$, 
we have 
\[
N^{\rho}_2(\delta,Q, \ell, \p, \r)\ll 2^{-\ell\p(n-1)}.
\]
\end{lem}

\begin{proof} 
Fix $j$, $\ell$, and $\mathbf{k}\in\mathscr{K}_{2}(\ell,\p,j)$.
We claim
$\mathbf{x}_{j,\mathbf{k},\ell, \p}
\notin\mathscr{W}_{\p}$.
To see this, suppose 
$\mathbf{x}_{j,\mathbf{k},\ell, \p}$ 
was an element
of $\mathscr{W}_{\p}$. 
Then \eqref{eq: def xlkj} gives $$\mathbf{k}/j=2^{-(d-\p)\ell}
\nabla f_{\p}(\mathbf{x}_{j,\mathbf{k},
\ell, \p})\in2^{-(d-\p)\ell}
\nabla f_{\p}(\mathscr{W}_{\p})=2^{-(d-\p)\ell}
\mathscr{W}_{d-\p}.$$
We conclude that $\mathrm{dist}(\mathbf{k}/j,
2^{-(d-\p)\ell}\mathscr{W}_{d-\p})=0$.
This contradicts the assumption $\mathbf{k}\in\mathscr{K}_{2}(\ell,\p,j)$.
So,
$\mathbf{x}_{j,\mathbf{k},\ell, \p}\notin\mathscr{W}_{\p}$. 

We plan to use Theorem \ref{thm: stationary phase higher order terms} (stationary phase) to estimate $I(\ell,\p,q,j,\mathbf{k})$ in this regime. Take a large integer $T\geq1$ to be specified 
in due course.
Because $\mathbf{x}_{j,\mathbf{k},\ell, \p}
\notin\mathscr{W}_{\p}$,
each of the functions
$\mathfrak{D}_{\tau,\phi}\rho^{[\p]}$ (arising from Theorem \ref{thm: stationary phase higher order terms} with $u=\rho^{[\p]}$ and $\phi=\phi^{j,\bk,\ell,\p}$)
vanishes when evaluated at 
$\mathbf{x}_{j,\mathbf{k},\ell, \p}$. 
Further, as $r\in\mathscr{G}(\delta, Q, \ell)$, 
we are guaranteed $\lambda = 2^{-d\ell} qj
\gg Q^{\varepsilon}$ for each $j\asymp 2^r$.

To apply Theorem \ref{thm: stationary phase higher order terms} for $I(\ell,\p,q,j,\mathbf{k})$, we also require
the phase function $\phi^{j, \bk, \ell, \p}$ and all its derivatives to be bounded (independently of $j, \bk$ and $\ell$) on the support of the amplitude function $\rho^{[\p]}$. Any derivative of $\phi^{j,\mathbf{k}, \ell, \p}$ of order greater than one is of the form
$f_{\p}^{(\boldsymbol{\alpha})}(\mathbf{x})$
where $\boldsymbol{\alpha}\in\mathbb{N}^{n-1}$ is a multi-index of length at least $2$, and is hence bounded for $\bx\in \mathscr{W}_{\p}$ which is contained in a compact set bounded away from the origin. The gradient of the phase function is given by $$\nabla \phi^{j,\mathbf{k}, \ell, \p}(\bx)=\nabla f_{\p}(\bx)-2^{(d-\p)\ell}\frac{\bk}{j}.$$
As $$\textrm{dist}(\bk/j, 2^{-(d-\p)\ell}\mathscr{W}_{d-\p})\leq \gamma_{\ell, \p},$$
we have
$$\textrm{dist}(2^{(d-\p)\ell}\bk/j, \mathscr{W}_{d-\p})\leq 2^{(d-\p)\ell}\gamma_{\ell, \p}=\frac{\mathrm{dist}(\partial\mathscr{V}_{d-\p},
\partial\mathscr{W}_{d-\p})}{2^{10}}\ll 1.$$
In other words, $2^{(d-\p)\ell}\bk/j$ lies in a 
$\frac{\mathrm{dist}(\partial\mathscr{V}_{d-\p},
\partial\mathscr{W}_{d-\p})}{2^{10}}$ thickening of 
$\mathscr{W}_{d-\p}$. Since $\mathscr{W}_{d-\p}$ 
is a bounded set itself (independent of $\bk, \ell$ and $j$) 
and $\nabla f_{\p}(\bx)$ 
is also uniformly bounded on $\mathscr{W}_{\p}$, 
we conclude that 
$$\|\nabla \phi^{j,\mathbf{k}, \ell, \p}(\bx)\|_2\leq \|
\nabla f_{\p}(\bx)\|_2+\left\|2^{(d-\p)\ell}
\frac{\bk}{j}\right\|_2\ll 1,$$
with the implicit constant independent of $\bk, j$ and $\ell$.

One can use a similar argument to show that 
$| \phi^{j,\mathbf{k}, \ell, \p}(\bx)|$ is also 
bounded from above on $\mathscr{W}_{\p}$. 
This verifies that the derivatives of 
$\phi^{j,\mathbf{k}, \ell,\p}$ are bounded on 
$\sW_{\p}$. The same is true for the derivatives of 
$\rho^{[\p]}$ (with the implicit constant again 
depending on $\rho^{[\p]}$).

Moreover, $$\det\,
H_{\phi^{j,\mathbf{k}, \ell,\p}}
(\bx_{j,\mathbf{k},\ell, \p})=\det\, H_{f_{\p}}
(\bx_{j,\mathbf{k},\ell, \p})=(\det\, H_{f_{d-\p}}
(2^{(d-\p)\ell}\bk/j))^{-1}.$$
Again, as $2^{(d-\p)\ell}\bk/j$ lies in a $2^{(d-\p)\ell}\gamma_{\ell, \p}=\frac{\mathrm{dist}(\partial\mathscr{V}_{d-\p},
\partial\mathscr{W}_{d-\p})}{2^{10}}$ thickening of $\mathscr{W}_{d-\p}$, we conclude that $$\det\, H_{f_{d-\p}}
(2^{(d-\p)\ell}\bk/j)\asymp 1.$$

Thus, we may finally apply
Theorem \ref{thm: stationary phase higher order terms}, with $\lambda=qj2^{-d\ell}\gg Q^{\varepsilon}$.
Taking into account, 
as remarked before, that the stationary phase coefficients 
$\mathfrak{D}_{\tau,\phi}\rho^{[\p]}$
vanishes in the present situation when evaluated at 
$\mathbf{x}_{j,\mathbf{k},\ell, \p}$,
we conclude that 
\begin{equation}
    \label{eq: intermediate decay}
    I(\ell,\p,q,j,\mathbf{k})= O(\lambda^{-T(n-1)})=O(Q^{-T\varepsilon (n-1)}).
\end{equation}

To show that the implied constant above is uniform
in the choice of the parameters, we need to show
that $\Vert \phi^{j, \bk, \ell, \p} \Vert_{C^{3T+1}}$-norms
are all $\ll_{T} 1$ uniformly, 
and the expression 
$\frac{\Vert\mathbf{x}-\mathbf{x}_{j,\mathbf{k},\ell, \p}\Vert_2}
{\Vert\nabla \phi^{j, \bk, \ell, \p}(\mathbf{x})\Vert_2}$ has a uniform upper bound for $\bx\in \mathscr{W}_{\p}$. We have already shown the former.
Denote by $\Vert H_{f_{\p}}(\mathbf{x})\Vert_{2}$
the spectral norm of $H_{f_{\p}}(\mathbf{x})$, 
that is the operator norm
induced by the Euclidean norm. To establish the latter, we first let
\begin{equation}
    \label{eqn:Ladef}
    \Lambda_1:=\Lambda_1(\p):=\sup_{\bx\in \mathscr{W}_{\p}}\Vert H_{f_{\p}}^{-1}(\mathbf{x})\Vert_{2}<\infty,\qquad\,\Lambda_2:=\Lambda_2(\p):=\sup_{\substack{\bx\in\mathscr{W}_{\p}\\\vert \boldsymbol{\alpha}\vert \leq 3}}
    \Vert f_{\p}^{(\boldsymbol{\alpha})}(\mathbf{x})\Vert_{2}<\infty.
\end{equation}
Notice $\Lambda_1$ and $\Lambda_2$ are are independent of $\ell, j$ and $\mathbf{k}$.  
By making a translation, if necessary, we may assume without loss
of generality that $\phi^{j, \bk, \ell, \p}(\mathbf{x}_{j,\mathbf{k},\ell, \p})=0$. We now use the Taylor expansion of $\phi^{j, \bk, \ell, \p}$ around $\xlkj$ and \eqref{eqn:Ladef} to conclude
$$\Vert \nabla \phi(\bx) - H_{f_{\p}}(\xlkj)\cdot(\bx-\xlkj) \Vert_2 
\leq C\Lambda_2\|\bx-\xlkj\|_2^2.$$ Here $C>0$ is an absolute constant.    
Dividing both sides by $\|\bx-\xlkj\|_2$ and using the triangle inequality gives
\begin{align*}
\frac{\|\nabla \phi(\bx)\|_2}{\|\bx-\xlkj\|_2}&\geq\left\|H_{f_{\p}}(\xlkj)\cdot\frac{(\bx-\xlkj)}{\|\bx-\xlkj\|_2^2}\right\|_2-C\Lambda_2\|\bx-\xlkj\|_2\\&\geq \Lambda_1^{-1}-C\Lambda_2\|\bx-\xlkj\|_2. 
\end{align*}
By choosing $\epsilon_{\bx_1}$ to be small enough, it can be arranged that $$\|\bx-\xlkj\|_2< C^{-1}\Lambda_2^{-1}\Lambda_1^{-1}/100$$ for all $\bx\in\mathscr{W}_{\p}$ (note that decreasing $\epsilon_{\bx_1}$ decreases both the supremums in \eqref{eqn:Ladef} and thus the right hand side above only gets bigger). 
This lets us conclude that 
$$\frac{\|\nabla \phi(\bx)\|_2}{\|\bx-\xlkj\|_2}\geq \frac{99}{100}\Lambda_1^{-1},$$
thus establishing the desired upper bound on the reciprocal.

Having shown that the implied constant in \eqref{eq: intermediate decay} is independent of $\ell, j$ and $\bk$, we proceed by  pointing out $\#\mathscr{K}_{2}(\ell,\p,j)
=O((2^{-\ell(d-\p)}j+1)^{n-1})$. 
Thus, choosing $T>\tfrac{3n}{n-1}\varepsilon^{-1}$, we get
\begin{align*}
N^{\rho}_2(\delta,Q, \ell, \p, \r)
&\ll 2^{-\ell\p(n-1)}
\sum_{q\leq Q}q^{n-1}\sum_{j\leq J}
\delta 
\vert \widehat{b}(j\delta)\vert
\#\mathscr{K}_{2}(\ell,j, \p)
Q^{-3n}
\\&\ll
2^{-\ell\p(n-1)}
\sum_{q\leq Q}q^{n-1}Q^{-n}
\ll2^{-\ell\p(n-1)}    
\end{align*}
as required.
\end{proof}

It is time to combine several of the bounds we have obtained.

\begin{prop}\label{prop: gathering bounds}
(i) For all $Q\geq 1$, $\delta\in (0, 1/2)$,
and \ellponecond, we have
\begin{align*}
\fN^{\rho}(\delta,Q,\ell,\1) &=
c_{\rho,\om, b, \1} 2^{-\ell (n-1)} 
\delta Q^{n}
+ \sum_{r\in \sG(\delta, Q, \ell)} N_1^{\rho}(\delta,Q,\ell,\1, \r)
+ O(2^{-\ell (n-1)} (\log J + 1))  \\
&+\begin{cases}
2^{-\ell((n-1)-d)}\delta Q^{n-1+\varepsilon}, & \mathrm{if}\,2
\leq d\leq n-1,\\
\left(\frac{\delta}{Q}\right)^{
\frac{n-1}{d}}Q^{n+\varepsilon}, 
& \mathrm{if}\,n-1 \leq d.  
\end{cases}
\end{align*}
(ii) For all $Q\geq 1$, $\delta\in (0, 1/2)$ and
\ellpdcond, we have
\begin{align*}
\fN^{\rho}(\delta,Q,\ell,\d1) &=
c_{\rho,\om, b, \d1} 2^{-\ell (\d1) (n-1)} \delta Q^{n}
+ \sum_{r\in \sG(\delta, Q, \ell)} N_1^{\rho}(\delta,Q,\ell,\d1, \r)
\\
& + O\Big(2^{-\ell(\d1) (n-1)} (\log J + 1)) 
+2^{-\ell((\d1)(n-1)-d)}\delta Q^{n-1+\varepsilon}\Big).
\end{align*}
(iii) For all $Q\geq 1$, $\delta\in (0, 1/2)$ and $\ell\in (\fL_{\d1}, \fL_{\1})\cap \mathbb{Z}$, we have
$$\fN^{\rho}(\delta,Q,\ell,\d1) = O(Q^{1+(n-1)\varepsilon}).$$
The implied constants in all estimates above are independent of $\delta, Q$ and $\ell$.
\end{prop}
\begin{proof}
Let $0\leq \ell\leq \fL_{\p}(\delta, Q)$,
and $r\in \sG(\delta, Q, \ell)$. 
By combining 
Lemma \ref{lem: non stationary phase regime}
and Lemma \ref{lem: intermediate regime} 
we infer 
$$ 
\fN^{\rho}(\delta,Q,\ell,\p,\r) 
= N_1^{\rho}(\delta,Q,\ell,\p ,\r) + O(2^{-\ell \p (n-1)})
$$
for $\p\in \{\1, \d1\}$, 
all $Q\geq 1$, $\delta\in (0, 1/2)$.
Plugging this into the relation \eqref{eq frakN p1} for $\p=\1$ and into the relation \eqref{eq frakN d1} for $\p=\d1$
gives the estimates in parts (i) and (ii). 

For deduce part (iii), we first 
note that $\fL_{\d1}(\delta, Q)<\fL_{1}(\delta, Q)$ 
if and only if $$\ell>L_{\d1}(Q)=\frac{(1-\varepsilon)\log Q}{(\d1)\log 2 }.$$ The desired estimate is then a direct consequence of the second estimate in Lemma \ref{lem: tail terms} for $\p=\d1$.
\end{proof}

\subsection{Stationary Phase Regime}
Finally, we come to the most difficult case when $\mathbf{k}\in\mathscr{K}_{1}(\ell,\p,j)$. 
The next lemma records a stationary phase expansion of 
$N_{1}^{\rho}(\delta,Q, \ell, \p,r)$,
which is a consequence of 
Theorem \ref{thm: stationary phase higher order terms}.
Again, all our constants will be uniform  
in the parameters $\delta,Q,\ell,\p$, and $r$.

From now on we make frequent use of the cut-off parameter \index{stationary_phase_cut_off@$t:= \Big\lfloor \frac{n-1}{\varepsilon} \Big\rfloor+1$:
truncation parameter used in the stationary phase expansion}
\begin{equation}\label{def: t}
    t:= \Big\lfloor 10 \frac{n-1}{\varepsilon} \Big\rfloor+1 .
\end{equation}
\begin{lem}
\label{lem: N1 contribution after stationary phase} 

Let $\p\in\{1, \d1\}$, $Q\geq 1$, $\delta\in (0, 1/2)$ and \ellpcond. Suppose $r\in\mathscr{G}(\delta, Q, \ell)$, and let $\sigma$ be the signature of $H_{f_\p}$. 
Then 
\begin{align}\label{eq: N1 stationary phase expansion}
& N_{1}^{\rho}(\delta,Q, \ell, \p, \r) =
2^{-\ell\p(n-1)}\sum_{q,j\in \bZ}
\om(\frac{q}{Q}) \om(\frac{j}{2^r})
q^{n-1}
\sum_{\mathbf{k}\in\mathscr{K}_{1}(\ell,\p,j)}
\Bigg[\Bigg(
\frac{ \delta \widehat{b}(j\delta) 
e(\sigma)}
{\vert\det(H_{f_{\p}}(\mathbf{x}_{j,\mathbf{k},\ell, \p}))
\vert^{\frac{1}{2}}}\cdot 
\\&
\cdot 
\sum_{0\leq \tau \leq t}
\frac{ (\mathfrak{D}_{\tau,\phi}\rho^{[\p]})
((\nabla f_\p)^{-1}(2^{\ell(d-\p)}\mathbf{k}/j))}{(2^{-d\ell}qj)^{\frac{n-1}{2}+\tau}}
e(-jqf_{d-\p}(\mathbf{k}/j)) \Bigg)+O(Q^{-10(n-1)})\Bigg], \nonumber
\end{align}
where the implied constant is uniform in $\ell, \bk,r$,
and the differential operator $\mathfrak{D}_{\tau,\phi}$
is as defined in Theorem \ref{thm: stationary phase higher order terms}.
\end{lem}

\begin{proof}
Recalling \eqref{def Ni}, we have
$$
N_{1}^{\rho}(\delta,Q, \ell, \p, \r) = 
2^{-\ell\p(n-1)}
\sum_{q,j \in \bZ} 
\om(\frac{q}{Q}) 
\om(\frac{j}{2^r})
q^{n-1}
\delta 
\widehat{b}(j\delta)
\sum_{\mathbf{k}\in\mathscr{K}_{1}(\ell,\p,j)}
I(\ell,\p,q,j,\mathbf{k}),
$$
where $I(\ell, \p, q, j, \bk)$ are given by 
\begin{align*}
I(\ell,\p,q,j,\mathbf{k})&=\int_{\mathbb{R}^{n-1}}\rho^{[\p]}(\mathbf{x})\,e(qj2^{-d\ell}(f_{\p}(\mathbf{x})-2^{(d-\p)\ell}\left\langle \mathbf{k},\mathbf{x}\right\rangle ))\,\mathrm{d}\mathbf{x}
\\&=\int_{\mathbb{R}^{n-1}}\rho^{[\p]}(\mathbf{x})\,e(qj2^{-d\ell}\phi^{j, \mathbf{k}, \ell,\p}(\bx))\,\mathrm{d}\mathbf{x}.    
\end{align*}

Fixing $\ell, \p, j$ and $\mathbf{k}$, we suppress the dependence on these parameters and refer to $\phi^{j, \mathbf{k}, \ell,\p}$ by $\phi$. All the implied constants in the proof will be independent of these parameters. By making a translation, if necessary, we may assume without loss
of generality that $\phi(\mathbf{x}_{j,\mathbf{k},\ell, \p})=0$. Observe that $\phi$ as in (\ref{def: critical points phase function-1})
satisfies $H_{\phi}=H_{f_{\p}}$. Using Taylor expansion of $\phi$ around
$\mathbf{x}_{j,\mathbf{k},\ell, \p}$ for 
$\bx\in \mathscr{W}_{\p}\supseteq \textrm{supp}\, \rho^{[\p]}$, we have
\[
\phi(\mathbf{x})=\frac{1}{2}\langle 
H_{f_{\p}}(\mathbf{x}_{j,\mathbf{k},\ell, \p})(\mathbf{x}-\mathbf{x}_{j,\mathbf{k},\ell, \p}),\mathbf{x}-
\mathbf{x}_{j,\mathbf{k},\ell, \p}\rangle
+O(\sup_{\substack{ \mathbf{x}\in\mathscr{W}_{\p} \\
\vert \boldsymbol{T}\vert \leq 3}}
\Vert f_{\p}^{(\mathbf{T})}(\mathbf{x})\Vert_{2}\,
\Vert\mathbf{x}-\mathbf{x}_{j,\mathbf{k},\ell, \p}\Vert_{2}^{3}).
\]
Here the implied constant is independent of 
$\ell$, $j, \mathbf{k}$, because 
$\mathscr{W}_{\p}\supseteq \mathrm{supp}(\rho^{[\p]})$ is contained in a compact set bounded away from the origin and hence $$\sup_{\substack{ \mathbf{x}\in\mathscr{W}_{\p} \\
\vert \boldsymbol{T}\vert \leq 3}}
\Vert f_{\p}^{(\mathbf{T})}(\mathbf{x})\Vert_{2}\ll 1.$$ 
By our assumption 
on $f_{\p}$, the matrix 
$H_{f_{\p}}(\mathbf{x})$ is invertible as long as
$\mathbf{x}\neq\mathbf{0}$. 
Recall from \eqref{eqn:Ladef} the quantities
\begin{equation*}
    \Lambda_1:=\Lambda_1(\p):=\sup_{\bx\in \mathscr{W}_{\p}}\Vert H_{f_{\p}}^{-1}(\mathbf{x})\Vert_{2}<\infty,\qquad\,\Lambda_2:=\Lambda_2(\p):=\sup_{\substack{\bx\in\mathscr{W}_{\p}\\\vert \boldsymbol{T}\vert \leq 3}}
    \Vert f_{\p}^{(\boldsymbol{T})}(\mathbf{x})\Vert_{2}<\infty,
\end{equation*}
which are independent of $\ell, j$ and $\mathbf{k}$. 

We intend to use the stationary phase expansion from Theorem \ref{thm: stationary phase higher order terms} to deal with the integrals $I(\ell, \p, q, j, \bk)$ occurring in the expansion of $N_{1}^{\rho}(\delta,Q, \ell, \p, \r)$. Since the eigenvalues of 
a matrix depend continuously on it, the condition $\det\, H_{f_{\p}}\asymp 1$ prevents the eigenvalues
of $H_{f_{\p}}$ from changing signs for all $\mathbf{x}\in\mathscr{W}_{\p}$.
 That is the signature $\sigma$ remains the same for all $H_{f_{\p}}(\mathbf{x}_{j,\mathbf{k},\ell, \p})$. Like before, in order to use Theorem \ref{thm: stationary phase higher order terms},  
we also need to verify that the phase function $\phi$ and its derivatives are bounded from above on $\mathscr{W}_{\p}$. 
To this effect, we use the Taylor expansion of $\phi$ around $\xlkj$, along with the observations that $\phi(\xlkj)=0$ and $\nabla\phi(\xlkj)=0$, to obtain
\[
\Vert\phi(\mathbf{x})\Vert_{2}\ll\sup_{\bz\in \mathscr{W}_{\p}}\Vert H_{f_{\p}}(\mathbf{z})\Vert_{2}\Vert\mathbf{x}-\mathbf{x}_{j,\mathbf{k},\ell, \p}\Vert_{2}^{2}+\Lambda_2(\p)\Vert\mathbf{x}-\mathbf{x}_{j,\mathbf{k},\ell, \p}\Vert_{2}^{3}\ll 1.
\]
The above implies that $\phi$ is bounded on the set $\mathscr{W}_\p$.
Similarly, using the Taylor expansion for $\nabla\phi$ around $\mathbf{x}=\mathbf{x}_{j,\mathbf{k},\ell, \p}$,
i.e. 
\[
\nabla\phi(\mathbf{x})=H_{f_{\p}}(\mathbf{x}_{j,\mathbf{k},\ell, \p})(\mathbf{x}-\mathbf{x}_{j,\mathbf{k},\ell, \p})+O(\Vert\mathbf{x}-\mathbf{x}_{j,\mathbf{k},\ell, \p}\Vert_{2}^{2}),
\]
we can infer that 
\[
\Vert\nabla\phi(\mathbf{x})\Vert_{2}\ll1
\]
on $\mathscr{W}_{\p}$.
Furthermore, observe that any derivative of the phase function $\phi$, of order $2$ or higher,
is independent of $\ell$. Thus, by again using the fact that $\mathscr{W}_\p$ is contained in a compact set bounded away from the origin, 
we can conclude that $\vert \phi^{(\boldsymbol{\alpha})}(\mathbf{x})\vert=\vert f^{(\boldsymbol{\alpha})}(\mathbf{x})\vert\ll_{\boldsymbol{\alpha}}1$
for any $\mathbf{x}\in\mathscr{W}_\p$. 
This verifies that the derivatives of $\phi$ are bounded in the set $\sW_{\p}$. The same is true for the derivatives of $\rho^{[\p]}$ (with the implicit constant again depending on $\rho^{[\p]}$). 
To show that the expression $\frac{\Vert\mathbf{x}-\xlkj\Vert_2}{\Vert\nabla\phi(\mathbf{x})\Vert_2}$ has a uniform upper bound for $\bx\in \mathscr{W}_{\p}$, we argue exactly as in the proof of Lemma \ref{lem: intermediate regime}, using the Taylor expansion of $\phi$ around $\xlkj$ and to obtain
$$\Vert \nabla \phi(\bx) - H_{f_{\p}}(\xlkj)\cdot(\bx-\xlkj) \Vert_2 
\leq C\Lambda_2\|\bx-\xlkj\|_2^2.$$ Here $C>0$ is an absolute constant.    
Dividing both sides by $\|\bx-\xlkj\|_2$,  using the triangle inequality and by choosing $\epsilon_{\bx_1}$ to be small enough, 
so that $$\|\bx-\xlkj\|_2< C^{-1}\Lambda_2^{-1}\Lambda_1^{-1}/100$$ for all $\bx\in\mathscr{W}_{\p}$, 
we can conclude that 
$$\frac{\|\nabla \phi(\bx)\|_2}{\|\bx-\xlkj\|_2}\geq \frac{99}{100}\Lambda_1^{-1}.$$
This implies the desired upper bound on the reciprocal of the term on the left.

Thus Theorem \ref{thm: stationary phase higher order terms} is applicable
with $u:=\rho^{[\p]}$, $\lambda:=2^{-d\ell}qj$,
and $t$ as in \eqref{def: t}. Hence,
\[
I(\lambda,u,\phi,t):=\frac{e(\lambda\phi(\mathbf{x}_{j,\mathbf{k},\ell, \p})+\sigma)}{(\vert\det(\lambda H_{\phi}(\mathbf{x}_{j,\mathbf{k},\ell, \p}))\vert)^{\frac{1}{2}}}\sum_{0\leq \tau\leq t}\frac{(\mathfrak{D}_{\tau,\phi}\rho^{[\p]})(\mathbf{x}_{j,\mathbf{k},\ell, \p})}{\lambda^{\tau}},
\]
satisfies
\[
\vert I(\ell,\p,q,j,\mathbf{k}) - I(\lambda,u,\phi,t) \vert
\leq C\Vert \rho^{[\p]}\Vert_{C^{2t}(\mathbb{R}^{n-1})}
\lambda^{-t}.
\]
Here the implied constant $C$ depends on $\Vert\phi\Vert_{C^{3t+1}(\mathscr{W_{\p}})}$ and the upper bound on 
$\frac{\Vert\mathbf{x}-\xlkj\Vert_2}
{\Vert\nabla\phi(\mathbf{x})\Vert_2}.$
Since $\lambda>Q^{\varepsilon}$ and $t$ was chosen as in \eqref{def: t},
we conclude the bound 
\begin{equation*}
    I(\ell,\p,q,j,\mathbf{k})-I(\lambda,u,\phi,t)\ll Q^{-10(n-1)}
\end{equation*}
where the implied constant is independent of $\ell$
and $j$. 
Recalling  \eqref{eq: def xlkj}, we note that
\begin{equation}
    \label{eqn: st rhostar sub}
    (\mathfrak{D}_{\tau,\phi}\rho^{[\p]})(\mathbf{x}_{j,\mathbf{k},\ell, \p})=(\mathfrak{D}_{\tau,\phi}\rho^{[\p]})((\nabla f_\p)^{-1}(2^{\ell(d-\p)}\mathbf{k}/j)).
\end{equation}
Furthermore, as $(\nabla f_\p)^{-1}=\nabla f_{d-\p}$ is homogeneous of degree $\frac{\p}{d-\p}$ and $f_\p$ of degree $\frac{d}{\p}$, the phase function (\ref{def: critical points phase function-1})
at the critical point simplifies to 
\begin{align*}
\phi(\mathbf{x}_{j,\mathbf{k},\ell, \p}) &
=f_{\p}(\mathbf{x}_{j,\mathbf{k},\ell, \p})-2^{(d-\p)\ell}\left\langle \mathbf{x}_{j,\mathbf{k},\ell, \p},
\mathbf{k}/j\right\rangle \\
 & =f_{\p}((\nabla f_{\p})^{-1}(2^{(d-\p)\ell}
 \mathbf{k}/j))-2^{(d-\p)\ell}\langle 
 (\nabla f_{\p})^{-1}(2^{(d-\p)\ell}\mathbf{k}/j),
 \mathbf{k}/j\rangle \\
 & =f_{\p}(2^{\ell\p}(\nabla f_{\p})^{-1}(\mathbf{k}/j))
 - 2^{(d-\p)\ell}\left\langle 2^{\ell\p}(\nabla f_{\p})^{-1}(\mathbf{k}/j),\mathbf{k}/j\right\rangle \\
 & =2^{\ell d}\left(f_{\p}((\nabla f_{\p})^{-1}
 (\mathbf{k}/j))-\left\langle 
 (\nabla f_{\p})^{-1}(\mathbf{k}/j),
 \mathbf{k}/j\right\rangle \right)=
 -2^{\ell d}f_{d-\p}(\bk/j).
\end{align*}
Thus 
\begin{equation}
    \label{eqn: st phase star sub}
    \lambda\phi^{j,\bk, \ell, \p}(\mathbf{x}_{j,\mathbf{k},\ell, \p})=2^{-\ell d}jq2^{\ell d}\tilde{f}_{\p}(\mathbf{k}/j)=jq f_{d-\p}(\mathbf{k}/j).
\end{equation}
Substituting \eqref{eqn: st rhostar sub}, 
and \eqref{eqn: st phase star sub} 
into the expression for $I(\lambda, u,\phi,t)$ gives the desired conclusion.
\end{proof}

\section{Enveloping and The Duality Principle}\label{sec: envelope}
Lemma \ref{lem: N1 contribution after stationary phase} offers a passage between the ``flat'' manifold (with defining function of degree $d$) and the dual ``rough'' manifold with defining function of degree $\frac{d}{d-1}$. 
It will turn out, as one might expect, 
that the main contribution to 
$ N^{\rho}_1(\delta,Q, \ell, \p, \r)$
comes from the complex weight function 
arising from choosing
$\tau=0$ in the 
innermost summation of 
\eqref{eq: N1 stationary phase expansion}.
For $\tau \geq 1$, the terms
are dampened by a negative power of 
$2^{-d\ell} q j > Q^\varepsilon$.

The issue though is that the amplitude 
functions $\mathfrak{D}_{\tau,\phi}\rho^{[\p]}$,
arising from the multi-term 
stationary phase expansion in 
Lemma \ref{lem: N1 contribution after stationary phase}, and as defined in 
\eqref{def: differential operator with phase function dependence}, cannot be interpreted as weights on this dual manifold (associated to $f_{d-\p}$) since they are not positive 
as soon as $\tau>0$. 
However, observe that 
the coefficients of 
the operators $\mathfrak{D}_{\tau,\phi}$ 
are bounded, 
since the entries 
of the Hessian $H_\phi(\xlkj)$, as well as all 
the derivatives (including of order zero) 
of the function 
\[g_{\xlkj}(\mathbf{x}):=\phi(x)-\phi(\xlkj)-\frac{1}{2}\langle H_\phi(\xlkj)(\mathbf{x}-\xlkj), \mathbf{x}-\xlkj\rangle,\]
depend only on derivatives of $\phi$ of order two or higher, which are all bounded from above and below in the support of $\rho^{[\p]}$; in particular, all the coefficients are independent of $\ell$. Thus, the functions $\mathfrak{D}_{\tau,\phi}\rho^{[\p]}$ can be majorised by a \emph{positive}, smooth envelope $\zeta_0^{[\p]}$ satisfying
\[
\vert\sum_{\tau=1}^t(\mathfrak{D}_{\tau,\phi}\rho^{[\p]})(\mathbf{x})\vert\leq\zeta_0^{[\p]}(\mathbf{x})
\]
for all $\mathbf{x}\in\mathbb{R}^{n-1}$ (see Lemma \ref{lem: env} for a precise construction). This procedure will be undertaken at every step of the forthcoming induction argument, leading us to define a finite sequence of envelopes $\{P_\kappa^{[\p]}\}_{\kappa=1}^t$, with an intrinsic order (depending on the step at which each envelope is introduced) and increasing width of support. At the same time, we need to make sure that each of these envelopes is supported in the set $\mathscr{W}_{\1}$, or $\mathscr{W}_{\d1}$, as defined in equations \eqref{eq W1def} and \eqref{eq Wd1def}.
(These are the sets where the derivatives of $f_\p$ are bounded from above, the Hessian $H_{f_\p}$ is invertible and $\nabla f_{\p}$ is invertible as a map.)
This requires us to maintain strict control over the rate at which each envelope decreases to zero at the periphery of its support, at the cost of a blow up in the size of its derivatives. However, as we shall see, we only require a finite number of envelopes since our induction will involve at most $O(\varepsilon^{-1})$ many steps (in fact $O(\log (\varepsilon^{-1})$ for $n\geq 4$). Thus we ultimately only lose a constant depending on $\varepsilon$.
\subsection{The Enveloping Argument}\label{subsec: envelope} In this subsection, we envelope (`majorise') the amplitude functions $\mathfrak{D}_{\tau,\phi}\rho^{[\p]}$,
arising from the multi-term stationary phase expansion in Lemma \ref{lem: N1 contribution after stationary phase}, by smooth
but \emph{positive} replacements which can act as weights. We shall do so in a manner that
reduces the analytic complexity --- estimating several bunches of
amplitude functions arising through the differential operators in
Theorem \ref{thm: stationary phase higher order terms}, by the very
same majorant. Doing so inductively, and keeping track of the decay
attached to each higher order term in the stationary phase expansion
allows us to envelope efficiently. As already mentioned, we shall need to envelope
only $O_{\varepsilon}(1)$-many times, as we may discard any term
with too large decay since it will turn out to be negligible for the
final count. The following lemma is a key tool in this endeavour. 

\begin{lem}[Enveloping lemma]
\label{lem: env}
Let $g:\bRn \rightarrow \bR$ be a smooth function which is supported 
in the ball $ \sB (\bc,r)$.
Fix an integer $M\geq 1$ and any $\eta \in (0,r)$.
Then there exists a non-negative function 
$\cE_{g}^{M,\eta}: \bRn \rightarrow \bR_{\geq 0}$
having the following two properties. \\
Firstly, for any multi-index $\boldsymbol{\alpha}\in \bZ^{n-1}_{\geq 0}$
with $\vert \boldsymbol{\alpha} \vert \leq M$
the function $\cE_{g}^{M,\eta}$ is a point-wise majorant to 
$g^{(\boldsymbol{\alpha})}$, i.e.
\begin{equation}\label{eq: point wise majorant}
    \vert g^{(\boldsymbol{\alpha})}(\bx)\vert \leq \cE_{g}^{M,\eta}(\bx)
    \,\,\, \mathrm{for\,any\,}\bx \in \bR^{n-1}.
\end{equation}
Secondly, $\cE_{g}^{M,\eta}$ has only marginally larger support than $g$
in the sense that 
\begin{equation}\label{eq: enlarged support}
\supp{(\cE_{g}^{M,\eta})} \subseteq 
\sB (\bc,r + \eta).   
\end{equation}
\end{lem}

\begin{proof}
The statement is geometrically obvious. 
For the convenience of the reader,
we include the details. 
Recalling \eqref{def: C^k norm}, put
$Y:= \Vert g \Vert_{C^M(\sB (\bc,r))} $.
For simplicity of the exposition we write 
$\cE$ in place of $\cE_{g}^{M,\eta}$.
Fix an even smooth function
$B: \mathbb{R}^{n-1} \rightarrow \mathbb{R}_{\geq0}$
such that $\Vert B\Vert_{L^{1}(\mathbb{R}^{n-1})}=1$.
Further, we stipulate that $B$ is supported in the unit ball
$\sB (\bzero,1)$.
Given $\vartheta>0$, we define 
$B_{\vartheta}(\mathbf{x}):=\vartheta^{-(n-1)}B(\vartheta^{-1}\mathbf{x})$.
Denote by $*$ the convolution of two square integrable functions
$g_1,g_2: \bR^{n-1} \rightarrow \bR$, that is 
$$
(g_1*g_2)(\bx) := \int_{\bRn} g_1(\bx - \by) g_2(\by) \rd \by.
$$
We claim the smooth function 
$$
\cE := Y \cdot 
(\mathds{1}_{\sB (\bc,r + \frac{\eta}{2})} * B_{\vartheta_0})
\,\,\, \mathrm{where} \,\, 
\vartheta_0:= \frac{\eta}{10(n-1)} 
$$
satisfies \eqref{eq: point wise majorant} 
and \eqref{eq: enlarged support}. We begin to verify the former.
Notice
$ \vert g^{(\boldsymbol{\alpha})}(\bx)\vert  \leq Y$
for any $\bx \in \supp (g) \subseteq \sB (\bc,r+ \eta)$.
Hence \eqref{eq: point wise majorant} follows if we show
$ \mathds{1}_{\sB (\bc,r + \frac{\eta}{2})} 
* B_{\vartheta_0}$ is equal to one
on the ball $\sB (\bc,r + \frac{\eta}{3})$.
We begin by writing
\begin{align*}
\mathds{1}_{\sB (\bc,r + \frac{\eta}{2})} * B_{\vartheta_0} (\bx)
 = \int_{\sB (\bc,r + \frac{\eta}{2})} 
B_{\vartheta_0} (\bx-\by) \rd \by 
= \int_{\sB (\bc,r + \frac{\eta}{2}) - \bx} 
B_{\vartheta_0} (-\by) \rd \by .
\end{align*} Observe 
that $B_{\vartheta_0}$ is supported inside 
the ball $\sB (\bzero,\vartheta_0)$. 
Further, for any $\bx \in \sB (\bc, r + \frac{\eta}{3})$, the triangle inequality implies
$\sB (\bzero,\vartheta_0)  
\subseteq \sB (\bc,r + \frac{\eta}{2}) - \bx $.
By using the fact that $B$ is non-negative, even and $L^1$-normalised, we infer 
$$
\mathds{1}_{\sB (\bc,r + \frac{\eta}{2})} * B_{\vartheta_0} (\bx) = \int_{\sB (\bzero,\vartheta_0)} 
B_{\vartheta_0} (-\by) \rd \by =1 
\quad \mathrm{for}\quad 
\bx \in \sB (\bc, r + \frac{\eta}{3}).
$$
Finally, the property \eqref{eq: enlarged support} follows
from noticing that 
$\sB (\bc,r + \frac{\eta}{2}) 
- \bx \cap \sB (\bzero,\vartheta_0) = \emptyset$
for any $\bx$ outside $\sB (\bc,r + \eta)$
which, in turn, forces 
$\mathds{1}_{\sB (\bc,r + \frac{\eta}{2})} * B_{\vartheta_0} (\bx)$ to vanish.
\end{proof}
Recall the quantity $ t$ from \eqref{def: t}, and the operators $\mathfrak{D}_{\tau,\phi}$
arising from the multi-term stationary phase expansion in Lemma \ref{lem: N1 contribution after stationary phase}, and as defined in \eqref{def: differential operator with phase function dependence}. The coefficients of these operators are uniformly bounded in the set $\mathscr{W}_{\p}$. 
Let $M$ be the order of the highest derivative which occurs in the expression $\sum_{\tau=0}^t\mathfrak{D}_{\tau, \phi}$. Let 
\begin{equation*}
    \varrho_0:=\varrho
\end{equation*}
be the non-negative smooth weight supported in the ball $\mathscr{B}(\bx_{\1}, \tfrac{\epsilon_{\bx_1}}4)$ which was introduced in \S\ref{subsec local}, satisfying properties \eqref{eq varrho prop} and \eqref{eq varrho pou}.
We define a finite sequence of envelopes $\{\varrho_\kappa\}_{1\leq \kappa\leq K(\varepsilon)}$ recursively in the following way:
\begin{align}
    \label{eq varrho}
    \varrho_\kappa:=c_\kappa\mathcal{E}^{M, 2^{-\kappa-4}\epsilon_{\bx_1}}_{\varrho_{\kappa-1}},
\end{align}
where $c_\kappa$ is a positive constant (depending only on 
$\varrho_0$, $\kappa$, $\varepsilon$, and $n$), chosen so that 
\[\sum_{\tau=1}^t|\mathfrak{D}_{\tau, \phi}\varrho_{\kappa-1}|\leq \varrho_\kappa.\]
Since
$$\frac{\epsilon_{\bx_1}}{4}+\sum_{i=1}^{\kappa}2^{-i-4}{\epsilon_{\bx_1}}\leq \frac{\epsilon_{\bx_1}}{4}+2^{-4}{\epsilon_{\bx_1}}\sum_{i\geq 1}2^{-i}< \frac{\epsilon_{\bx_1}}{2},$$
it follows that $\varrho_{\kappa}$ is supported in the set $\mathscr{W}_{\1}=\mathscr{B}(\bx_1, \tfrac{\epsilon_{\bx_1}}{2})$ for each $\kappa\in \mathbb{N}$. Following our notation, we shall denote each such $\varrho_\kappa$ by $\varrho_\kappa^{[\1]}$.

Notice that when $\tau=0$, we have 
$\mathfrak{D}_{\tau,\phi}u(x)=u(x)$ for any amplitude function $u$. 
What this implies for our induction argument is 
that the contribution from all the envelopes up to step $\kappa$ 
will also be present in addition to new terms which arise at step $\kappa+1$. Keeping this in mind and to keep track of the contributions from the different envelopes at each step, for $\kappa\in\{0, 1, \ldots, K(\varepsilon)\}$ (where $K(\varepsilon)\in\mathbb{Z}_{\geq 1}$ will be fixed later), we define a sequence of weights supported in $\mathscr{W}_{\1}$ in the following way
\begin{equation}
\label{eq: def kappa weight}
P_\kappa^{[\1]}(\mathbf{x}):=\sum_{\mu=0}^{\min\{\kappa,t-1\}} Q^{-\mu\varepsilon}\binom{\kappa}{\mu}\varrho^{[\1]}_\mu(\mathbf{x}). 
\end{equation}

We shall also need weights supports on the dual domain $\mathscr{W}_{\d1}$, given by
\begin{equation}
\label{eq: def dual kappa weight}
P_\kappa^{[\d1]}(\mathbf{y}):=P_\kappa^{[\1]}((\nabla f_{\1})^{-1}(\mathbf{y}))=\sum_{\mu=0}^{\min\{\kappa,t-1\}} Q^{-\mu\varepsilon}\binom{\kappa}{\mu}\varrho^{[\1]}_\mu((\nabla f_{\1})^{-1}(\mathbf{y})). 
\end{equation}
Observe that the number of summands is bounded above by $t$ 
and we only need to keep track of terms of the form 
$Q^{-\kappa\varepsilon}\varrho_\kappa$ for 
$\kappa\in \{0,1,\ldots, t-1\}$. This is because as 
$Q^{-\kappa\varepsilon} \ll Q^{1-n}$ 
for $\kappa\geq t$, which is an acceptable error 
in view of \eqref{eq: N1 stationary phase expansion}. 
On the other hand, the size of the binomial coefficients 
is of the order of $\kappa !$. 

The following lemma establishes that the weight 
$P_{\kappa+1}$ dominates the sum of functions 
obtained when the operator 
$\sum_{\tau=0}^tQ^{-\tau\varepsilon}
\mathfrak{D}_{\tau, \phi}$ 
acts on $P_{\kappa}$. 
\begin{lem}
\label{lem: binomial coeff}
Let $K=K(\varepsilon)$ be a positive natural number and $\p\in\{\1, \d1\}$. If $\kappa\in \{0, 1, \ldots, K(\varepsilon)\}$, then
\[\sum_{\tau=0}^tQ^{-\tau\varepsilon}|\mathfrak{D}_{\tau, \phi}P_{\kappa}^{[\p]}|\leq P_{\kappa+1}^{[\p]}+O(Q^{1-n}).\]
\end{lem}
\begin{proof}
We fix $\p$ and suppress the dependence on it, 
referring to $P_{\kappa}^{[\p]}$ or $\varrho_{\mu}^{\p}$ 
by writing $P_{\kappa}$ or $\varrho_{\mu}$. We have
\begin{align*}
\sum_{\tau=0}^tQ^{-\tau\varepsilon}\Big|\mathfrak{D}_{\tau, \phi}P_{\kappa}\Big|&\leq \sum_{\mu=0}^{\min\{\kappa,t-1\}} Q^{-\mu\varepsilon}\binom{\kappa}{\mu}\sum_{\tau=0}^tQ^{-\tau\varepsilon}\Big|\mathfrak{D}_{\tau, \phi}\varrho_\mu\Big|\\
&\leq \sum_{\mu=0}^{\min\{\kappa,t-1\}} Q^{-\mu\varepsilon}\binom{\kappa}{\mu}(\varrho_\mu+Q^{-\varepsilon}\varrho_{\mu+1}).    
\end{align*}

Setting $\kappa'=\min\{\kappa,t-1\}$, we can express the rightmost expression in the following way
\begin{align*}
\sum_{\mu=0}^{\kappa'} Q^{-\mu\varepsilon}\binom{\kappa}{\mu}(\varrho_\mu+Q^{-\varepsilon}\varrho_{\mu+1})&=\varrho_0+\sum_{\mu=1}^{\kappa'}Q^{-\mu\varepsilon}\Big[\binom{\kappa}{\mu-1}+\binom{\kappa}{\mu}\Big]\varrho_{\mu}+Q^{-(\kappa'+1)\varepsilon}\binom{\kappa}{\kappa'}\varrho_{\kappa'+1}\\
&=\varrho_0+\sum_{\mu=1}^{\kappa'}Q^{-\mu\varepsilon}\binom{\kappa+1}{\mu}\varrho_{\mu}+Q^{-(\kappa'+1)\varepsilon}\binom{\kappa}{\kappa'}\varrho_{\kappa'+1}.
\end{align*}
Now we have
\begin{align*}
&\varrho_0+\sum_{\mu=1}^{\kappa'}Q^{-\mu\varepsilon}\binom{\kappa+1}{\mu}\varrho_{\mu}+Q^{-(\kappa'+1)\varepsilon}\binom{\kappa}{\kappa'}\varrho_{\kappa'+1}\\
&\leq \begin{cases}
\sum_{\mu=0}^{\kappa+1}Q^{-\mu\varepsilon}
\binom{\kappa+1}{\mu}\varrho_{\mu}, &\text{ when } \kappa'=\kappa,\\
\sum_{\mu=0}^{t-1}Q^{-\mu\varepsilon}\binom{\kappa+1}{\mu}
\varrho_{\mu}+Q^{-(n-1)}\binom{\kappa}{t-1}\varrho_{t},
\,\, &\text{ when }\kappa'=t-1.
\end{cases}
\end{align*}
Thus in both of the cases above, we conclude that
\[\sum_{\tau=0}^t 
Q^{-\tau\varepsilon}\Big| \mathfrak{D}_{\tau, \phi}P_{\kappa}\Big|\leq P_{\kappa+1}+O(Q^{1-n}).\]
\end{proof}

\subsection{Duality via an Uncertainty Principle}\label{subsec: duality} Now, we can link the right hand side of 
\eqref{eq: N1 stationary phase expansion}
to the number of rational points on the 
dual hypersurface. 
This is the content of the 
Duality Principle formulated in the next lemma.
Throughout this section, we fix $K:=K(\varepsilon)\in\mathbb{Z}_{\geq 1}$.
\begin{lem}[Duality Principle]\label{lem: N1 after stationary phase and geom sum}
Let $\p\in\{1, \d1\}$. 
Suppose $\kappa\in 
\{0, 1, \ldots, K(\varepsilon)-1\}$. 
If $A>1$, then
\begin{align*}
& N_{1}^{P_{\kappa}^{[\p]}}(\delta,Q, \ell, \p, \r)\\ 
& \ll_A
 2^{(\frac{d}{2}-\p) \ell (n-1)} 
 (1+ (2^r\delta)^{A})^{-1}
Q^{\frac{n+1}{2}} \delta 
2^{-r \frac{n-1}{2}}\cdot \Big(\sum_{0\leq 
i \leq  \frac{\log (Q/\delta)}{100}}  
\frac{\fN^{{P}^{[d-\p]}_{\kappa+1}}(2^{i+1}/Q, 
2^r,\ell, d-\p)}{2^{Ai}}
\\ 
& + 2^r Q^{1-n} \max(1, 2^{-\ell (d-\p)(n-1)}2^{r(n-1)})\Big) 
+ 2^{-\p \ell (n-1)} 
(1+ (2^r\delta)^{A})^{-1} Q^{-5(n-1)}\delta 2^{rn},
\end{align*}
uniformly in $Q\geq 1$, 
$\delta\in (0, 1/2)$, 
\ellpcond,
and $r\in\sG(\delta,Q,\ell)$.
\end{lem}
\begin{proof}
We rewrite the right hand side of 
\eqref{eq: N1 stationary phase expansion} as 
\begin{align}
\label{eq st phase first term}
E&:= 2^{-\ell\p(n-1)}\sum_{q,j\in \bZ}
\om(\frac{q}{Q}) \om(\frac{j}{2^r})
q^{n-1}\delta \widehat{b}(j\delta)
\sum_{\mathbf{k}\in\mathscr{K}_{1}(\ell,\p,j)}
\Bigg[
\sum_{0\leq \tau\leq t}
\frac{b_\kappa(\mathbf{k}/j,\tau,\p,\ell)}
{(2^{-d\ell}qj)^{\frac{n-1}{2}+\tau}} 
e(-jqf_{d-\p}(\mathbf{k}/j))
\\&+ O(Q^{-10(n-1)})
\Bigg],\nonumber
\end{align} where
\begin{equation}\label{def: wtildetau}
b_\kappa(\mathbf{k}/j,\tau,\p,\ell):=
\frac{ e(\sigma)(\mathfrak{D}_{\tau,\phi}P_\kappa^{[\p]})(\mathbf{x}_{j,\mathbf{k},\ell, \p})}{\vert\det(H_{f_{\p}}(\mathbf{x}_{j,\mathbf{k},\ell, \p}))
\vert^{\frac{1}{2}}}.
\end{equation}
Next, we consider the contribution to $E$
arising through the 
$O$-term in the rectangular brackets. 
To this end, we first notice that 
$\sK_1(\ell,\p,j)$ is contained in a cube in $\bRn$ 
whose side-length are of size
$2^{-\ell (d-\p)}j\asymp 2^{-\ell (d-\p)}2^{r}$. 
As a result,
\begin{equation}\label{eq size of K1}
\# \sK_1(\ell,\p,j) \ll (1+ 2^{-\ell (d-\p)}j)^{n-1} \ll 
\max(1, 2^{-\ell (d-\p)(n-1)}2^{r(n-1)}).
\end{equation}
Let
$$
E_0:=2^{-\p \ell (n-1)} 
\delta 2^r (1+ (2^r\delta)^{A})^{-1} Q^{-5(n-1)}2^{r(n-1)}
$$
for fixed $A>0$. 
We use the rapid decay of $\widehat{b}$
in the form 
$$
\widehat{b}(t) \ll_A \frac{1}{1+\vert t\vert^A}.
$$
Using \eqref{eq size of K1} in the slackened form
$
\# \sK_1(\ell,\p,j) \ll 2^{r(n-1)},
$
we conclude that 
$$
\sum_{q,j\in \bZ}
\om(\frac{q}{Q}) \om(\frac{j}{2^r})
q^{n-1}\delta \widehat{b}(j\delta)
\sum_{\mathbf{k}\in\mathscr{K}_{1}(\ell,\p,j)}
Q^{-10(n-1)}
\ll  2^{r(n-1)} 
\delta 2^r (1+ (2^r\delta)^{A})^{-1} Q^{-5(n-1)}.
$$
Therefore, $E$ equals
$$
O(E_0) +
2^{-\ell\p(n-1)}\sum_{q,j\in \bZ}
\om(\frac{q}{Q}) \om(\frac{j}{2^r})
q^{n-1}\delta \widehat{b}(j\delta)
\sum_{\substack{
\mathbf{k}\in\mathscr{K}_{1}(\ell,\p,j)\\
0\leq \tau\leq t}}
\frac{b_\kappa(\mathbf{k}/j,\tau,\p,\ell)}
{(2^{-d\ell}qj)^{\frac{n-1}{2}+\tau}} 
e(-jqf_{d-\p}(\mathbf{k}/j)).
$$
Fix $0\leq \tau \leq t$ and set
$g_\tau(x):= \om(x) x^{\frac{n-1}{2}-\tau}$. 
By using Lemma \ref{lem: smooth geometric summation}
we infer
\begin{align*}
\sum_{q\in \bZ}
\om(\frac{q}{Q}) 
q^{n-1}
\frac{e(-jqf_{d-\p}(\mathbf{k}/j))}
{q^{\frac{n-1}{2}+\tau}}
& =
Q^{\frac{n-1}{2}-\tau}
\sum_{q\in \bZ}
g_\tau(\frac{q}{Q})
e(-jqf_{d-\p}(\mathbf{k}/j))\\
& =
Q^{\frac{n-1}{2}-\tau}\left[Q\widehat{g_\tau}
(Q \Vert jf_{d-\p}(\mathbf{k}/j)\Vert)
+ O_A(\Vert g_\tau \Vert_{C^A} Q^{-A})\right]
\end{align*}
where we used in the last step that $\Vert -x \Vert=\Vert x\Vert$.
Since $t$ and $A$ are finite constants,
$\Vert g_\tau \Vert_{C^A}$ is uniformly
bounded for all $0\leq \tau \leq t$.
Take $A> 100 (n-1)/\varepsilon$. 
Summing the above relation over $j\asymp 2^r$ and $0\leq \tau\leq t$,
we conclude that 
\begin{equation}
\label{dec E1}
E=
 E_1 + O(E_0+E_2),
\end{equation}
where
\begin{align}
    \label{def E1}
    E_1:= &2^{-\ell\p(n-1)}
\sum_{j\in \bZ}
\sum_{\substack{0\leq \tau\leq t}}
2^{d\ell (\frac{n-1}{2}+\tau)}
Q^{\frac{n+1}{2}-\tau}
\om(\frac{j}{2^r})
\frac{\delta \widehat{b}(j\delta)}{j^{\frac{n-1}{2}+\tau}}
\\&\times\sum_{\mathbf{k}\in\mathscr{K}_{1}(\ell,\p,j)}
b_\kappa(\mathbf{k}/j,\tau,\p,\ell)
\widehat{g_\tau}(Q \Vert jf_{d-\p}(\mathbf{k}/j)\Vert),\nonumber
\end{align}
and 
\begin{align*}
E_2:=
2^{-\ell\p(n-1)}
\sum_{j\in \bZ}
\om(\frac{j}{2^r})
\sum_{\substack{0\leq \tau\leq t \\
\mathbf{k}\in\mathscr{K}_{1}(\ell,\p,j)}}
Q^{1-n}\cdot  
2^{d\ell (\frac{n-1}{2}+\tau)}
Q^{\frac{n+1}{2}-\tau}
\vert b_\kappa(\mathbf{k}/j,\tau,\p,\ell)\vert
\frac{\delta \vert 
\widehat{b}(j\delta)\vert }{j^{\frac{n-1}{2}+\tau}}.
\end{align*}
We first deal with $E_2$, and verify that
\begin{equation}\label{eq E2 estimate}
E_2 \ll U:=2^{(\frac{d}{2}-\p)\ell (n-1)}
\frac{\delta Q^{-\frac{n-3}{2}}2^{-r \frac{n-3}{2}}}
{1+ (2^r\delta)^{A}} 
\max(1, 2^{-\ell (d-\p)(n-1)}2^{r(n-1)}).
\end{equation}
Let $M:=
\max(1, 2^{-\ell (d-\p)(n-1)}2^{r(n-1)})$.
Upon using $
b_\kappa(\mathbf{k}/j,\tau,\p,\ell)=O(1)$,
we infer 
$$
E_2 \ll 
2^{-\ell\p(n-1)}
\sum_{j\in \bZ}
Q^{1-n}\sum_{\substack{0\leq \tau\leq t}}
2^{d\ell (\frac{n-1}{2}+\tau)}
Q^{\frac{n+1}{2}-\tau}
\om(\frac{j}{2^r})
\frac{\delta \vert \widehat{b}(j\delta)\vert }{2^{r\frac{n-1}{2}+\tau}}M.
$$
Further, bearing in mind that
$(2^{-d\ell} Q 2^{r})^{-\tau} \ll  Q^{-\varepsilon \tau}$,
it is clear that the regime $\tau=0$ 
contributes more than the regime $0<\tau \leq t$.
Hence 
$$ 
E_2 \ll
2^{-\ell\p(n-1)}
2^{d\ell (\frac{n-1}{2})}
Q^{1-n} Q^{\frac{n+1}{2}}
\sum_{j \in \bZ}
\frac{\delta \om(\frac{j}{2^r})2^{-r\frac{n-1}{2}}}{1+ (2^r\delta)^{A}}M.
$$
This implies \eqref{eq E2 estimate}.
Next, we turn our attention to $E_1$. 
By
$(2^{-d\ell} Q 2^{r})^{-\tau} \ll  
Q^{-\varepsilon \tau}$,
we obtain
\begin{align}\label{eq: bassa}
& E_1 \ll 
\\&2^{-\ell\p(n-1)}2^{d\ell (\frac{n-1}{2})}
Q^{\frac{n+1}{2}}
\sum_{j\in \bZ}
\om(\frac{j}{2^r})
\frac{\delta \vert \widehat{b} (j\delta) \vert}{2^{r(\frac{n-1}{2})}}
\sum_{\substack{\mathbf{k}\in\mathscr{K}_{1}(\ell,\p,j)\\
0\leq \tau\leq t}}Q^{-\varepsilon \tau}
\vert b_\kappa(\mathbf{k}/j,\tau,\p,\ell)
\widehat{g_\tau}(Q \Vert jf_{d-\p}(\mathbf{k}/j)\Vert)
\vert.\nonumber
\end{align}
We turn our attention to estimating the innermost sum
which is at most
$$
\max_{0\leq \tau \leq t}
\vert
\widehat{g_\tau}(Q \Vert jf_{d-\p}(\mathbf{k}/j)\Vert)
\vert
\cdot 
\sum_{0\leq \tau\leq t} Q^{-\varepsilon \tau}
\vert b_\kappa(\mathbf{k}/j,\tau,\p,\ell)
\vert.
$$
Since $\bx_{j,\bk,\ell,\p} \in \mathscr{W}_{\p}$, we can bound
$\vert \det H_{f_\p}(\bx_{j,\bk,\ell,\p}) \vert^{-1}$
from above (uniformly in $j,\ell$ and $p$).
We recall \eqref{def: wtildetau} and infer using
Lemma \ref{lem: binomial coeff} that 
\begin{align*}
\sum_{0\leq \tau \leq t}
Q^{-\varepsilon \tau} \vert 
b_\kappa(\mathbf{k}/j,\tau,\p,\ell)
\vert
&=
\frac{1}
{\vert \det(H_{f_{\p}}(\mathbf{x}_{j,\mathbf{k},\ell, \p}))
\vert^{\frac{1}{2}}}
\sum_{0\leq \tau \leq t}
Q^{-\varepsilon \tau}
\vert(\mathfrak{D}_{\tau,\phi}P_{\kappa}^{[\p]})
(\mathbf{x}_{j,\mathbf{k},\ell, \p})\vert
\\&\ll
P_{\kappa+1}^{[\p]}(\mathbf{x}_{j,\mathbf{k},\ell, \p})+ Q^{1-n}.    
\end{align*}

Using the definition of $\xlkj$ from \eqref{eq: def xlkj},
we get
\[
P_{\kappa+1}^{[\p]}(\mathbf{x}_{j,\mathbf{k},\ell, \p})=
{P}_{\kappa+1}^{[d-\p]}(2^{\ell(d-\p)}\mathbf{k}/j).
\]
The upshot is that the inner most sum in \eqref{eq: bassa} equals
\begin{align*}
& 
\sum_{0\leq \tau\leq t} Q^{-\varepsilon \tau}
\vert b_\kappa(\mathbf{k}/j,\tau,\p,\ell)
\widehat{g_\tau}(Q \Vert jf_{d-\p}(\mathbf{k}/j)\Vert)
\vert\\
& \ll
\max_{0\leq \tau \leq t}
\vert
\widehat{g_\tau}(Q \Vert jf_{d-\p}(\mathbf{k}/j)\Vert)
\vert 
\cdot
{P}_{\kappa+1}^{[d-\p]}(2^{\ell(d-\p)}\mathbf{k}/j)
+ Q^{1-n}.
\end{align*}
In view of \eqref{eq: bassa}, we have 
\begin{equation}
    \label{eq E1E3}
    E_1 \ll E_3
\end{equation}
where
\begin{align*}
&E_3:=2^{(\frac{d}{2}-\p) \ell (n-1)}
Q^{\frac{n+1}{2}}
\sum_{j\in \bZ}
\om(\frac{j}{2^r})
\frac{\delta (1+ (2^r\delta)^{A})^{-1}}{2^{r\frac{n-1}{2}}}
\\
&
\sum_{\mathbf{k}\in\mathscr{K}_{1}(\ell,\p,j)}
\Big( 
\max_{0\leq \tau \leq t}
\vert
\widehat{g_\tau}(Q \Vert jf_{d-\p}(\mathbf{k}/j)\Vert)
\vert 
\cdot
{P}_{\kappa+1}^{[d-\p]}(2^{\ell(d-\p)}\mathbf{k}/j)
+ Q^{1-n}
\Big).    
\end{align*}
To proceed, the further decomposition
\begin{equation}\label{eq E4 decomp}
E_3 =
2^{(\frac{d}{2}-\p) \ell (n-1)} 
Q^{\frac{n+1}{2}} 
(1+ (2^r\delta)^{A})^{-1}\delta 2^{-r \frac{n-1}{2}}
\big( E_4 + E_5 \big)
\end{equation}
is useful where
$$
E_4 := 
\sum_{j\in \bZ}
\om(\frac{j}{2^r})
\sum_{\mathbf{k}\in\mathscr{K}_{1}(\ell,\p,j)}
\max_{0\leq \tau \leq t}
\vert
\widehat{g_\tau}(Q \Vert jf_{d-\p}(\mathbf{k}/j)\Vert)
\vert 
{P}_{\kappa+1}^{[d-\p]}(2^{\ell(d-\p)}\mathbf{k}/j),
$$
and
$$
E_5:=
\sum_{j\in \bZ}
\om(\frac{j}{2^r})
\sum_{\mathbf{k}\in\mathscr{K}_{1}(\ell,\p,j)}
Q^{1-n}.
$$
We now turn to bounding $E_5$.
In view of \eqref{eq size of K1} we deduce 
\begin{equation} \label{eq: II case 1}
E_5 \ll 
\sum_{j\in \bZ} 
\om(\frac{j}{2^r})
Q^{1-n} (\#\sK_1(\ell,\p,j))
\ll 
2^r
Q^{1-n} 
\max(1, 2^{-\ell (d-\p)(n-1)}2^{r(n-1)}) .
\end{equation}
Thus, 
\begin{equation}\label{eq X interm uncert}
E_3\ll 
2^{(\frac{d}{2}-\p) \ell (n-1)} 
Q^{\frac{n+1}{2}} 
(1+ (2^r\delta)^{A})^{-1}\delta 2^{-r \frac{n-1}{2}}
E_4 + U
\end{equation}
where we recall that $U$ denotes the right hand
of \eqref{eq E2 estimate}.
It remains to analyse $E_4$.
To this end, we use the rapid decay of $ \widehat{g_\tau}$ 
to conclude
$$
\sum_{\mathbf{k}\in\mathscr{K}_{1}(\ell,\p,j)}
\max_{0\leq \tau \leq t}
\vert
\widehat{g_\tau}(Q \Vert jf_{d-\p}(\mathbf{k}/j)\Vert)
\vert 
\cdot
{P}_{\kappa+1}^{[d-\p]}(2^{\ell(d-\p)}\mathbf{k}/j)
\ll 
\sum_{\mathbf{k}\in \bZn}
\frac{
{P}_{\kappa+1}^{[d-\p]}(2^{\ell(d-\p)}\mathbf{k}/j)}{1+
( Q \Vert jf_{d-\p}(\mathbf{k}/j)\Vert) ^A}.
$$
Thus 
$$
E_4 \ll 
\sum_{(j,\bk) \in \bZ\times \bZn}
\om(\frac{j}{2^r})
\frac{{P}_{\kappa+1}^{[d-\p]}(2^{\ell(d-\p)}\mathbf{k}/j)}{1+
( Q \Vert jf_{d-\p}(\mathbf{k}/j)\Vert) ^A}.
$$
Next, we (approximately) decompose the sum on the right hand side 
in according to the size of $Q \Vert jf_{d-\p}(\mathbf{k}/j)\Vert$.
Let 
$$
E_{4,\mathrm{far}}:=
\sum_{\substack{(j,\bk) \in \bZ\times \bZn
\\ Q \Vert jf_{d-\p}(\mathbf{k}/j)\Vert >2^{\frac{\log (Q/\delta)}{100}}}}
\om(\frac{j}{2^r})
\frac{
{P}_{\kappa+1}^{[d-\p]}(2^{\ell(d-\p)}\mathbf{k}/j)}{1+
( Q \Vert jf_{d-\p}(\mathbf{k}/j)\Vert) ^A},
$$
and 
$$
E_{4,\mathrm{near}}:=
\sum_{\substack{(j,\bk) \in \bZ\times \bZn
\\ Q \Vert jf_{d-\p}(\mathbf{k}/j)\Vert \leq 2^{\frac{\log (Q/\delta)}{100}}}}
\om(\frac{j}{2^r})
\frac{
{P}_{\kappa+1}^{[d-\p]}(2^{\ell(d-\p)}\mathbf{k}/j)}{1+
( Q \Vert jf_{d-\p}(\mathbf{k}/j)\Vert) ^A},
$$
so that
$$
E_4 \ll E_{4,\mathrm{far}}+ E_{4,\mathrm{near}}.
$$ 
We estimate 
$$
E_{4,\mathrm{far}} \ll 
\sum_{\substack{(j,\bk) \in \bZ\times \bZn
}}
\om(\frac{j}{2^r})
\frac{
{P}_{\kappa+1}^{[d-\p]}(2^{\ell(d-\p)}\mathbf{k}/j)}{1+(Q/\delta)^A}
\ll 
\frac{\delta^{-n}}{(Q/\delta)^A}
\ll 1.
$$
Recalling \eqref{eq X interm uncert} and using the above, we get
\begin{equation}\label{eq X inter uncert 2}
E_3 \ll 
2^{(\frac{d}{2}-\p) \ell (n-1)} 
Q^{\frac{n+1}{2}} 
(1+ (2^r\delta)^{A})^{-1}\delta 2^{-r \frac{n-1}{2}}
E_{4,\mathrm{near}}+ U.
\end{equation}
To estimate $E_{4,\mathrm{near}}$ efficiently,
the remaining range of $j$ and $\mathbf{k}$
can be estimated via
\begin{align*}
E_{4,\mathrm{near}}
&\ll 
\sum_{\substack{(j,\bk) \in \bZ\times \bZn
\\ Q \Vert jf_{d-\p}(\mathbf{k}/j)\Vert < 1}}
\om(\frac{j}{2^r})
{P}_{\kappa+1}^{[d-\p]}(2^{\ell(d-\p)}\mathbf{k}/j)
\\&+ 
\sum_{0< i \leq  \frac{\log (Q/\delta)}{100}} 
\sum_{\substack{(j,\bk) \in \bZ\times \bZn
\\ Q \Vert jf_{d-\p}(\mathbf{k}/j)\Vert \in [2^i,2^{i+1})}}
\om(\frac{j}{2^r})
\frac{{P}_{\kappa+1}^{[d-\p]}(2^{\ell(d-\p)}\mathbf{k}/j)}{2^{Ai}}.
\end{align*}
For each $0< i \leq  \frac{\log (Q/\delta)}{100}$, 
we widen the interval $[2^i,2^{i+1})$
to $[0,2^{i+1})$. Thus, the second term above is 
at most 
$$ 
\sum_{0< i \leq  \frac{\log (Q/\delta)}{100}}  
\frac{\fN^{{P}_{\kappa+1}}(2^{i+1}/Q, 
2^r,\ell, d-\p)}{2^{Ai}}.
$$
The first term can be bounded by 
$\fN^{{P}_{\kappa+1}}(1/Q, 2^r,\ell, d-\p) $.
Plugging the above into \eqref{eq X inter uncert 2}, we obtain
\begin{equation}
\label{eq X inter uncert 3}
    E_3\ll 2^{(\frac{d}{2}-\p) \ell (n-1)}
Q^{\frac{n+1}{2}} \delta 
\frac{2^{-r \frac{n-1}{2}}}
{1+ (2^r\delta)^{A}}
\sum_{0\leq i \leq  \frac{\log (Q/\delta)}{100}}  
\frac{\fN^{{P}_{\kappa+1}}(2^{i+1}/Q, 
2^r,\ell, d-\p)}{2^{Ai}}
+ U.
\end{equation}
We thus have the chain of inequalities
\begin{align*}
E  &\stackrel{\eqref{dec E1}}{\ll}E_0 + E_1+E_2
\stackrel{\eqref{eq E2 estimate}}{\ll} E_0+ E_1+ U
\stackrel{\eqref{eq E1E3}}{\ll}
E_0+ E_3 + U
\\&\stackrel{\eqref{eq X inter uncert 3}}{\ll}
2^{(\frac{d}{2}-\p) \ell (n-1)}
Q^{\frac{n+1}{2}} \delta 
\frac{2^{-r \frac{n-1}{2}}}
{1+ (2^r\delta)^{A}}
\sum_{0\leq i \leq  \frac{\log (Q/\delta)}{100}}  
\frac{\fN^{{P}_{\kappa+1}}(2^{i+1}/Q, 
2^r,\ell, d-\p)}{2^{Ai}}
+ U+E_0.
\end{align*}
Plugging in the expression for $U$, we conclude that
$N_{1}^{P_{\kappa}}(\delta,Q, \ell, \p, \r)$
is no more than a constant times
\begin{align*}
2^{(\frac{d}{2}-\p) \ell (n-1)} &
Q^{\frac{n+1}{2}} \delta \frac{2^{-r \frac{n-1}{2}}}
{1+ (2^r\delta)^{A}}
\Bigg(\sum_{0\leq i \leq  
\frac{\log (Q/\delta)}{100}}  
\frac{\fN^{{P}_{\kappa+1}}(2^{i+1}/Q, 
2^r,\ell, d-\p)}{2^{Ai}}
+ 2^r Q^{1-n} M \Bigg) + \\
+E_0
\end{align*}
Recalling the definitions of 
$E_0$ and $M$ completes the proof.
\end{proof}

\subsection{Stationary Phase Weights}
\label{subsec strange terms}
As stated and already evident from Lemma \ref{lem: N1 after stationary phase and geom sum}, the stationary phase principle 
connects the rational point count around a hypersurface
to the same for its dual. Going deeper into this process to set up a bootstrapping argument (inspired from the one in \cite{Huang rational points} in the case of non-vanishing curvature),
we shall encounter
a weighted sum of 
pruning terms on the 
`dual manifold' (such terms were not present in \cite{Huang rational points}).

Although these terms,
at the first glance, 
seem cumbersome 
we shall see 
that they are actually
harmless. 
To this end, we need
to bound various dyadic 
summations and play
the $\ell$-variables 
out against the $Q$
and $\delta$ variables. 
This maneuver is somewhat 
delicate in several of the cases.
For later convenience,
we gather the required 
technicalities here. 
The reader might skip over 
the present bounds 
on a first reading and return 
at a later point.\\
\\
Our task is to bound 
\begin{equation}\label{def: Xpfrak}
\fX_\p(\delta,Q,\ell)
:=
2^{(\frac{d}{2}-\p) 
\ell (n-1)}
Q^{\frac{n+1}{2}} 
\sum_{(i,r)\in \sX} 
\frac{\fN^{{P}_{\kappa+1}}(2^{i+1}/Q, 
2^r,\ell, d-\p)}
{2^{iA}(1+ (2^r\delta)^{A})}
\delta 2^{-r \frac{n-1}{2}} 
\end{equation}
where 
$\sX:=\sX(Q,\ell, d-\p)$ is
the set of integer pairs 
$(i,r)\in [0,\frac{\log Q}{100}]\times [R_{-},R]$,
with $R_{-}$ as defined in \eqref{eq: R Rminus},
such that 
$$
\ell >
\min\left\{\frac{\log 
(Q2^{r(1+\varepsilon)-i})}
{d\log 2}, 
\frac{(1-\varepsilon)r}{d-\p}\right\}.
$$
\begin{prop} \label{prop: pruning and stationary phase}
Let $\fX_\p$
be the function defined in 
\eqref{def: Xpfrak}.
If $2 \leq d\leq n-1 $, then the estimates
\begin{align*}
\fX_\d1(\delta,Q,\ell)
& \ll 
2^{-\frac{d}{2}\ell(n-1)}
2^{\ell d}
(\delta 
Q^{n - \varepsilon\frac{n-1}{2}}+ 
Q^{\frac{n-1}{2}}
\delta^{-\frac{n-1}{2}}),\\ 
\fX_\1(\delta,Q,\ell)
& \ll 
2^{-\ell(n-1)} 2^{\ell d}
(\delta 
Q^{n -
\varepsilon\frac{n-1}{2}}+ 
Q^{\frac{n-1}{2}}
\delta^{-\frac{n-1}{2}})
\end{align*}
hold true for all $\delta\in (0, 1/2), Q\geq 1$ and all $0\leq \ell \leq \Lc(\delta, Q)$, with the implied constants being uniform in these parameters.
\end{prop}
To prove Proposition \ref{prop: pruning and stationary phase}, 
we use Lemma \ref{lem: tail terms}
to bound $ \fN^{{P}_{\kappa+1}}(2^{i+1}/Q, 
2^r,\ell, d-\p)$.
Thus $ \fX_\p$
naturally decomposes 
into two different quantities.
To define these, we let 
$$
\cT_{1,\p}(\delta',Q')
:=\Big(\frac{\delta'}
{Q'}\Big)^{\frac{\p(n-1)}{d}}
(Q')^n,
\quad
\mathrm{and}
\quad 
\cT_{2,\p}(\delta',Q'):= (Q')^{1+(n-1)\varepsilon}.
$$
Define $\sX_1:=\sX_1(\delta, Q, \ell)$ to be the subset of 
$(i,r) \in \sX$ for which
\begin{equation}
\label{eq: strange case}
    \frac{\log (Q2^{-i+r(1+\varepsilon)})}{d\log 2} <  \ell.
\end{equation}
Furthermore, put
$\sX_2 := \sX\setminus \sX_1$. Then $\sX_2:=\sX_2(\delta, Q, \ell, d-\p)$ is the subset of 
$(i,r) \in \sX$ for which
\begin{equation}
\label{eq: real strange case}
    (1-\varepsilon)\frac{r}{d-\p}< \ell\leq \frac{\log (Q2^{-i+r(1+\varepsilon)})}{d\log 2}.
\end{equation}
Let $A>1$ be a constant
which is large (in terms
of $d$ and $n$). 
For $m=1,2$, let
$$
X_{m,\p} (\delta, Q,\ell)
:= \sum_{(i,r)\in \sX_m} 
\frac{2^{(\frac{d}{2}-\p) 
\ell (n-1)} }
{2^{iA}(1+ (2^r\delta)^{A})}
Q^{\frac{n+1}{2}} \delta 2^{-r \frac{n-1}{2}}
\cT_{m,d-\p}(2^{i+1}/Q,2^r).
$$
Notice that 
Lemma \ref{lem: tail terms} tells us 
\begin{equation}\label{eq: decomp Xfrak}
    \fX_\p(\delta,Q,\ell)
    \ll X_{1,\p}(\delta,Q,\ell)
    + X_{2,\p}(\delta,Q,\ell).
\end{equation}
The following simple lemma 
is useful to make the 
strong uncertainty principle 
precise in our context.
\begin{lem}
Let $v>0$, $\delta\in (0,1)$, 
and $A>0$ be large. 
Suppose $R$ is as in 
\eqref{def: R}. Then 
\begin{equation}\label{eq: uncertaintly prin}
\sum_{0\leq r\leq R}
\frac{2^{vr}}{1+(2^{r}\delta)^A}
\ll \delta^{-v}.
\end{equation}  
\end{lem}
\begin{proof}
Clearly,
$$
\sum_{0\leq r\leq \frac{\log (1/\delta)}{\log 2}}
\frac{2^{vr}}{1+(2^{r}\delta)^A}
\leq \sum_{0\leq r\leq \frac{\log (1/\delta)}{\log 2}}
2^{vr}
\ll \delta^{-v}.
$$
The tail term can be bounded by
$$
\sum_{\frac{\log (1/\delta)}{\log 2}
< r \leq \frac{\log (1/\delta)}{\log 2} +
\varepsilon \log Q}
\frac{2^{vr}}{1+(2^{r}\delta)^A}
\ll
\sum_{\frac{\log (1/\delta)}{\log 2}
< r \leq \frac{\log (1/\delta)}{\log 2} +
0.01 \log Q}
\frac{2^{(v-A)r}}{\delta^A}
$$
and an index shift reveals that the right hand side is equal to
$$
\sum_{0< r \leq 
0.01 \log Q}
\frac{2^{(v-A)(r+\frac{\log (1/\delta)}{\log 2})}}{\delta^A}
=
\delta^{-v} \sum_{0< r \leq 
0.01 \log Q}
2^{(v-A)r}
\ll \delta^{-v}.
$$
Thus \eqref{eq: uncertaintly prin} follows. 
\end{proof}
The following lemma 
contains the bulk 
of work to prove Proposition
\ref{prop: pruning and stationary phase}.
\begin{lem}\label{lem: Xfrac decompo bound}
Let $\p\in\{1, \d1\}$, $Q\geq 1$, $\delta\in (0, 1/2)$,
$d\geq2$, and \ellpdcond. Then
\begin{equation}\label{eq: X1 bound}
    X_{1,\p}(\delta, Q,\ell) \ll
2^{-\frac{d}{2}\ell (n-1)} 
\delta^{-\frac{n-1}{2}},
\end{equation}
and
\begin{equation}\label{eq: X2 bound}
    X_{2,\p}
    (\delta, Q,\ell)\ll
2^{-\p
\ell (n-1)}
2^{\ell d}
\delta 
Q^{n - \varepsilon\frac{n-1}{2}}.
\end{equation}
\end{lem}
\begin{proof}
First, we focus on bounding
$X_{1,\p}(\delta, Q,\ell)$ from above.
Due to 
\eqref{eq: strange case},
we know
$2^i>2^{-\ell d} 2^r Q$.
Since $r\in \sG(\delta, Q, \ell)$,
by \eqref{def: good set} we are guaranteed that
$2^i > Q^\varepsilon$.
Therefore, 
\begin{align*}
&X_{1,\p}(\delta, Q,\ell)\\
& \ll 
 \sum_{(i,r)\in [0,\frac{\log Q}{100}]\times [R_{-},R]} 
\frac{
2^{(\frac{d}{2}-\p) 
\ell (n-1)} }
{Q^{\varepsilon A}
(1+ (2^r\delta)^{A})}
Q^{\frac{n+1}{2}} \delta 2^{-r \frac{n-1}{2}}
\cT_{1,d-\p}(2^i /Q,2^r)\\
& \ll 
2^{(\frac{d}{2}-\p)
\ell (n-1)} 
Q^{\frac{n+1}{2}} 
\delta 
 \sum_{(i,r)\in [0,\frac{\log Q}{100}]\times [R_{-},R]} 
\frac{2^{-r \frac{n-1}{2}}}
{Q^{\varepsilon A}
(1+ (2^r\delta)^{A})}
\Big(\frac{2^{i-r}}
{Q}\Big)^{\frac{(d-\p)(n-1)}{d}}
2^{rn}\\
& \ll 
2^{(\frac{d}{2}-\p)
\ell (n-1)} 
Q^{\frac{n+1}{2}
-\varepsilon A 
-\frac{(d-\p)(n-1)}{d}} 
\delta 
 \sum_{(i,r)\in 
 [0,\frac{\log Q}{100}]\times 
 [R_{-},R]} 
\frac{2^{r \big( 
n- \frac{n-1}{2} - 
\frac{(d-\p) (n-1)}{d}\big)}}
{1+ (2^r\delta)^{A}}
2^{i\frac{(d-\p)(n-1)}{d}}.
\end{align*}
Next, observe $100 \log 2\geq 30$ and hence
$ 2^{\frac{\log Q}{100}}\leq 
Q^\frac{1}{30}$. Thus, 
$$
\sum_{0\leq i \leq 
\frac{\log Q}{100}} 
2^{i\frac{(d-\p)(n-1)}{d}}
\ll 
Q^{\frac{(d-\p)(n-1)}{30d}}.
$$
Now let $\p = \d1$. Then
$ 
n- \frac{n-1}{2} - 
\frac{\1 (n-1)}{d} >0$,
since $d\geq 2$.
Thus, we can apply
\eqref{eq: uncertaintly prin} 
to conclude 
\[
 \sum_{(i,r)\in 
 [0,\frac{\log Q}{100}]\times 
 [R_{-},R]} 
\frac{2^{r (n- \frac{n-1}{2} - 
\frac{n-1}{d})}}
{1+ (2^r\delta)^{A}}
2^{i\frac{n-1}{d}}
\ll 
\delta^{-( 
n- \frac{n-1}{2} - 
\frac{n-1}{d})}
Q^{\frac{n-1}{30d}}.
\]
The upshot is that
\begin{align*}
X_{1, \d1} & \ll 
2^{-(\frac{d}{2}-\1) 
\ell (n-1)} 
Q^{\frac{n+1}{2}
-\varepsilon A 
-\frac{n-1}{d}} 
\delta 
\cdot
\delta^{-( 
n- \frac{n-1}{2} - 
\frac{n-1}{d})}
Q^{\frac{n-1}{30d}} \\
& =
2^{-(\frac{d}{2}-\1)
\ell (n-1)} 
Q^{\frac{n-1}{30d}
-\varepsilon A} 
\delta 
\Big( \frac{Q}{\delta} 
\Big)^{\frac{n+1}{2} - 
\frac{n-1}{d}}.
\end{align*}
Upon using $2^{\ell (n-1)} \leq \left(\delta^{-1}{Q}\right)^{\frac{n-1}{d}}$ (recall \eqref{eq: 2^L}),
we have
$$
X_{1, \d1} \ll 
2^{-\frac{d}{2}\ell (n-1)} 
Q^{\frac{n-1}{30d}
-\varepsilon A} 
\delta 
\Big( \frac{Q}{\delta} 
\Big)^{\frac{n+1}{2}}
= 
2^{-\frac{d}{2}\ell (n-1)} 
Q^{\frac{n-1}{30d}
+ \frac{n+1}{2}
-\varepsilon A} 
\delta^{-\frac{n-1}{2}}.
$$
Because $A$ is assumed to be large,
this readily implies the estimate
\eqref{eq: X1 bound} 
if $\p = \d1$.\\
\\
Now let $\p = \1$.
Presently, 
$$ 
n- \frac{n-1}{2} - 
\frac{(d-\1) (n-1)}{d}= 
n- \frac{n-1}{2} - 
(1-\frac{1}{d})(n-1)
=(n-1) ( \frac{1}{d}- 
\frac{1}{2}) + 1\leq 1.
$$ 
So, 
$$
 \sum_{(i,r)\in 
 [0,\frac{\log Q}{100}]\times 
 [R_{-},R]} 
\frac{2^{r ( 
n- \frac{n-1}{2} - 
\frac{(\d1) (n-1)}{d})}}
{1+ (2^r\delta)^{A}}
2^{i\frac{(\d1)(n-1)}{d}}
\ll 
\delta^{-1}
Q^{\frac{(\d1)(n-1)}{30d}}.
$$
Since $\ell\leq \fL_{\1}(\delta, Q)$ (in particular, recall \eqref{eq: 2^Lp} for $\p=\1$), we have
$2^{\ell} \leq Q^{1-\varepsilon}$. Consequently,
\begin{align*}
X_{1, \1} 
&\ll2^{\frac{d}{2} \ell (n-1)} 2^{-\ell (n-1)}
Q^{\frac{n+1}{2}-\varepsilon A -\frac{(\d1)(n-1)}{d}} 
\delta \cdot \delta^{-1} Q^{\frac{(\d1)(n-1)}{30d}}\\
& \ll 
2^{-\frac{d}{2} 
\ell (n-1)} 2^{(\d1)\ell(n-1)}
Q^{\frac{n+1}{2}
-\varepsilon A 
-\frac{(\d1)(n-1)}{d}}  
Q^{\frac{(\d1)(n-1)}
{30d}}
\\
&\ll 2^{-\frac{d}{2} 
\ell (n-1)} Q^{(1-\varepsilon)(\d1)(n-1)}
Q^{\frac{n+1}{2}
-\varepsilon A 
-\frac{(\d1)(n-1)}{2d}} 
\\
& \ll 
2^{-\frac{d}{2} \ell (n-1)} 
Q^{ 2(\d1)(n+1)
-\varepsilon A}.
\end{align*}
Again as $A$ is assumed to be large, the above implies the estimate
\eqref{eq: X1 bound} for $\p = \1$. 
\\
Next, we shall bound $X_{2, \p}$.
We write 
$$
2^{(\frac{d}{2}-\p) 
\ell (n-1)}
Q^{\frac{n+1}{2}}
2^{-r \frac{n-1}{2}}
= 
2^{-\p
\ell (n-1)}
Q^{n}
(2^{-d\ell}
Q
2^{r})^{-\frac{n-1}{2}}.
$$
Because $r\geq R_{-}$ (recall \eqref{eq: R Rminus}),
we have $2^{-{d} \ell}
Q2^{r} \geq Q^{\varepsilon}$.
Hence, 
$$
X_{2, \p}\ll 
2^{-\p
\ell (n-1)}
Q^{n - \varepsilon\frac{n-1}{2}}
\delta
 \sum_{(i,r)\in 
 [0,\frac{\log Q}{100}]
 \times [R_{-},R]} 
\frac{\cT_{2,d-\p}(2^i /Q,2^r)}
{2^{iA}
(1+ (2^r\delta)^{A})}.
$$
Further, the summation 
on the right hand side equals
$$
\sum_{(i,r)\in 
 [0,\frac{\log Q}{100}]
 \times [R_{-},R]} 
\frac{2^{r(1+(n-1)\varepsilon)}}
{2^{iA}
(1+ (2^r\delta)^{A})}
\ll 
2^{\ell d}
\sum_{R_{-}\leq r \leq R} 
2^{-\ell d}
\frac{2^{r(1+(n-1)\varepsilon)}}
{1+ (2^r\delta)^{A}}.
$$
Due to \eqref{eq: real strange case} being true, we know 
$2^{\ell (d-\p)} 
>2^{(1-\varepsilon)r}$. 
Therefore, we have
$$
\sum_{(i,r)\in 
 [0,\frac{\log Q}{100}]
 \times [R_{-},R]} 
\frac{2^{r(1+(n-1)\varepsilon)}}
{2^{iA}
(1+ (2^r\delta)^{A})}
\ll 
2^{\ell d}
\sum_{R_{-}\leq r \leq R} 
\frac{2^{r[1+(n-1)\varepsilon
- (1-\varepsilon)\frac{d}{d-\p}]}}
{1+ (2^r\delta)^{A}}.
$$
As $\frac{d}{d-\p}>1$ 
and since $\varepsilon$ 
is assumed to be small, we conclude that
$$ 1+(n-1)\varepsilon
- (1-\varepsilon)\frac{d}{d-\p}< 0.
$$ 
So, the previous $r$-sum 
is $O(1)$. Thus
$$
X_{2, \p} \ll 
2^{-\p
\ell (n-1)}
Q^{n -
\varepsilon\frac{n-1}{2}}
\delta 2^{\ell d}.
$$
This produces the required
bound \eqref{eq: X2 bound},
and thus the proof is complete.
\end{proof}
Now we are in a position
to prove Proposition 
\ref{prop: pruning and stationary phase}.
\begin{proof}[Proof of Proposition \ref{prop: pruning and stationary phase}]
By combining \eqref{eq: decomp Xfrak} and Lemma \ref{lem: Xfrac decompo bound}, we see that
\begin{align*}
\fX_\d1 (\delta,Q,\ell) 
& \ll 2^{-\frac{d}{2}\ell (n-1)} 
\delta^{-\frac{n-1}{2}}
+ 2^{-(\d1) \ell (n-1)}
2^{\ell d}
\delta 
Q^{n - \varepsilon\frac{n-1}{2}}  \\
& \leq 2^{-\frac{d}{2}\ell(n-1)}
2^{\ell d}
\big( 
\delta^{-\frac{n-1}{2}}
+ \delta 
Q^{n - \varepsilon\frac{n-1}{2}} 
\big) \\
& \leq
2^{-\frac{d}{2}\ell(n-1)}
2^{\ell d}
(\delta 
Q^{n -
\varepsilon\frac{n-1}{2}}+ 
Q^{\frac{n-1}{2}}
\delta^{-\frac{n-1}{2}}).
\end{align*}
Similarly, by slackening 
or inserting the appropriate
powers of $2^\ell$ and $Q$
we obtain 
\begin{align*}
\fX_\1 (\delta,Q,\ell) 
& \ll 2^{-\frac{d}{2}\ell (n-1)} 
\delta^{-\frac{n-1}{2}}
+ 2^{-\1 \ell (n-1)}
2^{\ell d}
\delta 
Q^{n - \varepsilon\frac{n-1}{2}}  \\
& \leq 2^{-\ell(n-1)} 2^{\ell d}
\big( 
\delta^{-\frac{n-1}{2}}
+ \delta 
Q^{n - \varepsilon\frac{n-1}{2}} 
\big) \\
& \leq
2^{-\ell(n-1)} 2^{\ell d}
(\delta 
Q^{n -
\varepsilon\frac{n-1}{2}} + 
Q^{\frac{n-1}{2}}
\delta^{-\frac{n-1}{2}}).
\end{align*}
Thus the proof is complete.
\end{proof}
Next, we embark 
on proving the counterpart 
of Proposition \ref{prop: pruning and stationary phase}
for the regime $d>n-1$.

\begin{prop} \label{prop: pruning and stationary phase d large}
Let $\fX_\p$
be the function defined in 
\eqref{def: Xpfrak}.
If $d>n-1 $, then the estimates
\begin{align*}
\fX_\d1(\delta,Q,\ell)
& \ll 
2^{-\frac{d}{2}\ell(n-1)}
2^{\ell (n-1)}
(\delta^{-\frac{n-1}{2}}
+ \delta Q^{\frac{n+1}{2}} ),\\ 
\fX_\1(\delta,Q,\ell)
& \ll 
 Q^{n} 
\Big(\frac{\delta}{Q}
\Big)^{\frac{n-1}{d}}
\delta^{-2\varepsilon(n-1)} + 
Q^{\frac{n-1}{2}}
\delta^{-\frac{n-1}{2}}
\end{align*}
hold uniformly
for all $\delta\in (0, 1/2), Q\geq 1$ 
and all $0\leq \ell \leq \Lc(\delta, Q)$.
\end{prop}
In order to show 
the above proposition,
we need the next lemma.
\begin{lem}
We have
\begin{equation}\label{eq: X2 bound d large}
    X_{2,\d1}
    (\delta, Q,\ell)\ll
    2^{- \frac{d}{2}\ell (n-1)}
    2^{\ell (n-1)}
    \delta Q^{\frac{n+1}{2}},
\end{equation}
and 
\begin{equation}\label{eq: X2 bound d large 2}
    X_{2,\1}(\delta, Q,\ell)
    \ll 
\Big(\frac{\delta}{Q}
\Big)^{\frac{n-1}{d}}
Q^{n} \delta^{-2\varepsilon(n-1)}.
\end{equation}
\end{lem}
\begin{proof}
Notice that
$$
X_{2,\d1}(\delta, Q,\ell)
\leq 
 \sum_{(i,r)\in 
 [0,\frac{\log Q}{100}]\times [R_{-},R]}
\frac{2^{(\frac{d}{2}-(\d1)) 
\ell (n-1)} }
{2^{iA}(1+ (2^r\delta)^{A})}
Q^{\frac{n+1}{2}} \delta 2^{-r \frac{n-1}{2}}
2^{r(1+(n-1)\varepsilon)}.
$$
Inspecting the exponents 
yields that the $r$-summation and the $i$-summation
are dominated by the contribution 
from $r=0=i$. Thus,
$$
X_{2,\d1}(\delta, Q,\ell)
\ll 
2^{(\frac{d}{2}-(\d1)) 
\ell (n-1)} 
Q^{\frac{n+1}{2}} \delta 
$$
which is \eqref{eq: X2 bound d large}.
It remains to show \eqref{eq: X2 bound d large 2}. 
We have
\begin{align*}
X_{2,\1}(\delta, Q,\ell)
& \leq 
 \sum_{(i,r)\in 
 [0,\frac{\log Q}{100}]\times [R_{-},R]}
\frac{2^{(\frac{d}{2}-\1) 
\ell (n-1)} }
{2^{iA}(1+ (2^r\delta)^{A})}
Q^{\frac{n+1}{2}} \delta 2^{-r \frac{n-1}{2}}
2^{r(1+(n-1)\varepsilon)}.
\end{align*}
For $\p=\1$, \eqref{eq: real strange case} gives
$$
\frac{\log 
(Q2^{r(1+\varepsilon)-i})}{d\log 2} 
\geq 
\ell \geq (1-\varepsilon)
\frac{r}{d-\1};
$$
in other words
$$
Q^{\frac{1}{d}}
2^{\frac{r(1+\varepsilon)-i}{d}}
\geq 
2^\ell \geq 
2^{(1-\varepsilon)
\frac{r}{d-\1}}.
$$
So, $2^{\ell d}
\leq Q 
2^{(1+\varepsilon)r-i}$.
Writing
$\frac{d}{2}-1=
d(\frac{1}{2}-\frac{1}{d})$,
we infer have for each fixed 
pair of $i,r$ that
\begin{align*}
&2^{(\frac{d}{2}-\1) \ell (n-1)}
Q^{\frac{n+1}{2}} \delta 
2^{r(1+(n-1)\varepsilon)-r \frac{n-1}{2}}
\\& \leq 
(Q 
2^{(1+\varepsilon)r-i}
)^{(\frac{1}{2}-\frac{1}{d}) 
(n-1)} 
Q^{\frac{n+1}{2}} \delta 
2^{r(1+(n-1)\varepsilon)-r \frac{n-1}{2}}
\\
& = 
2^{-i(\frac{1}{2}
-\frac{1}{d})(n-1)}
Q^{\frac{n+1}{2} + 
(\frac{1}{2}-\frac{1}{d})(n-1)} \delta 
2^{r(1+(n-1)\varepsilon)-r \frac{n-1}{2} +
(1+\varepsilon)r
(\frac{1}{2}-\frac{1}{d}) 
(n-1)}.
\end{align*}
In the last expression above, we 
shall trivially bound the 
in the $i$-aspect, i.e. use
$2^{-i(\frac{1}{2}
-\frac{1}{d})(n-1)} \leq 1$.
To analyse the expression in the $Q$-aspect, 
we observe that
$$
\frac{n+1}{2} + 
(\frac{1}{2}-\frac{1}{d})(n-1)
= 
\frac{n+1+n-1}{2} -\frac{n-1}{d}
=
n -\frac{n-1}{d}.
$$
Finally, in order 
to analyse the $r$-aspect, we 
make use of the identity 
\begin{align*}
 & 1+(n-1)\varepsilon- \frac{n-1}{2} +
(1+\varepsilon)
(\frac{1}{2}-\frac{1}{d}) 
(n-1)
\\& =
  1- \frac{n-1}{2} +
(\frac{1}{2}-\frac{1}{d}) 
(n-1) +\varepsilon
\Big[
n-1+
(\frac{1}{2}-\frac{1}{d}) 
(n-1)
\Big]\\
& =
1
-\frac{n-1}{d}
+
\varepsilon
(\frac{1}{2}-\frac{1}{d}+1) 
(n-1).
\end{align*}
Consequently,
\begin{align*}
X_{2,\1}(\delta, Q,\ell)
& \leq 
 \sum_{(i,r)\in 
 [0,\frac{\log Q}{100}]\times [R_{-},R]}
\frac{Q^{\frac{n+1}{2} + 
(\frac{1}{2}-\frac{1}{d})(n-1)} \delta }
{2^{iA}(1+ (2^r\delta)^{A})}
2^{r[1+(n-1)\varepsilon
- \frac{n-1}{2} +
(1+\varepsilon)
(\frac{1}{2}-\frac{1}{d}) 
(n-1)]}\\
& \leq 
 \sum_{(i,r)\in 
 [0,\frac{\log Q}{100}]\times [R_{-},R]}
\frac{
Q^{n -\frac{n-1}{d}} \delta}
{2^{iA}(1+ (2^r\delta)^{A})}
2^{r[1-\frac{n-1}{d}
+ \varepsilon
(\frac{1}{2}-\frac{1}{d}+1) 
(n-1)]}.
\end{align*}
Since $d>n-1$ and $\varepsilon>0$
is small, the exponent
in the $2^r$-power
reveals itself to be positive.
Hence, 
\begin{align*}
X_{2,\1}(\delta, Q,\ell)
& \ll 
Q^{n -\frac{n-1}{d}} \delta\cdot
\delta^{-[1-\frac{n-1}{d}
+ \varepsilon
(\frac{1}{2}-\frac{1}{d}+1) 
(n-1)]} \\
& =
Q^{n -\frac{n-1}{d}}
\delta^{\frac{n-1}{d}}
\delta^{-\varepsilon
(\frac{1}{2}-\frac{1}{d}+1) (n-1)}
\\
& =  
\Big(\frac{\delta}{Q}
\Big)^{\frac{n-1}{d}}Q^{n}
\delta^{-\varepsilon
(\frac{1}{2}-\frac{1}{d}+1) (n-1)}.
\end{align*}
Because $\frac{1}{2}-\frac{1}{d}+1 < 2$,
this produces the required bound.
\end{proof}
Now proving Proposition \ref{prop: pruning and stationary phase d large} is a simple matter. 
\begin{proof}[Proof of Proposition \ref{prop: pruning and stationary phase d large}]
Let us consider the case $\p=\d1$ first.
By combining \eqref{eq: decomp Xfrak},
\eqref{eq: X1 bound},
and \eqref{eq: X2 bound d large} we see 
\begin{align*}
\fX_\d1 (\delta,Q,\ell) 
& \ll 2^{-\frac{d}{2}\ell (n-1)} 
\delta^{-\frac{n-1}{2}}
+ 
2^{- \frac{d}{2}\ell (n-1)}
2^{\ell (n-1)}
\delta Q^{\frac{n+1}{2}}  \\
& \leq 2^{-\frac{d}{2}\ell(n-1)}
2^{\ell (n-1)}
\big( 
\delta^{-\frac{n-1}{2}}
+ \delta Q^{\frac{n+1}{2}} 
\big),
\end{align*}
as required. It remains to bound 
$\fX_\1(\delta,Q,\ell)$.
By combining \eqref{eq: decomp Xfrak},
\eqref{eq: X1 bound},
and \eqref{eq: X2 bound d large 2}, we obtain
\begin{align*}
\fX_\1 (\delta,Q,\ell) 
& \ll 2^{-\frac{d}{2}\ell (n-1)} 
\delta^{-\frac{n-1}{2}}
+  
\Big(\frac{\delta}{Q}
\Big)^{\frac{n-1}{d}}Q^{n} 
\delta^{-2\varepsilon(n-1)} \\
& \leq 
\delta^{-\frac{n-1}{2}}
+ 
\Big(\frac{\delta}{Q}
\Big)^{\frac{n-1}{d}}Q^{n} 
\delta^{-2\varepsilon(n-1)},
\end{align*}
as required.
\end{proof}
\section{The Bootstrapping Procedure}\label{sec: bootstrapping}
\subsection{The regime 
$d\leq n-1$}\label{subsec: d small}
We now describe the basic bootstrapping argument used to bring down the exponent of $Q$ in the error term from $\beta_{\kappa+1}$ to $\beta_{\kappa}$, up to an $\varepsilon$ of loss.  Given a starting estimate on the number of rational points in the neighborhood of a fixed dyadic piece of the original manifold, 
the first theorem below yields a slightly better 
estimate for the corresponding piece on the dual manifold. 
Moreover, the next theorem establishes the converse, 
thus paving the way for a bootstrapping process. In our case this process needs to be carried out individually at each dyadic scale and for two different types of hypersurfaces (flat versus rough). We shall see that the complementary terms (negative powers of $2^{\ell}$) which arise out of the two regimes combine in a multiplicative way to just offset the loss of decay due to lack of curvature at each scale. There is also a critical dependence of the spatial scale $\ell$ on $\delta$ which plays into the induction on scales (using monotonicity) in the $\delta$ parameter.

Throughout this subsection and the next, 
we abbreviate by 
$$
    \const = \const (\Vert \om \Vert_{\sC^t},n,d)>1
$$
a constant which solely depends on 
on $\Vert \om \Vert_{\sC^t},n,d$ (and is in particular independent
of $\delta, Q$ and $\ell$).
Moreover, we first fix a $K:=K(\varepsilon)\in \mathbb{Z}_{\geq 1}$ and then a 
$$ 
\kappa\in \{0, 1, \ldots, 
K-1\}.
$$ Let $P_\kappa^{[\1]}, P_{\kappa+1}^{[\1]}$ \index{pkappa@$P_{\kappa+1}^{[\p]}:$ envelope of order $\kappa+1 \geq 0$}
be as defined in \eqref{eq: def kappa weight}, 
and let $\beta_{\kappa+1}$ be a fixed real number in $(n-1, n]$. Further, recall the definition of $\fL_{\p}(\delta, Q)$ from \eqref{def Ls}.

\begin{thm}[Flat to Rough Manifold]
\label{thm: og to dual sdeg}
Suppose there exists a constant $C_{\kappa+1}>1$
so that for all $\delta\in(0,1/2)$ and $Q\geq1$, we have
\[
\mathfrak{N}^{P_{\kappa+1}}(\delta,Q, \ell, \1)
\leq 
C_{\kappa+1}
2^{-\ell(n-1-d)}
\big(\delta Q^{n}
+
Q^{\beta_{\kappa+1}+n\varepsilon}
\big)
\]
whenever \ellponecond.
Then there exists a constant $C_{\kappa}>0$, depending on $C_{\kappa+1}, P_\kappa$ and $\varepsilon$, such that for all $Q\geq 1$ and $\delta\in (0, 1/2]$, we have 
\begin{equation}
\label{eq concl boots 1}
\mathfrak{N}^{P_{\kappa}}(\delta,Q, \ell, \d1)\leq 
C_{\kappa}2^{-\frac{d}{2}\ell(n-1)}2^{\ell d}\left(\delta Q^{n}
+Q^{\beta_\kappa+n\varepsilon}\right),    
\end{equation}
for all \ellpdcond, and with 
\[\beta_{\kappa}:=n-\frac{n-1}{2\beta_{\kappa+1}-n+1}.\]
\end{thm}
\begin{proof}
Proposition \ref{prop: gathering bounds}
tells us that for all $Q\geq 1$ and $\delta\in (0, 1/2)$,  $\fN^{P_{\kappa}}(\delta,Q,\ell,\d1)$
equals 
\begin{align*}
c_{\kappa} 2^{-\ell (\d1) (n-1)} \delta Q^{n}
&+ \sum_{r\in \mathscr{G}(\delta, Q, \ell)} N_1^{P_{\kappa}}(\delta,Q,\ell,\d1, \r)
\\&+ O(2^{-\ell (\d1) (n-1)} \log J 
+ 2^{-\ell((\d1)(n-1)-d)}
\delta Q^{n-1+\varepsilon}),    
\end{align*}
whenever \ellpdcond. 
Thus we can write
\begin{equation*}
 \fN^{P_{\kappa}}(\delta,Q,\ell,\d1)
 \ll 
2^{-\ell  (\d1)(n-1)} 2^{\ell d} 
( \delta Q^{n}+ \log J)
+ 
\sum_{r\in \mathscr{G}(\delta, Q, \ell)} 
N_1^{P_{\kappa}}(\delta,Q,\ell,\d1, \r).  
\end{equation*} 
Let 
\begin{equation}\label{def B}
B(\delta,Q):=
\delta^{-\frac{n-1}{2}} Q^{\frac{n-1}{2}}
+ (\delta^{-1} Q^{-5})^{n-1}.
\end{equation}
We claim that for all $\delta\in(0,1/2)$ and $Q\geq1$,
\begin{align}\label{eq: bound on N1 sum 1}
&\sum_{r\in \mathscr{G}(\delta, Q, \ell)} 
N_1^{P_{\kappa}}(\delta,Q,\ell,\d1, \r)
\\&\ll 
2^{-\frac{d}{2}\ell(n-1)}2^{\ell d}
\left(\delta Q^{n - \varepsilon\frac{n-1}{2}}
+Q^{\frac{n+1}{2}}
\delta^{\frac{n+1}{2}
-\beta_{\kappa+1}-n\varepsilon}
+ B(\delta,Q)
\right)\nonumber
\end{align}
whenever $1\leq \ell\leq \fL_{\d1}(\delta, Q)$, with an implicit constant which is independent of $\delta, Q$ and $\ell$.

Let us first see how the above claim implies estimate \eqref{eq concl boots 1}. Since $d\geq 2$, we have $\frac{d}{2}\leq \d1$. Slackening the $\ell$-powers and absorbing the term 
$\log J \leq \log Q + \log \frac{1}{\delta}$
into the right hand side of 
\eqref{eq: bound on N1 sum 1}, we conclude that for all $\delta\in(0,1/2)$ and $Q\geq1$, we have 
\begin{equation}
    \label{boots intermediate 1}
    \mathfrak{N}^{{P}_{\kappa}}(\delta,Q, \ell, \d1)
\leq 
C_{\kappa}'
2^{-\frac{d}{2}\ell(n-1)}2^{\ell d}
\left(\delta Q^{n - \varepsilon\frac{n-1}{2}}
+Q^{\frac{n+1}{2}}
\delta^{\frac{n+1}{2}-\beta_{\kappa+1}-n\varepsilon}
+ B(\delta,Q)
\right)
\end{equation}
whenever
\ellpdcond \,for a large constant $C_{\kappa}'$ depending on $P_{\kappa}$ and $\varepsilon$, satisfying $C_{\kappa}'>\const \cdot C_{\kappa+1}+c_\kappa$.

Fixing $Q$, $\delta$ and \ellpdcond \, now, we distinguish two cases, depending on the size of $\delta$ with respect to the scale $Q^{\beta_{\kappa}-n}\in (Q^{-1}, 1]$,
where $$\beta_{\kappa}:=n-\frac{n-1}{2\beta_{\kappa+1}-n+1}.$$
Rearranging terms in the above equality, we get
\begin{equation}
    \label{eq beta alt}
    n-\beta_{\kappa}=\frac{n-1}{2\beta_{\kappa+1}-n+1}.
\end{equation}

In the case when $\delta\geq Q^{\beta_{\kappa}-n}$, we have
\[
    B(\delta,Q)+Q^{\frac{n+1}{2}}
    \delta^{\frac{n+1}{2}-
    \beta_{\kappa+1}-n\varepsilon}\leq 
    2Q^{(n-\beta_{\kappa})(\frac{n-1}{2})}
    Q^{\frac{n-1}{2}}
    +Q^{\frac{n+1}{2}}Q^{(n-\beta_{\kappa})
    (-\frac{n+1}{2}+\beta_{\kappa+1}+n\varepsilon)}.
\]
For the power of $Q$ in the first term, we have the estimate
\[\frac{n-1}{2}+(n-\beta_{\kappa})(\frac{n-1}{2})\leq n-1.\]
Further, simplifying the power in the second term gives
\begin{align*}
&\frac{n+1}{2}+(n-\beta_{\kappa})(-\frac{n+1}{2}+\beta_{\kappa+1}+n\varepsilon)\\&\leq \frac{n+1}{2}+(n-\beta_{\kappa})(-\frac{n-1}{2}+\beta_{\kappa+1})-(n-\beta_{\kappa})+n\varepsilon  \\
&=\frac{n+1}{2}+\frac{n-1}{2\beta_{\kappa+1}-n+1}(-\frac{n+1}{2}+\beta_{\kappa+1})-(n-\beta_{\kappa})+n\varepsilon\\
&=\frac{n+1}{2}+\frac{n-1}{2}-(n-\beta_{\kappa})+n\varepsilon=\beta_{\kappa}+n\varepsilon,
\end{align*}
where we used \eqref{eq beta alt} to get the second equality.

It follows that when $\delta \geq Q^{\beta_{\kappa}-n}$, we have
\[\delta Q^n+B(\delta,Q)
+Q^{\frac{n+1}{2}}
\delta^{\frac{n+1}{2}-\beta_{\kappa+1}-n\varepsilon}
\leq \delta Q^n+2 Q^{n-1}
+Q^{\beta_{\kappa}+n\varepsilon}
\leq \delta Q^n+3Q^{\beta_{\kappa}+n\varepsilon}.\]
Combining the above with 
\eqref{boots intermediate 1}, we conclude that
$$
    \mathfrak{N}^{{P}_{\kappa}}(\delta,Q, \ell, \d1)
\leq 
C_{\kappa}'
2^{-\frac{d}{2}\ell(n-1)}2^{\ell d}
\left(\delta 
Q^{n}+2Q^{\beta_{\kappa}+n\varepsilon}
\right),
$$
which implies \eqref{eq concl boots 1} 
for $\delta \geq Q^{\beta_{\kappa}-n}$
with $C_\kappa=3C_\kappa'$. 

We now turn to the case when $\delta\leq Q^{\beta_{\kappa}-n}$. Using the monotonicity of $\mathfrak{N}^{P_{\kappa}}(\delta,Q, \ell, \d1)$ as a function of $\delta$, we can  bound 
$$
\mathfrak{N}^{P_{\kappa}}(\delta,Q, \ell, \d1)\leq  \mathfrak{N}^{P_{\kappa}}(Q^{\beta_{\kappa}-n},Q, \ell, \d1).$$
In the case when $2^{\ell d}\geq Q^{n-\beta_\kappa+1}$, we use the trivial bound 
\begin{equation*}
    \mathfrak{N}^{P_{\kappa}}(Q^{\beta_{\kappa}-n},Q, \ell, \d1)\leq B_\kappa 2^{-\ell(d-1)(n-1)}Q^n+1\leq 2B_\kappa 2^{-\ell(d-1)(n-1)}Q^n,
\end{equation*}
where $B_\kappa>1$ is a positive constant depending on $P_\kappa$ but independent of $\delta, Q$ and $\ell$, and the last inequality follows from the fact that  $2^{\ell(\d1)(n-1)}\leq 2^{\fL_{\d1}(d-1)(n-1)}\leq Q$ (on account of \eqref{eq: 2^Lp}). 
Note that since $2^{\ell d}\geq Q^{n-\beta_\kappa+1}\geq Q$ in this regime, we can bound
\begin{align*}
    \mathfrak{N}^{P_{\kappa}}(Q^{\beta_{\kappa}-n},Q, \ell, \d1)
    &\leq B_\kappa 2^{-\ell\frac{d}{2}(n-1)}2^{\ell d}2^{-\ell d} Q^n\leq  B_\kappa 2^{-\ell\frac{d}{2}(n-1)}2^{\ell d} Q^{n-1}
    \\&\leq B_\kappa 2^{-\ell\frac{d}{2}(n-1)}2^{\ell d}Q^{\beta_\kappa+n\varepsilon}.
\end{align*}
This implies the desired conclusion \eqref{eq concl boots 1} for $\mathfrak{N}^{P_{\kappa}}(\delta,Q, \ell, \d1)$ when $2^{\ell d}\geq Q^{n-\beta_\kappa+1}$, with $C_\kappa=B_\kappa$.

Finally, when $2^{\ell d}\leq Q^{n-\beta_\kappa+1}$, we use \eqref{boots intermediate 1} with $\delta$ replaced by $Q^{\beta_\kappa-n}$ to estimate
\begin{align*}
&\mathfrak{N}^{P_{\kappa}}(\delta,Q, \ell, \d1)\\&\leq  \mathfrak{N}^{P_{\kappa}}(Q^{\beta_{\kappa}-n},Q, \ell, \d1)\\ &\leq C_{\kappa}'
2^{-\frac{d}{2}\ell(n-1)}2^{\ell d}\left(Q^{\beta_{\kappa}-n}
Q^{n}+2 Q^{(n-\beta_{\kappa})(\frac{n-1}{2})}Q^{\frac{n-1}{2}}
+Q^{\frac{n+1}{2}}
Q^{(n-\beta_{\kappa})
(-\frac{n+1}{2}+\beta_{\kappa+1}+n\varepsilon)}\right)
\\
&\leq C_{\kappa}'
2^{-\frac{d}{2}\ell(n-1)}
2^{\ell d}\left(Q^{\beta_{\kappa}}
+ 2Q^{n-1}+Q^{\beta_{\kappa}+n\varepsilon}\right)
\leq 4C_{\kappa}'
2^{-\frac{d}{2}\ell(n-1)}2^{\ell d}
Q^{\beta_{\kappa}+n\varepsilon}.
\end{align*}
Thus we obtain \eqref{eq concl boots 1} again with $C_\kappa=4C_\kappa'$.

To finish the proof, it now remains to prove claim \eqref{eq: bound on N1 sum 1}.

\textbf{Proof of Claim \eqref{eq: bound on N1 sum 1}:}
Recall the set $\sX=\sX(Q,\ell, \1)$, defined in \S\ref{subsec strange terms}, consisting of all integer pairs 
$(i,r)\in [0,\frac{\log Q}{100}]\times [R_{-},R]$ such that 
$$\ell >\min\left\{\frac{\log (Q2^{r(1+\varepsilon)-i})}{d\log 2}, 
\frac{(1-\varepsilon)r}{\1}\right\}=\fL_{\1}(2^i/Q, 2^r).$$ 
Let 
\begin{equation}\label{def T add}
T(\delta,Q,\ell,\d1, r):=
2^{-(\d1) \ell (n-1)} 
\delta 2^r (1+ (2^r\delta)^{A})^{-1} Q^{-5(n-1)}2^{r(n-1)}.
\end{equation}
In view of
Lemma \ref{lem: N1 after stationary phase and geom sum}, we have
\begin{align}
    \label{eq: N1 bd 1}
    &\sum_{r\in 
    \mathscr{G}(\delta, Q, \ell)} 
    N_1^{P_{\kappa}}(\delta,Q,\ell,\d1, \r)
    \\& \ll_A \fX_\d1(\delta,Q,\ell)
    + 2^{(\frac{d}{2}-(\d1)) 
\ell (n-1)}
\delta Q^{\frac{n+1}{2}} 
\Big(\sum_{(i,r)} 
\frac{\fN^{{P}_{\kappa+1}}(2^{i+1}/Q, 
2^r,\ell, \1)}
{2^{iA}(1+ (2^r\delta)^{A})
2^{r \frac{n-1}{2}}}\nonumber\\
& + 
2^r Q^{1-n} \max(1, 2^{-\ell 
(d-(\d1))(n-1)}2^{r(n-1)})\Big) +  
\sum_{r\in \mathscr{G}(\delta, Q, \ell)} 
    T(\delta,Q,\ell,\d1, r)\nonumber
\end{align}
where the summation on the right 
is taken over all integer pairs
$(i,r) $ 
in $[0,\frac{\log Q}{100}]\times \mathscr{G}(\delta, Q, \ell)$
with $\ell\leq \fL_{\1}(2^i/Q, 2^r)$.
Here $\fX_\d1(\delta,Q,\ell)$ is the contribution coming from $\sX=\sX(Q,\ell, \1)$ (see \eqref{def: Xpfrak}).
Notice that 
$$
\sum_{r\in \mathscr{G}(\delta, Q, \ell)} 
    T(\delta,Q,\ell,\d1, r)
\ll
2^{-(\d1) \ell (n-1)} 
Q^{-5(n-1)}\delta^{-(n-1)}
\ll 2^{-(\d1) \ell (n-1)} 
B(\delta,Q).
$$
Clearly, $d\geq 2$ implies $\frac{d}{2}\leq d-1$,
and thus
\begin{equation}
\label{def d estimate}
2^{-(\d1) \ell (n-1)} 
B(\delta,Q)\leq 2^{- \frac{d}{2} \ell (n-1)}  2^{\ell d} B(\delta,Q).   
\end{equation}
Upon using this estimate,
we infer 
\begin{equation}
\label{eq Tterm1}
\sum_{r\in \mathscr{G}(\delta, Q, \ell)} T(\delta,Q,\ell,\d1, r)
\leq 2^{- \frac{d}{2} \ell (n-1)}  2^{\ell d} B(\delta,Q).    
\end{equation}
Moreover,
Proposition \ref{prop: pruning and stationary phase} already tells us that
\begin{equation}
    \label{eq alrd1}
    \fX_\d1(\delta,Q,\ell)
 \ll 
2^{-\frac{d}{2}\ell(n-1)}
2^{\ell d}
(\delta 
Q^{n - \varepsilon\frac{n-1}{2}}+ 
B(\delta,Q)),
\end{equation}
to establish the claim we now 
focus on bounding the two summands 
on the right in \eqref{eq: N1 bd 1} in the bracket.
Observe that the sum in the second term in \eqref{eq: N1 bd 1} now runs over a range of parameters where we can apply our assumption. We introduce the functions
\begin{align}\label{def: T1 T2}
&T_1(\delta',Q') 
:=\delta' (Q')^{n}, \qquad
\mathrm{and} \qquad
T_2(\delta',Q') :=
 (Q')^{\beta_{\kappa+1} + n \varepsilon},
\end{align}
to write our assumption in the form 
$$
\fN^{P_{\kappa+1}}(2^i/Q, 2^r,\ell, \1) 
\ll 
2^{-\ell (n-1-d)}
\big[
T_{1}(2^i/Q ,2^r) + 
T_{2}(2^i/Q ,2^r)\big],
$$
for $$\ell\leq \fL_{\1}(2^i/Q, 2^r).$$
A quick inspection shows that if $A>0$
is chosen sufficiently large, then
$$
\sum_{ 0\leq i\leq \frac{\log Q}{100}} 
\frac{T_{s}(\frac{2^i}{Q},2^r,\ell)}{2^{Ai}}
\ll T_{s}(\frac{1}{Q},2^r)
$$
for any choice of $s\leq 2$. 
As a result, for each fixed $r\in \mathscr{G}(\delta, Q, \ell)$, we have
$$
\sum_{ 0\leq i\leq \frac{\log Q}{100}} 
\frac{\fN^{P_{\kappa+1}}(\frac{2^i}{Q}, 2^r,\ell, \1)}{2^{Ai}}
\ll 
2^{-\ell (n-1-d)}
\big[
T_{1}(Q^{-1},2^r) + 
T_{2}(Q^{-1} ,2^r)\big].
$$
Combining the above with \eqref{eq: N1 bd 1}, \eqref{eq Tterm1} and \eqref{eq alrd1}, we get 
\begin{align}\label{eq: N1 intermediate bound 1}
& \sum_{r\in \mathscr{G}(\delta, Q, \ell)}N_{1}^{{P}_\kappa}(\delta,Q, \ell, \d1, \r) 
\\ &\ll_A
 \sum_{r\in \mathscr{G}(\delta, Q, \ell)}\frac{2^{(\frac{d}{2}-(\d1)) \ell (n-1)} }
{1+ (2^r\delta)^{A}}
Q^{\frac{n+1}{2}} \delta 2^{-r \frac{n-1}{2}}
\Big(
 2^{-\ell (n-1-d)}
\sum_{s\leq 2} 
T_{s}(Q^{-1} ,2^r)\nonumber\\
&+ 2^r Q^{1-n} (1+ 2^{-\ell}2^{r})^{n-1}\Big)
+
2^{-\frac{d}{2}\ell(n-1)}
2^{\ell d}
\delta 
Q^{n - \varepsilon\frac{n-1}{2}}+ 
2^{(\frac{d}{2}-\p) \ell (n-1)}  B(\delta,Q).\nonumber
\end{align}
Next, we claim that
\begin{equation}\label{eq: sum over r}
\sum_{s \leq 2}
\sum_{0\leq r \leq R}
T_{s}(\frac{1}{Q},2^r)
 \frac{2^{-r \frac{n-1}{2}}}{1+(\delta 2^r)^A}
 \ll 
 \sum_{s \leq 2}
 \delta^{\frac{n-1}{2}}
T_{s}(\frac{1}{Q},
\frac{1}{\delta}).
\end{equation}
To verify this, we first recall that
$r\leq R= \frac{\log J}{\log 2}=\frac{\log (\delta^{-1}Q^{\varepsilon})}{\log 2}$ for $r\in \mathscr{G}(\delta, Q, \ell)$.
Observe for $s=1,2$ that
$2^{-r\frac{n-1}{2}}
T_{s}(\frac{1}{Q},2^r,\ell)$
is of the form $2^{rv}$ times some function
depending on $\ell, Q$
where $v > 0$.
With \eqref{eq: uncertaintly prin}
at hand we deduce
$$
\sum_{ r\in \mathscr{G}(\delta, Q, \ell)}
2^{-r\frac{n-1}{2}}
\frac{T_{s}(\frac{1}{Q},2^r)}
{1+(2^{r}\delta)^A}
\ll \delta^{\frac{n-1}{2}}
T_{s}(\frac{1}{Q},
\frac{1}{\delta}).
$$
As a result,
we infer \eqref{eq: sum over r}. 
Combining \eqref{eq: N1 intermediate bound 1} and \eqref{eq: sum over r}
gives
\begin{align}\label{eq: sum N1 over r when d small}
& \sum_{r\in \mathscr{G}(\delta, Q, \ell)}
 N_{1}^{{P}_\kappa}(\delta,Q, \ell, \d1, \r) 
\\ & \ll 
2^{(\frac{d}{2}-(\d1)) \ell (n-1)} 2^{-\ell (n-1-d)}
Q^{\frac{n+1}{2}} \delta 
 \Big(
\sum_{s \leq 2}
 \delta^{\frac{n-1}{2}}
T_{s}(\frac{1}{Q},
\frac{1}{\delta})
+ 2^{\ell (n-1-d)} Y 
\Big)\nonumber\\
& + 
2^{-\frac{d}{2}\ell(n-1)}
2^{\ell d}
(\delta 
Q^{n - \varepsilon\frac{n-1}{2}}+ 
B(\delta,Q)) \nonumber\end{align}
where 
\begin{equation} \label{def: Y term}
Y:= Q^{1-n}  \sum_{r\in \mathscr{G}(\delta, Q, \ell)} 
\frac{2^{-r \frac{n-3}{2}}}
{1+ (2^r\delta)^{A}}
 (1+ 2^{-\ell}2^{r})^{n-1}.
\end{equation}
We claim that
\begin{equation}\label{eq: Y at most T1 T2}
2^{\ell (n-1-d)} Y \ll \delta^{\frac{n-1}{2}} T_{1}(\frac{1}{Q}, \frac{1}{\delta}).
\end{equation}
For justifying that, we begin with remarking
that for any $x>0$ the inequality 
$(1 + x)^{n-1}\ll 1+x^{n-1} $
holds true. 
Thus,
the estimate 
\eqref{eq: uncertaintly prin}
produces
\begin{equation}\label{eq: upper bound on Y}
Y \ll Q^{1-n} 
\big(
R
+ \delta^{\frac{n-3}{2}} 
2^{- \ell (n-1)}
\delta^{-(n-1)}
\big)
\ll 
Q^{1-n} 
\big(
\log(Q/\delta)
+ 
2^{- \ell (n-1)}
\delta^{-\frac{n+1}{2}} 
\big).
\end{equation}
To proceed, we notice 
\begin{align}
& \delta^{\frac{n-1}{2}} T_1(\frac{1}{Q},\frac{1}{\delta}) 
=  \delta^{\frac{n-1}{2}} Q^{-1} \delta^{-n}
= \delta^{-\frac{n+1}{2}} Q^{-1}, \label{eq: T1 things} \\
& 
\delta^{\frac{n-1}{2}} T_2(\frac{1}{Q},\frac{1}{\delta}) 
=
\delta^{\frac{n-1}{2}} 
\delta^{-(\beta_{\kappa+1} + n \varepsilon)}
= \delta^{-(\beta_{\kappa+1} 
-\frac{n-1}{2}+ n \varepsilon)}\label{eq: T2 things} .
\end{align}
Now to verify \eqref{eq: Y at most T1 T2},
we distinguish two cases. If 
$2^{-\ell(n-1)}\delta^{-\frac{n+1}{2}} \geq \log (Q/\delta)$,
then \eqref{eq: upper bound on Y} implies
$$
2^{\ell (n-1-d)} Y
\ll 
Q^{1-n} 
2^{-\ell d}
\delta^{-\frac{n+1}{2}}\leq \delta^{-\frac{n+1}{2}} Q^{-1}\stackrel{\eqref{eq: T1 things}}{=}\delta^{\frac{n-1}{2}} T_1(\frac{1}{Q},\frac{1}{\delta}).
$$
Now suppose $2^{-\ell(n-1)}\delta^{-\frac{n+1}{2}} < \log (Q/\delta)$.
By replacing $\ell$ by $L=L(\delta, Q)$ from \eqref{eq: 2^L} and using that $d\geq 2$,
we derive
\begin{align*}
2^{\ell (n-1-d)} Y
\ll 
2^{L \cdot(n-1-d)} \left(\frac{Q}{\delta}\right) Q^{1-n}
\ll \left(\frac{Q}{\delta}\right)^{\frac{n-1-d}{d}+1} Q^{1-n}
&\leq 
\left(\frac{Q}{\delta}\right)^{\frac{n-3}{2}+1} Q^{1-n}
\\&= 
\delta^{-\frac{n-1}{2}}Q^{-\frac{n-1}{2}}.    
\end{align*}
Since $n\geq 3$, recalling \eqref{eq: T1 things} establishes \eqref{eq: Y at most T1 T2} in this case as well.
Observe $$2^{(\frac{d}{2}-(\d1)) \ell (n-1)} 2^{-\ell (n-1-d)}
 = 2^{-\frac{d}{2} \ell (n-1)}2^{\ell d}.$$ 
Combining \eqref{eq: sum N1 over r when d small}, \eqref{eq: Y at most T1 T2}, \eqref{eq: T1 things} and \eqref{eq: T2 things}
yields
\begin{align*}
    &\sum_{r\in \mathscr{G}
    (\delta, Q, \ell)}N_{1}^{{P}_\kappa}
    (\delta, Q, \ell, \d1, \r) \\&\ll
    2^{-\frac{d}{2} \ell (n-1)}2^{\ell d}
Q^{\frac{n+1}{2}} \delta
 \Big(\delta^{-\frac{n+1}{2}} Q^{-1}
+
\delta^{-(\beta_{\kappa+1} -\frac{n-1}{2}
+ n \varepsilon)}
\Big) \\
&+2^{-\frac{d}{2} \ell (n-1)}2^{\ell d}
\Big(\delta Q^{n - \varepsilon\frac{n-1}{2}} + 
B(\delta,Q)\Big)\\
& \ll
2^{-\frac{d}{2} \ell (n-1)}2^{\ell d}
 \Big(\delta Q^{n - \varepsilon\frac{n-1}{2}}
 + B(\delta,Q)
 + Q^{\frac{n+1}{2}} 
\delta^{-(\beta_{\kappa+1}-\frac{n+1}{2}
+ n \varepsilon)}
\Big).
\end{align*}
Since this is exactly the 
bound \eqref{eq: bound on N1 sum 1},
the proof is complete.
\end{proof}

We now consider the reverse scenario, which shall be proven in an analogous way. 

\begin{thm}  [Rough to Flat Manifold] 
\label{thm: dual to og sdeg}
Suppose there exists a constant $C_{\kappa+1}>1$
so that for all $\delta\in(0,1/2)$ and $Q\geq1$, we have
\[
\mathfrak{N}^{{P}_{\kappa+1}}(\delta,Q, \ell, \d1)\leq 
C_{\kappa+1}
2^{-\frac{d}{2}\ell(n-1)}2^{\ell d}
\Big(\delta Q^{n}
+
Q^{\beta_{\kappa+1}+n\varepsilon}\Big),
\] 
whenever \ellpdcond.
Then there exists a constant $C_{\kappa}>0$, depending on $C_{\kappa+1}$, $P_\kappa$ and $\varepsilon$, such that for all $Q\geq 1$ and $\delta\in (0, 1/2]$, we have
\begin{equation}
\label{eq concl boots 2}
\mathfrak{N}^{P_{\kappa}}(\delta,Q, \ell, \1)\leq 
C_{\kappa}2^{-\ell(n-1-d)}\left(\delta Q^{n}
+Q^{\beta_\kappa+n\varepsilon}\right),    
\end{equation}
for all \ellponecond \, 
with \[\beta_{\kappa}:=n-\frac{n-1}{2\beta_{\kappa+1}-n+1}.\]
\end{thm}
\begin{proof}
We again take Proposition \ref{prop: gathering bounds} as our starting point.
So, we know 
$\fN^{P_\kappa}(\delta,Q,\ell,\1)$
equals 
$$
c_{\kappa} 2^{-\ell (n-1)} \delta Q^{n}
+ \sum_{r\in \mathscr{G}(\delta, Q, \ell)} N_1^{P_{\kappa}}(\delta,Q,\ell,\1, \r)
+ O(2^{-\ell (n-1)} \log J 
+ 2^{-\ell(n-1-d)}
\delta Q^{n-1+\varepsilon})
$$
for all $Q\geq 1$ and 
$\delta\in (0, 1/2)$,
whenever $1\leq \ell\leq \fL_{\1}(\delta, Q)$.
Since
$2^{-\ell(n-1-d)}
\delta Q^{n-1+\varepsilon}
\ll 2^{-\ell(n-1-d)}\delta Q^n$,
we have
\begin{equation*}
 \fN^{P_\kappa}(\delta,Q,\ell,\1)
 \ll 
2^{-\ell(n-1-d)}  
( \delta Q^{n}+ \log J)
+ \sum_{r\in \mathscr{G}
(\delta, Q, \ell)} 
N_1^{P_\kappa}(\delta,Q,\ell,\1, \r) 
\end{equation*} 
for all $Q$, $\delta$ and $\ell$ in the aforementioned ranges.

For ease of exposition, 
we define 
\begin{equation}\label{def U and BdeltaQ}
U:=\sum_{r\in \mathscr{G}
(\delta, Q, \ell)} 
N_1^{P_\kappa}(\delta,Q,\ell,\1, \r) \quad
\mathrm{and}\quad 
B(\delta,Q):=\delta^{-\frac{n-1}{2}}
Q^{\frac{n-1}{2}} + (\delta^{-1}Q^{-5})^{n-1}.
\end{equation}
We claim that for all $Q\geq 1$ and $\delta\in (0, 1/2)$, we have
\begin{equation}\label{eq: bound on N1 sum 2}
U
\ll 
2^{-\ell(n-1-d)}
\left(\delta Q^n+B(\delta,Q)+Q^{\frac{n+1}{2}}
\delta^{\frac{n+1}{2}
-\beta_{\kappa+1}-n\varepsilon}
\right),
\end{equation}
for $\ell\in \{0, 1, \ldots, \fL_{1}(\delta, Q)\}$, 
with the implicit constant independent of 
$\delta, Q$ and $\ell$.
We again assume 
claim \eqref{eq: bound on N1 sum 2} for now and see how it implies 
\eqref{eq concl boots 2}.

The term 
$\log J \leq \log Q + \log \frac{1}{\delta}$
can be absorbed into the right hand side of 
\eqref{eq: bound on N1 sum 2} to get
\begin{align}
\label{boots intermediate 2}
\mathfrak{N}^{P_{\kappa}}(\delta,Q, \ell, \1)
&\leq C_{\kappa}' 
2^{-\ell(n-1-d)}
\left(\delta 
Q^{n}+B(\delta,Q)
+Q^{\frac{n+1}{2}}
\delta^{\frac{n+1}{2}
-\beta_{\kappa+1}-n\varepsilon}\right),
\end{align}
for a large constant $C_{\kappa}'$ depending on $P_{\kappa}$ and $\varepsilon$, satisfying $C_{\kappa}'>\const \cdot C_{\kappa+1}+c_\kappa$.

Fixing $Q$, $\delta$ and $\ell \in \{0, 1, \ldots, \fL_{\1}(\delta, Q)\}$ now, we again distinguish two cases, depending on the size of $\delta$ with respect to the scale $Q^{\beta_{\kappa}-n}\in (Q^{-1}, 1]$. 

The proof of \eqref{eq concl boots 2} when $\delta\geq Q^{\beta_{\kappa}-n}$, or when $\delta\leq Q^{\beta_{\kappa}-n}$ and $2^{\ell d}\leq Q^{n-\beta_\kappa-1}$, being independent of the powers of $2^\ell$, follows from \eqref{boots intermediate 2} in the same way as the proof of \eqref{eq concl boots 1} from \eqref{boots intermediate 1}, with $C_\kappa=4C_\kappa'$.

Thus we focus on the scenario when $\delta\leq Q^{\beta_\kappa-n}$ and $2^{\ell d}\geq Q^{n-\beta_\kappa+1}\geq Q$. We can use the monotinicity of $\mathfrak{N}^{P_{\kappa}}(\delta,Q, \ell, \1)$ as a function of $\delta$ to bound 
$$ \mathfrak{N}^{P_{\kappa}}(\delta,Q, \ell, \1)\leq \mathfrak{N}^{P_{\kappa}}(Q^{\beta_{\kappa}-n},Q, \ell, \1).$$
We now combine the trivial estimate on $\mathfrak{N}^{P_{\kappa}}(Q^{\beta_{\kappa}-n},Q, \ell, \1)$, noting that \eqref{eq: 2^Lp} implies  $$2^{\ell(n-1)}\leq 2^{\fL_{\1}(n-1)}\leq Q,$$  with the lower bound on $2^{\ell d}$, to get
\begin{equation*}
    \mathfrak{N}^{P_{\kappa}}(Q^{\beta_{\kappa}-n},Q, \ell, \1)\leq B_\kappa 2^{-\ell(n-1)}Q^n=B_\kappa 2^{-\ell(n-1-d)}2^{-\ell d}Q^n\leq B_\kappa 2^{-\ell(n-1-d)}Q^{n-1}.
\end{equation*}
Here $B_\kappa$ is a positive constant depending on $P_\kappa$ but independent of $\delta, Q$ and $\ell$.
This implies \eqref{eq concl boots 2} for $\mathfrak{N}^{P_{\kappa}}(\delta,Q, \ell, \d1)$ when $2^{\ell d}\geq Q^{n-\beta_\kappa+1}$ with $C_\kappa=B_\kappa$. 

To conclude the proof, it only remains to establish \eqref{eq: bound on N1 sum 2}. 

\textbf{Proof of Claim \eqref{eq: bound on N1 sum 2}:}
Recall the set $\sX=\sX(Q,\ell, \d1)$, defined in \S\ref{subsec strange terms}, consisting of all integer pairs 
$(i,r)\in [0,\frac{\log Q}{100}]\times [R_{-},R]$ such that 
$$\ell >\min\left\{\frac{\log (Q2^{r(1+\varepsilon)-i})}{d\log 2}, 
\frac{(1-\varepsilon)r}{\d1}\right\}=\fL_{\d1}(2^i/Q, 2^r).$$ 
Using Lemma \ref{lem: N1 after stationary phase and geom sum}, we conclude
\begin{align}
    \label{eq: N1 bd 2}
   & U  \ll \fX_\1(\delta,Q,\ell)  +
2^{(\frac{d}{2}-\1 ) 
\ell (n-1)}
\delta Q^{\frac{n+1}{2}} \cdot 
\\
&\cdot \Bigg(\sum_{(i,r)} 
\frac{\fN^{{P}_{\kappa+1}}(2^{i+1}/Q, 
2^r,\ell, \d1)}
{2^{iA}(1+ (2^r\delta)^{A})}
2^{-r \frac{n-1}{2}} \nonumber + 2^r Q^{1-n} \max(1, 2^{-\ell (d-(\d1))(n-1)}2^{r(n-1)})\Bigg)+ \nonumber\\
& + \sum_{r\in \mathscr{G}(\delta, Q, \ell)} 
    T(\delta,Q,\ell,\1, r)\nonumber
\end{align}
where 
\begin{equation}\label{def T backsl 1}
    T(\delta,Q,\ell, \1, r):= 
2^{- \ell (n-1)} 
\delta 2^r (1+ (2^r\delta)^{A})^{-1} Q^{-5(n-1)}2^{r(n-1)} 
\end{equation}is as defined in analogy 
with
\eqref{def T add} (but now with $\p= \1$ instead of $\d1$), 
and the sum on the right hand side of \eqref{eq: N1 bd 2}
is taken over all integer pairs
$(i,r) $ 
in $[0,\frac{\log Q}{100}]\times 
\mathscr{G}(\delta, Q, \ell)$
so that $\ell\leq \fL_{\d1}(2^i/Q, 2^r)$.
Here $\fX_\1(\delta,Q,\ell)$ is the contribution coming from $\sX=\sX(Q,\ell, \d1)$ (see \eqref{def: Xpfrak}).
Notice that 
\begin{equation}
    \label{eq Tterm2}
    \sum_{r\in \mathscr{G}(\delta, Q, \ell)} 
    T(\delta,Q,\ell, r)
\ll
2^{-\ell (n-1)} 
Q^{-5(n-1)}\delta^{-(n-1)}
\leq 2^{-\ell (n-1)} 2^{\ell d} B(\delta,Q).
\end{equation}
Since Proposition \ref{prop: pruning and stationary phase} already tells us that
\begin{equation}
    \label{eq alrd2}
    \fX_\1(\delta,Q,\ell)
 \ll 
2^{-\ell(n-1)}
2^{\ell d}
(\delta Q^{n} + 
Q^{\frac{n-1}{2}}
\delta^{-\frac{n-1}{2}}),
\end{equation}
to establish the claim we now focus on bounding the two summands on the right in the bracket 
in \eqref{eq: N1 bd 2}.
Recalling 
$T_1(\delta,Q)$ and $ T_2(\delta,Q)$
from \eqref{def: T1 T2},
we have
$$
\fN^{P_{\kappa+1}}(\delta, Q,\ell, \d1) 
\ll 
2^{-\frac{d}{2}\ell(n-1)}2^{\ell d}
\big[
T_{1}(\delta ,Q) + 
T_{2}(\delta ,Q)\big],
$$
for $$\ell\leq \fL_{\d1}(2^i/Q, 2^r).$$
A quick inspection shows that if $A>0$
is chosen sufficiently large, then
$$
\sum_{ 0\leq i\leq \frac{\log Q}{100}} 
\frac{T_{s}(\frac{2^i}{Q},2^r)}{2^{Ai}}
\ll T_{s}(\frac{1}{Q},2^r)
$$
for any choice of $s=1,2$.
As a result, for each fixed $r\in \mathscr{G}(\delta, Q, \ell)$, we have
$$
\sum_{ 0\leq i\leq \frac{\log Q}{100}} 
\frac{\fN^{P_{\kappa+1}}(\frac{2^i}{Q}, 2^r,\ell, \d1)}{2^{Ai}}
\ll 
2^{-\ell \frac{d}{2}(n-1)}2^{\ell d}
\big[
T_{1}(Q^{-1},2^r) + 
T_{2}(Q^{-1} ,2^r)\big].
$$
Hence combining
Lemma \ref{lem: N1 after stationary phase and geom sum} with the estimates \eqref{eq: N1 bd 2}, \eqref{eq Tterm2} and \eqref{eq alrd2} yields 
\begin{align*}
U &\ll_A
 \sum_{r\in \mathscr{G}(\delta, Q, \ell)}\frac{2^{(\frac{d}{2}-\1) \ell (n-1)} }
{1+ (2^r\delta)^{A}}
Q^{\frac{n+1}{2}} \delta 2^{-r \frac{n-1}{2}}
\Big(
 2^{-\ell \frac{d}{2}(n-1)}2^{\ell d}
\sum_{s\leq 2} 
T_{s}(Q^{-1} ,2^r)\\
&+ Q^{1-n} (1+ 2^{-(\d1)\ell}2^{r})^{n-1}\Big)
+2^{-\ell(n-1-d)}
(\delta Q^{n} + B(\delta,Q)).\nonumber
\end{align*}

Hence we infer
\begin{align}\label{eq: sum N1 over r when d small dual 2}
 U & \ll 
2^{(\frac{d}{2}-\1) \ell (n-1)} 
Q^{\frac{n+1}{2}} \delta 
 \Big(
 2^{-\ell \frac{d}{2}(n-1)}2^{\ell d}
\sum_{s \leq 2}
 \delta^{\frac{n-1}{2}}
T_{s}(\frac{1}{Q},
\frac{1}{\delta})
+ Z
\Big) \\
& +2^{-\ell(n-1-d)}
\Big(\delta Q^{n} 
 + B(\delta,Q)\Big),\nonumber
\end{align}
where 
\begin{equation}\label{def: Z}
Z:= Q^{1-n}  \sum_{R_{-} \leq r \leq R} 
\frac{2^{-r \frac{n-3}{2}}}
{1+ (2^r\delta)^{A}}
 (1+ 2^{-(\d1)\ell}2^{r})^{n-1}.
\end{equation}
Notice $ 2^{-\ell \frac{d}{2}(n-1)}2^{\ell d} 
= 2^{-\ell d \frac{n-3}{2}}$.
We claim that 
\begin{equation}\label{eq: Z at most sum}
2^{\ell d \frac{n-3}{2}} Z
\ll 
 \delta^{\frac{n-1}{2}}
T_{1}(\frac{1}{Q},
\frac{1}{\delta}).
\end{equation}
By using \eqref{eq: uncertaintly prin}
we infer
\begin{align}\label{eq: estimate Y}
Z &\ll Q^{1-n} 
\big(
R
+ \delta^{\frac{n-3}{2}} 
2^{- \ell (\d1)(n-1)}
\delta^{-(n-1)}
\big) \\
& \ll 
Q^{1-n} 
\big(
\log (Q/\delta)
+ 
2^{- \ell (\d1)(n-1)}
\delta^{-\frac{n+1}{2}} 
\big).\nonumber
\end{align}
Recall the equations \eqref{eq: T1 things} and 
\eqref{eq: T2 things}.
Now we verify \eqref{eq: Z at most sum}
by distinguishing two cases. Suppose, for the moment,
$2^{-\ell (\d1)(n-1)}\delta^{-\frac{n-1}{2}} \geq \log (Q/\delta)$.
In view of the identity 
$ \ell d \frac{n-3}{2} - \ell (\d1)(n-1) 
= \ell (n-1) (1-\frac{d}{2}) - \ell d\leq 0$, we conclude
$$
2^{\ell d \frac{n-3}{2}} Z
\ll 
Q^{1-n} 
2^{\ell d \frac{n-3}{2}}
2^{- \ell (\d1)(n-1)}
\delta^{-\frac{n+1}{2}} 
\leq 
Q^{1-n} 
\delta^{-\frac{n+1}{2}} \leq 
\delta^{\frac{n+1}{2}} T_1(\frac{1}{Q},\frac{1}{\delta}).
$$
Now suppose $2^{-\ell (\d1)(n-1)}\delta^{-\frac{n-1}{2}} < \log (Q/\delta)$.
Then 
$$
2^{\ell (d-1) \frac{n-3}{2}} Z
\ll 
\left(\frac{Q}{\delta}\right)^{\frac{(n-3)(d-1)}{d}+1} Q^{1-n}
\leq 
\left(\frac{Q}{\delta}\right)^{\frac{n-3}{2}+1} Q^{1-n}
$$
Hence, 
$$
2^{\ell (d-1) \frac{n-3}{2}} Z
\leq 
\delta^{-\frac{n-1}{2}} 
Q^{-\frac{n-1}{2}}
\stackrel{\eqref{eq: T1 things}}{\leq }
\delta^{\frac{n-1}{2}} T_1(\frac{1}{Q},\frac{1}{\delta}).
$$
As a result \eqref{eq: Z at most sum} holds.
Observe that $2^{(\frac{d}{2}-\1) \ell (n-1)} 
2^{-\ell \frac{d}{2}(n-1)}2^{\ell d}
 = 2^{-\ell (n-1-d)}$. 
Thus, combining \eqref{eq: sum N1 over r when d small dual 2} and 
\eqref{eq: Z at most sum}, 
yields
\begin{align*}
    U &\ll_A
    2^{-\ell (n-1-d)}
Q^{\frac{n+1}{2}} \delta
 \big(\delta^{-\frac{n+1}{2}} Q^{-1}
+
\delta^{-(\beta_{\kappa+1} -\frac{n-1}{2}+ n \varepsilon)}
\big) 
+2^{-\ell(n-1-d)}\big(\delta Q^{n} + B(\delta,Q)\big)\\
& \ll
2^{-\ell(n-1-d)}
 \big(\delta Q^n+
B(\delta,Q)+
Q^{\frac{n+1}{2}} 
\delta^{-(\beta_{\kappa+1}-\frac{n+1}{2}
+ n \varepsilon)}
\big).
\end{align*}
This implies \eqref{eq: bound on N1 sum 2} and completes the proof. 
\end{proof}

\subsection{The regime 
$d\geq n-1$}\label{subsec: d large}
Now we prove the counterparts 
of Theorem \ref{thm: og to dual sdeg} 
and Theorem \ref{thm: dual to og sdeg}
if $ d\geq n-1$.
\begin{thm}[Flat to Rough Manifold]
\label{thm: og to dual ldeg}
Suppose there exists a constant $C_{\kappa+1}>1$
so that for all $\delta\in(0,1/2)$ and $Q\geq1$, we have
\[
\mathfrak{N}^{P_{\kappa+1}}(\delta,Q, \ell, \1)
\leq 
 C_{\kappa+1} \Big(\delta Q^{n}
+ Q^{\beta_{\kappa+1}+n\varepsilon}
+ 
\left(\frac{\delta}{Q} \right)^{\frac{n-1}{d}}
Q^{n(1+\varepsilon)}
\Big)
\]
whenever  \ellponecond.
Then there exists a constant $C_{\kappa}>0$, depending on $C_{\kappa+1}$, $P_\kappa$ and $\varepsilon$, such that for all $Q\geq 1$ and $\delta\in (0, 1/2]$, we have
\begin{equation}
    \label{eq boots concl 3}
\mathfrak{N}^{P_{\kappa}}(\delta,Q, \ell, \d1)\leq 
C_{\kappa}2^{-\frac{d}{2}\ell(n-1)}2^{\ell(n-1)}\left(\delta Q^{n}
+Q^{\beta_\kappa'+n\varepsilon}\right),
\end{equation}
for all $$\ell\in \{0, 1, \ldots, \fL_{\d1}(\delta, Q)\},$$
with \[\beta_{\kappa}':=\max\left(n-\frac{n-1}{2\beta_{\kappa+1}-n+1}, n-\frac{2(n-1)}{d}\right).\]
\end{thm}
\begin{proof}
The method of proof is similar to the proof of Theorem 
\ref{thm: og to dual sdeg}. However,
the $\ell$-aspect of the estimates and the presence of the  term $(\frac{\delta}{Q})^{\frac{n-1}{d}}
Q^{n(1+\varepsilon)}$ in the assumption make certain arguments significantly different.
We shall briefly indicate how
derive the required estimates from the previously 
presented machinery. 

Using Proposition \ref{prop: gathering bounds} for $\p=\d1$, noting that $\frac{d}{\d1}\leq 2\leq n-1$, we again have that
for each $Q\geq 1$, $\delta\in (0, 1/2)$ and all \ellpdcond , $\fN^{P_{\kappa}}(\delta,Q,\ell,\d1)$
equals 
\begin{align}
\label{eq: b4 l slack}
& c_{\kappa} 2^{-\ell (\d1) (n-1)} \delta Q^{n}
+ \sum_{r\in \mathscr{G}(\delta, Q, \ell)} 
N_1^{P_{\kappa}}(\delta,Q,\ell,\d1, \r)\\
&
+ O(2^{-\ell (\d1) (n-1)} \log J  
+ 2^{-\ell((\d1)(n-1)-d)}
\delta Q^{n-1+\varepsilon})\nonumber.    
\end{align}
Since $n\geq 3$, we have the inequality
$$d-(\d1)(n-1)\leq (n-1)-\frac{d}{2}(n-1).$$
This observation lets us deduce that
$$2^{-\ell((\d1)(n-1)-d)}
\delta Q^{n-1+\varepsilon}\leq 2^{-\ell \frac{d}{2}(n-1)}2^{\ell(n-1)} \delta Q^{n}.$$
Thus, slackening the $\ell$-powers in \eqref{eq: b4 l slack}, we obtain
\begin{equation*}
 \fN^{P_{\kappa}}(\delta,Q,\ell,\d1)
 \ll 
2^{-\ell  \frac{d}{2}(n-1)}  2^{\ell(n-1)}
( \delta Q^{n}+ \log J)
+ \sum_{r\in \mathscr{G}(\delta, Q, \ell)} N_1^{P_{\kappa}}(\delta,Q,\ell,\d1, \r).
\end{equation*} 
For the ease of exposition, we define $U:=\sum_{r\in \mathscr{G}(\delta, Q, \ell)} N_1^{P_{\kappa}}(\delta,Q,\ell,\d1, \r)$.
Let $B(\delta,Q)$ be as in \eqref{def U and BdeltaQ}.
For each $Q\geq 1$, $\delta\in (0, 1/2)$ and for all $\ell\in\{0, 1, \ldots, \fL_{\d1}(\delta, Q)\}$, we now claim that
\begin{equation}\label{eq: bound on N1 sum 3}
U
\ll 
2^{-\frac{d}{2}\ell(n-1)}2^{\ell (n-1)}
\left(\delta Q^{\frac{n+1}{2}}
+B(\delta,Q)+Q^{\frac{n+1}{2}}
\delta^{\frac{n+1}{2}
-\beta_{\kappa+1}-n\varepsilon}+\Big( \frac{\delta}{Q} \Big)^{\frac{n-1}{d} 
-\frac{n-1}{2}} Q
\right).
\end{equation}
Like before, we postpone the proof of the above claim to later and proceed with the argument to establish \eqref{eq boots concl 3}. The term $\log J \leq \log Q + \log \frac{1}{\delta}$
can be readily absorbed into the right hand side of 
\eqref{eq: bound on N1 sum 3} to get 
\begin{align}
\label{boots intermediate 3}
\mathfrak{N}^{{P}_{\kappa}}(\delta,Q, \ell, \d1)
& \leq 
C_{\kappa}'
2^{-\frac{d}{2}\ell(n-1)}2^{\ell (n-1)}
 \Bigg(\delta 
Q^{n} +B(\delta,Q) \\
& +Q^{\frac{n+1}{2}}
\delta^{\frac{n+1}{2}-
\beta_{\kappa+1}-n\varepsilon}+
\Big( \frac{\delta}{Q} \Big)^{\frac{n-1}{d} 
-\frac{n-1}{2}} Q
\Bigg) \nonumber
\end{align}
for each $\delta\in(0,1/2)$, 
$Q\geq1$ and for all 
\ellpdcond ,
for a large constant $C_{\kappa}'$ depending on $P_{\kappa}$ and $\varepsilon$, satisfying $C_{\kappa}'>\const \cdot C_{\kappa+1}+c_\kappa$.
We set
\[\beta_{\kappa}:=n-\frac{n-1}{2\beta_{\kappa+1}-n+1}.\]
Fixing $Q$, $\delta$ and \ellpdcond \, now, we distinguish between two cases, depending on the size of $\delta$ with respect to the scale $Q^{\beta_{\kappa}-n}\in (Q^{-1}, 1]$.

In the case when $\delta\geq Q^{\beta_{\kappa}-n}$, we have
\begin{align*}
& B(\delta,Q)+Q^{\frac{n+1}{2}}\delta^{\frac{n+1}{2}-\beta_{\kappa+1}-n\varepsilon}+\Big( \frac{\delta}{Q} \Big)^{\frac{n-1}{d} -\frac{n-1}{2}} Q \\&\leq 2Q^{(n-\beta_{\kappa})(\frac{n-1}{2})}Q^{\frac{n-1}{2}}+Q^{\frac{n+1}{2}}Q^{(n-\beta_{\kappa})(-\frac{n+1}{2}+\beta_{\kappa+1}+n\varepsilon)}+Q^{(\beta_{\kappa}-n-1)(\frac{n-1}{d}-\frac{n-1}{2})}Q.  
\end{align*}
Arguing exactly as before in the proof of Theorem \ref{thm: og to dual sdeg}, we can conclude that 
\begin{align*}
\delta Q^n+B(\delta,Q)
+Q^{\frac{n+1}{2}}\delta^{\frac{n+1}{2}
-\beta_{\kappa+1}-n\varepsilon}
& \leq \delta Q^n+3Q^{n-1}+Q^{\beta_{\kappa}+n\varepsilon}\\
& \leq \delta Q^n+4Q^{\beta_{\kappa}+n\varepsilon}
\leq \delta Q^n+4Q^{\beta_{\kappa}'+n\varepsilon}.
\end{align*}
Further, since $Q^{\beta_\kappa-n-1}\geq Q^{-2}$ and the power $\frac{n-1}{d}-\frac{n-1}{2}\leq 0$, we have
\[Q^{(\beta_\kappa-n-1)(\frac{n-1}{d}-\frac{n-1}{2})}Q\leq Q^{2(\frac{n-1}{2}-\frac{n-1}{d})}Q=Q^{n-\frac{2(n-1)}{d}}\leq Q^{\beta_\kappa'}.\]
Combining the two estimates above and using \eqref{boots intermediate 3}, we get
$$\mathfrak{N}^{
{P}_{\kappa}}(\delta,Q, \ell, \d1)
\leq 
C_{\kappa}
2^{-\frac{d}{2}\ell(n-1)}2^{\ell (n-1)}
\left(\delta Q^n+Q^{\beta_{\kappa}'+n\varepsilon}\right),$$
which is the desired estimate with $C_\kappa=5C_{\kappa}'$ for $\delta \geq Q^{\beta_{\kappa}-n}$. 

When $\delta\leq Q^{\beta_{\kappa}-n}$, we use the monotinicity of $\mathfrak{N}^{P_{\kappa}}(\delta,Q, \ell, \d1)$ as a function of $\delta$ to bound 
$$
\mathfrak{N}^{P_{\kappa}}(\delta,Q, \ell, \d1)\leq  \mathfrak{N}^{P_{\kappa}}(Q^{\beta_{\kappa}-n},Q, \ell, \d1).$$
In the case when 
$2^{\ell d}\geq Q^{n-\beta_\kappa+1}$, 
noting that $2^{\ell(\d1)(n-1)}\leq 2^{\fL_{\d1}(d-1)(n-1)}\leq Q^{n-1}$ 
(due to \eqref{eq: 2^Lp}), we combine the lower bound on $2^{\ell d}$ with the trivial estimate to get
\begin{equation}
    \label{eq b inf trivial bd 1}
    \mathfrak{N}^{P_{\kappa}}(Q^{\beta_{\kappa}-n},Q, \ell, \d1)\leq B_\kappa 2^{-\ell(d-1)(n-1)}Q^n\leq B_\kappa 2^{-\ell\frac{d}{2}(n-1)}2^{\ell (n-1)} Q^{-(n-\beta_{\kappa}+1)(\frac{1}{2}+\frac{n-1}{d})}Q^n,
\end{equation}
where $B_\kappa$ is a positive constant depending on $P_\kappa$ but independent of $\delta, Q$ and $\ell$.
Since $n-\beta_\kappa+1\geq 1$ and
$$\frac{1}{2}+\frac{n-1}{d}\geq \min\left\{1, \frac{2(n-1)}{d}\right\},$$
we can conclude from \eqref{eq b inf trivial bd 1} that
\begin{align*}
\mathfrak{N}^{P_{\kappa}}(Q^{\beta_{\kappa}-n},Q, \ell, \d1)
& \leq B_\kappa 2^{-\ell\frac{d}{2}(n-1)}
2^{\ell (n-1)}\max\left\{Q^{n-1}, Q^{n-\frac{2(n-1)}{d}}
\right\}\\
& \leq B_\kappa 2^{-\ell\frac{d}{2}(n-1)}
2^{\ell (n-1)}Q^{\beta_\kappa'}.
\end{align*}
This implies \eqref{eq boots concl 3} for $\mathfrak{N}^{P_{\kappa}}(\delta,Q, \ell, \d1)$ when $2^{\ell d}\geq Q^{n-\beta_\kappa+1}$ (with $C_\kappa=B_\kappa$).

On the other hand, when $2^{\ell d}\leq Q^{n-\beta_\kappa+1}$, we use \eqref{boots intermediate 3} with $\delta$ replaced by $Q^{\beta_\kappa-n}$ to estimate
\begin{align*}
\mathfrak{N}^{P_{\kappa}}(\delta,Q, \ell, \d1)&\leq  \mathfrak{N}^{P_{\kappa}}(Q^{\beta_{\kappa}-n},Q, \ell, \d1)\\ &\leq C_{\kappa}'
2^{-\frac{d}{2}\ell(n-1)}2^{\ell (n-1)}\bigg(Q^{\beta_{\kappa}-n}
Q^{n}+ 2 Q^{(n-\beta_{\kappa})(\frac{n-1}{2})}Q^{\frac{n-1}{2}}
\\
&+Q^{\frac{n+1}{2}}Q^{(n-\beta_{\kappa})
(-\frac{n+1}{2}+\beta_{\kappa+1}+n\varepsilon)}
+Q^{(\beta_\kappa-n-1)(\frac{n-1}{d}
-\frac{n-1}{2})}Q\bigg)
\\
&\leq C_{\kappa}'
2^{-\frac{d}{2}\ell(n-1)}2^{\ell (n-1)}
\left(Q^{\beta_{\kappa}}+2 Q^{n-1}
+Q^{\beta_{\kappa}+n\varepsilon}+Q^{\beta_\kappa'}
\right)\\
& \leq 5C_{\kappa}'
2^{-\frac{d}{2}\ell(n-1)}2^{\ell(n-1)}Q^{\beta_{\kappa}'+n\varepsilon}.
\end{align*}
Thus we obtain \eqref{eq boots concl 3} again, with $C_\kappa=5C_{\kappa}'$.

To conclude the proof, it now remains to establish the

\textbf{Proof of Claim \eqref{eq: bound on N1 sum 3}:}
Using Lemma \ref{lem: N1 after stationary phase and geom sum}, we are ensured that
\begin{align}
    \label{eq: N1 bd 3}
    U & \ll \fX_\d1(\delta,Q,\ell)+
2^{(\frac{d}{2}-(\d1)) 
\ell (n-1)}
\delta Q^{\frac{n+1}{2}} 
\Big(\sum_{(i,r)} 
\frac{\fN^{{P}_{\kappa+1}}(2^{i+1}/Q, 
2^r,\ell, \1)}
{2^{iA}(1+ (2^r\delta)^{A}) 2^{r \frac{n-1}{2}}}\\
&+ 2^r Q^{1-n} 
\max(1, 2^{-\ell (d-(\d1))(n-1)}2^{r(n-1)})\Big)+ 
\sum_{r\in \mathscr{G}(\delta, Q, \ell)} 
    T(\delta,Q,\ell,\d1, r)\nonumber
\end{align}
where $T(\delta,Q,\ell,\d1, r)$
is as in \eqref{def T add} and the sum on the right 
is taken over all integer pairs
$(i,r) $ 
in $[0,\frac{\log Q}{100}]\times 
\mathscr{G}(\delta, Q, \ell)$
so that $\ell\leq \fL_{\1}(2^i/Q, 2^r)$.
Here $\fX_\d1(\delta,Q,\ell)$ is the contribution coming 
from $\sX=\sX(Q,\ell, \1)$ (see \eqref{def: Xpfrak}).
By slacking the $\ell$-powers via \eqref{def d estimate}, 
we observe that
\begin{equation}
    \label{eq Tterm3}
    \sum_{r\in \mathscr{G}(\delta, Q, \ell)} 
    T(\delta,Q,\ell,\d1, r)
    \ll 2^{-(\d1)\ell(n-1)}(\delta^{-1}Q^{-5})^{n-1}
    \ll 2^{-\frac{d}{2}\ell(n-1)}2^{\ell (n-1)} B(\delta,Q).
\end{equation}
Moreover, Proposition \ref{prop: pruning and stationary phase d large} tells us 
\begin{equation}
    \label{eq alrd3}
    \fX_\d1(\delta,Q,\ell)
 \ll 
2^{-\frac{d}{2}\ell(n-1)}
2^{\ell (n-1)}
(\delta^{-\frac{n-1}{2}}
+ \delta Q^{\frac{n+1}{2}}).
\end{equation}
Thus to establish the claim we now focus on bounding the two summands on the right in \eqref{eq: N1 bd 3}.
Recall the functions
$T_1(\delta,Q)$ and $ T_2(\delta,Q)$
from \eqref{def: T1 T2}.
Furthermore, we introduce
\begin{align*}
&T_3(\delta,Q) 
:=\left(\frac{\delta}{Q} \right)^{\frac{n-1}{d}}
Q^{n(1+\varepsilon)}
\end{align*}
in order to write the 
inequality from the assumption in the form
$$
\fN^{P_{\kappa+1}}(\delta, Q,\ell, \1) 
\ll 
\sum_{s\leq 3}
T_{s}(\delta ,Q).
$$
We observe that
$$
\sum_{ 0\leq i\leq \frac{\log Q}{100}} 
\frac{\fN^{P_{\kappa+1}}(\frac{2^i}{Q}, 2^r,\ell, \1)}{2^{Ai}}
\ll 
\sum_{s\leq 3}
T_{s}(\frac{1}{Q},2^r).
$$
Combining
Lemma \ref{lem: N1 after stationary phase and geom sum} with the estimates \eqref{eq: N1 bd 3}, \eqref{eq Tterm3} and \eqref{eq alrd3}, we conclude
that $  
\sum_{r\in \mathscr{G}
(\delta, Q, \ell)}
N_{1}^{P_{\kappa}}(\delta,Q, \ell, \d1, \r) $
is at most a constant times
\begin{align*}
& \frac{2^{(\frac{d}{2}-(\d1)) \ell (n-1)}}
{1+ (2^r\delta)^{A}}
Q^{\frac{n+1}{2}} \delta 2^{-r \frac{n-1}{2}} 
\Big(
\sum_{s\leq 3}
T_{s}(\frac{1}{Q},2^r)  
+ 2^r Q^{1-n} (1+ 2^{-\ell}2^{r})^{n-1}\Big)\\
& +2^{-\frac{d}{2}\ell(n-1)}
2^{\ell (n-1)}
(\delta^{-\frac{n-1}{2}}
+ \delta Q^{\frac{n+1}{2}}) + 2^{-\frac{d}{2}\ell(n-1)}2^{\ell (n-1)} B(\delta,Q).\nonumber
\end{align*}
Notice 
$$
2^{-r\frac{n-1}{2}} 
T_{3}(\frac{1}{Q},2^r)
= 2^{-r\frac{n-1}{2}}   \Big( \frac{Q^{-1}}{2^r} 
\Big)^{\frac{n-1}{d}} 2^{rn}
= Q^{-\frac{n-1}{d}}
2^{r [ n - (n-1) (\frac{1}{2} + \frac{1}{d})]}
$$
is an expression which we can sum 
over $r \leq R$ 
by appealing to \eqref{eq: uncertaintly prin}.
Thus,
$$
\sum_{r\leq R}
2^{-r\frac{n-1}{2}} 
T_{3}(\frac{1}{Q},2^r)
\ll 
\delta^{\frac{n-1}{2}}  
\Big( \frac{\delta}{Q}\Big)^{\frac{n-1}{d}} \delta^{-n}
= \Big( \frac{\delta}{Q} 
\Big)^{\frac{n-1}{d}} \delta^{-\frac{n+1}{2}}.
$$
By summing this quantity over 
$r\in \sG(\delta, Q, \ell)$,
we infer (by arguing as in the proof of 
Theorem \ref{thm: og to dual sdeg}) that
\begin{align*}
U \ll 
2^{-\frac{d}{2}\ell(n-1)}2^{\ell(n-1)}
& \Bigg[
Q^{\frac{n+1}{2}} \delta
\bigg(\delta^{\frac{n-1}{2}}
\sum_{s\leq 3}
T_{s}(\frac{1}{Q},\frac{1}{\delta})
+Y 
\bigg)+
B(\delta,Q)
+ \delta Q^{\frac{n+1}{2}}
\Bigg],
\end{align*}
where $Y$ is as in \eqref{def: Y term}.
The proof of the bound \eqref{eq: Y at most T1 T2} 
implies, in particular,
$ Y \ll 
\delta^{\frac{n-1}{2}}
T_{1}(\frac{1}{Q},\frac{1}{\delta})$
and hence $Y$ can be absorbed 
into the second and third 
term on the right hand side. 
To simplify the expressions, we notice
$$
Q^{\frac{n+1}{2}} \delta \cdot 
\delta^{\frac{n-1}{2}}
T_{3}(\frac{1}{Q},\frac{1}{\delta})
= 
Q^{\frac{n+1}{2}} \delta
\cdot 
\Big( \frac{\delta}{Q} 
\Big)^{\frac{n-1}{d}}  \delta^{-\frac{n+1}{2}}
= 
\Big( \frac{\delta}{Q} \Big)^{\frac{n-1}{d} 
-\frac{n-1}{2}} Q.
$$
Therefore $U$
is at most a constant times
\begin{align*}
2^{-\frac{d}{2}\ell(n-1)}
2^{\ell(n-1)}
\Bigg(
 \delta Q^{\frac{n+1}{2}}
+B(\delta,Q)
+Q^{\frac{n+1}{2}}
\delta^{\frac{n+1}{2}-\beta_{\kappa+1}-n\varepsilon}
 +
\Big( \frac{\delta}{Q} \Big)^{\frac{n-1}{d} 
-\frac{n-1}{2}} Q
\Bigg).
\end{align*}
This establishes \eqref{eq: bound on N1 sum 3} and finishes the proof.
\end{proof}
We now turn to the bootstrapping procedure for the flat manifold. 
In this case, the value of $\beta$ remains stationery (we are not able to obtain an improved $\beta_{\kappa}$ from $\beta_{\kappa+1}$). However, this will not be a problem as we shall carry out out the induction procedure for the flat and rough hypersurfaces in tandem (as alternate steps). Since the latter always leads to an improvement, we shall still be able to bring down the exponent $\beta_{\kappa}$ in the error term to the required value in a constant many number of steps (depending on $\varepsilon$).
\begin{thm}  [Rough to Flat Manifold] 
\label{thm: dual to og ldeg}
Suppose that for all $Q\geq1$ and $\delta\in(0, 1/2)$,
\[
\mathfrak{N}^{{P}_{\kappa+1}}(\delta,Q, \ell, \d1)\leq 2^{-\frac{d}{2}\ell(n-1)}2^{\ell (n-1)}
\Big(c_{\kappa+1}\delta Q^{n}
+C_{\kappa+1}
Q^{\beta_{\kappa+1}+n\varepsilon}\Big)
,
\]
  holds true for all \ellpdcond .
Then there exists a constant $C_{\kappa}>0$, depending on $C_{\kappa+1}$, $P_\kappa$ and $\varepsilon$, such that for all $Q\geq 1$ and $\delta\in (0, 1/2]$, we have
\begin{equation}
\label{eq concl boots 4}
\mathfrak{N}^{P_{\kappa}}(\delta,Q, \ell, \1)\leq C_{\kappa}\left(2^{-\ell(n-1)}\delta Q^{n}
+Q^{\beta_{\kappa+1}+3n\varepsilon}+\left(\frac{\delta}{Q}\right)^{\frac{n-1}{d}}Q^{n(1+\varepsilon)}\delta^{-2\varepsilon(n-1)}\right).   
\end{equation}
whenever \ellponecond, and 
$$\beta_{\kappa+1}\geq n-\frac{2(n-1)}{d}.$$
\end{thm}
\begin{proof}
The proof is rather similar to the proof of Theorem 
\ref{thm: dual to og sdeg}, and we describe it briefly. Proposition \ref{prop: gathering bounds} for $d\geq n-1$
tells us that for each $Q\geq 1$, $\delta\in (0, 1/2)$ and all $\ell \in \{0, 1, \ldots, \fL_{\1}(\delta, Q)\}$, $\fN^{P_\kappa}(\delta,Q,\ell,\1)$
equals 
$$
c_{\kappa} 2^{-\ell (n-1)} \delta Q^{n}
+ \sum_{r\in \mathscr{G}(\delta, Q, \ell)} N_1^{P_{\kappa}}(\delta,Q,\ell,\1, \r)
+ O\left(2^{-\ell (n-1)} \log J 
+ \left(\frac{\delta}{Q}\right)^{\frac{n-1}{d}}Q^{n(1+\varepsilon)}\right).  
$$
Let us put  
$U:=\sum_{r\in \mathscr{G}(\delta, Q, \ell)} 
N_1^{P_{\kappa}}(\delta,Q,\ell,\1, \r)$, for ease of exposition.
Recall the abbreviation 
$B(\delta, Q)$ from \eqref{def U and BdeltaQ}.
We claim that for each $Q\geq 1$, each 
$\delta\in (0, 1/2)$ and for all 
$\ell\in \{0, 1, \ldots, \fL_{1}(\delta, Q)\}$, it is true that
\begin{equation}\label{eq: bound on N1 sum 4}
U
\ll  B(\delta,Q) + Q^{\frac{n+1}{2}}
\delta^{\frac{n+1}{2}
-\beta_{\kappa+1}-n\varepsilon}+\Big(\frac{\delta}{Q}
\Big)^{\frac{n-1}{d}}
Q^n\delta^{-2\varepsilon(n-1)}.
\end{equation}
We again assume claim \eqref{eq: bound on N1 sum 4} for now and see how it implies \eqref{eq concl boots 4}. The term 
$\log J \leq \log Q + \log \frac{1}{\delta}$
can be absorbed into the right hand side of 
\eqref{eq: bound on N1 sum 4} to get
\begin{align}
\label{boots intermediate 4}
\mathfrak{N}^{P_{\kappa}}(\delta,Q, \ell, \1)
&\leq C_{\kappa}'\Bigg(2^{-\ell(n-1)}\delta 
Q^{n}+ B(\delta,Q)
\\ 
& +Q^{\frac{n+1}{2}}\delta^{\frac{n+1}{2}-\beta_{\kappa+1}-n\varepsilon}+
\left(\frac{\delta}{Q}\right)^{\frac{n-1}{d}}Q^{n(1+\varepsilon)}\delta^{-2\varepsilon(n-1)}\Bigg)\nonumber,
\end{align}
for each $Q\geq 1$, $\delta\in (0, 1/2)$ and all $\ell \in \{0, 1, \ldots, \fL_{\1}(\delta, Q)\}$, for a large constant $C_{\kappa}'$ depending on $P_{\kappa}$ and $\varepsilon$, satisfying $C_{\kappa}'>\const \cdot C_{\kappa+1}+c_\kappa$.

Fixing $Q\geq 1$, $\delta\in (0, 1/2)$ and $\ell \in \{0, 1, \ldots, \fL_{\1}(\delta, Q)\}$, we distinguish two cases, depending on the size of $\delta$ with respect to the scale $Q^{-1}$. 
For $\delta\geq Q^{-1}$, we can directly bound
\[ 
B(\delta,Q)
+Q^{\frac{n+1}{2}}\delta^{\frac{n+1}{2}
-\beta_{\kappa+1}-n\varepsilon}\leq 2Q^{n-1}
+ Q^{\beta_{\kappa+1}+n\varepsilon}
\leq 3Q^{\beta_{\kappa+1}+n\varepsilon}.\]
Combining the above with \eqref{boots intermediate 4} implies \eqref{eq concl boots 4} with $C_\kappa=3C_\kappa'$.

We now assume that $\delta\leq Q^{-1}$. We can use the monotinicity of $\mathfrak{N}^{{P}_{\kappa}}(\delta,Q, \ell, \1)$ as a function of $\delta$ to bound 
$$
\mathfrak{N}^{{P}_{\kappa}}(\delta,Q, \ell, \1)\leq  \mathfrak{N}^{{P}_{\kappa}}(Q^{-1},Q, \ell, \1).$$
In the case when $2^{\ell d}\geq Q^{2}$, we use the trivial bound to obtain
\begin{equation*}
   \mathfrak{N}^{{P}_{\kappa}}(Q^{-1},Q, \ell, \1)\leq B_\kappa 2^{-\ell(n-1)}Q^n\leq B_\kappa Q^\frac{2(n-1)}{d}Q^n=B_\kappa Q^{n-\frac{2(n-1)}{d}}\leq B_\kappa Q^{\beta_{\kappa+1}}.
\end{equation*}
Here $B_\kappa$ is a positive constant depending on $P_\kappa$ but independent of $\delta, Q$ and $\ell$. This is the required estimate \eqref{eq concl boots 4} with $C_{\kappa}'=B_\kappa$.

Finally, when $2^{\ell d}\leq Q^{2}$, we use \eqref{boots intermediate 4} with $\delta=Q^{-1}$ to estimate
\begin{align*}
\mathfrak{N}^{P_{\kappa}}(\delta,Q, \ell, \1)
&\leq \mathfrak{N}^{P_{\kappa}}(Q^{-1},Q, \ell, \1)
\\
& \leq C_{\kappa}'\left(2^{-\ell(n-1)}
2Q^{n-1}+Q^{n-1}+Q^{\beta_{\kappa+1}+n\varepsilon}
+Q^{n-\frac{2(n-1)}{d}+3n\varepsilon}\right)\\
& \leq 5 C_{\kappa}' Q^{\beta_{\kappa+1}+3n\varepsilon}.
\end{align*}
This implies the desired conclusion for $\mathfrak{N}^{{P}_{\kappa}}(\delta,Q, \ell, \1)$ (with $C_\kappa=5C_{\kappa}'$) if $2^{\ell d}\geq Q^{2}$.
It now remains to give the 

\textbf{Proof of Claim \eqref{eq: bound on N1 sum 4}:}
Lemma \ref{lem: N1 after stationary phase and geom sum} implies
\begin{align}
    \label{eq: N1 bd 4}
  U   \ll \fX_\1(\delta,Q,\ell) & +
2^{(\frac{d}{2}-(\1)) 
\ell (n-1)}
\delta Q^{\frac{n+1}{2}} 
\Big(\sum_{(i,r)} 
\frac{\fN^{{P}_{\kappa+1}}(2^{i+1}/Q, 
2^r,\ell, \d1)}
{2^{iA}(1+ (2^r\delta)^{A})}
2^{-r \frac{n-1}{2}} \\
& + 2^r Q^{1-n} \max(1, 2^{-\ell (d-(\d1))(n-1)}2^{r(n-1)})\Big)
+ \sum_{r\in \mathscr{G}(\delta, Q, \ell)} 
T(\delta,Q,\ell,\1, r).\nonumber
\end{align}
where $T(\delta,Q,\ell,\1, r)$ 
is as in \eqref{def T backsl 1}, and 
the sum on the right 
is taken over all integer pairs
$(i,r) $ 
in $[0,\frac{\log Q}{100}]\times 
\mathscr{G}(\delta, Q, \ell)$
so that $\ell\leq \fL_{\d1}(2^i/Q, 2^r)$.
Here $\fX_\1(\delta,Q,\ell)$ is the contribution 
coming from $\sX=\sX(Q,\ell, \d1)$ (see \eqref{def: Xpfrak}).
Notice that 
\begin{equation}
    \label{eq Tterm4}
    \sum_{r\in \mathscr{G}(\delta, Q, \ell)} 
    T(\delta,Q,\ell,\1, r)
\ll
2^{-\ell (n-1)} 
Q^{-5(n-1)}\delta^{-(n-1)}
\leq  B(\delta,Q).
\end{equation}
Moreover, Proposition \ref{prop: pruning and stationary phase d large} implies
\begin{equation}
    \label{eq alrd4}
    \fX_\1(\delta,Q,\ell)
 \ll 
\Big(\frac{\delta}{Q}
\Big)^{\frac{n-1}{d}}Q^{n} 
\delta^{-2\varepsilon(n-1)} + 
Q^{\frac{n-1}{2}}
\delta^{-\frac{n-1}{2}}.
\end{equation}
Thus to establish the claim we focus on bounding the two summands in the bracket on the right hand side in \eqref{eq: N1 bd 4}.
The functions
\begin{align}
&T_1(\delta,Q) 
:=\delta Q^{n}, \qquad
\mathrm{and} \qquad
T_2(\delta,Q) :=
 Q^{\beta_{\kappa+1} + n \varepsilon}\nonumber
\end{align}
allow us to write the inequality from the assumption in the form
$$
\fN^{P_{\kappa+1}}(\delta, Q,\ell, \d1) 
\ll 
2^{-\frac{d}{2}\ell(n-1)}2^{\ell (n-1)}
\big[
T_{1}(\delta ,Q) + 
T_{2}(\delta ,Q)\big].
$$
A quick inspection shows that if $A>0$
is chosen sufficiently large, then
$$
\sum_{ 0\leq i\leq \frac{\log Q}{100}} 
\frac{T_{s}(\frac{2^i}{Q},2^r,\ell)}{2^{Ai}}
\ll T_{s}(\frac{1}{Q},2^r)
$$
for any choice of $s=1,2$.
As a result, for each $r\in \mathscr{G}(\delta, Q, \ell)$, we have
$$
\sum_{ \substack{0\leq i\leq \frac{\log Q}{100}\\ 
\fL_{\d1}(2^i/Q, 2^r)\geq \ell}} 
\frac{\fN^{P_{\kappa+1}}(\frac{2^i}{Q}, 2^r,\ell, \d1)}{2^{Ai}}
\ll 
2^{-\frac{d}{2}\ell(n-1)}2^{\ell (n-1)}
\big[
T_{1}(Q^{-1},2^r) + 
T_{2}(Q^{-1} ,2^r)\big].
$$
Plugging the above, \eqref{eq Tterm4} and \eqref{eq alrd4} into \eqref{eq: N1 bd 4}
yields
\begin{align}\label{eq: N1 intermediate bound 4}
 U  &\ll_A
\sum_{r\in \mathscr{G}(\delta, Q, \ell)}
\frac{2^{(\frac{d}{2}-\1) \ell (n-1)} }
{1+ (2^r\delta)^{A}}
Q^{\frac{n+1}{2}} \delta 2^{-r \frac{n-1}{2}}
\Big(
2^{-\frac{d}{2}\ell(n-1)}2^{\ell (n-1)}
\sum_{s\leq 2} 
T_{s}(Q^{-1} ,2^r)\\
&+ 2^r  Q^{1-n} (1+ 2^{-(\d1)\ell}2^{r})^{n-1}\Big)+
\Big(\frac{\delta}{Q}
\Big)^{\frac{n-1}{d}}Q^{n} 
\delta^{-2\varepsilon(n-1)} + 
B(\delta,Q).\nonumber\end{align}
Recall \eqref{eq: sum over r} holds in the current 
context as well.
Combining \eqref{eq: sum over r}
and \eqref{eq: N1 intermediate bound 4} implies
\begin{align}\label{eq: sum N1 over r when d small dual}
 U  &\ll 
2^{(\frac{d}{2}-\1) \ell (n-1)} 
Q^{\frac{n+1}{2}} \delta 
 \Big(
 2^{-\ell \frac{d}{2}(n-1)}2^{\ell(n-1)}
\sum_{s \leq 2}
 \delta^{\frac{n-1}{2}}
T_{s}(\frac{1}{Q},
\frac{1}{\delta})
+ Z\Big)
\\&  + \Big(\frac{\delta}{Q}
\Big)^{\frac{n-1}{d}}Q^{n} 
\delta^{-2\varepsilon(n-1)} 
+ B(\delta,Q).\nonumber
\end{align}
Here $Z$ is as in \eqref{def: Z}; that is,
\begin{equation*}
Z= Q^{1-n}  \sum_{R_{-} \leq r \leq R} 
\frac{2^{-r \frac{n-3}{2}}}
{1+ (2^r\delta)^{A}}
 (1+ 2^{-(\d1)\ell}2^{r})^{n-1}.
\end{equation*}
Our aim is to show that $Z$ can be absorbed 
into the other term in the bracket 
on the right hand side of 
\eqref{eq: sum N1 over r when d small dual}. In particular, we shall establish that
\begin{equation}\label{eq: Z bound}
   2^{\ell \frac{d}{2}(n-1)}2^{\ell(n-1)} Z \ll 
 \delta^{\frac{n-1}{2}}
T_{1}(\frac{1}{Q},
\frac{1}{\delta}).
\end{equation}
To show \eqref{eq: Z bound}, using \eqref{eq: uncertaintly prin}, we infer
\begin{equation*}
Z \ll Q^{1-n} 
\big(
R
+ \delta^{\frac{n-3}{2}} 
2^{- \ell (\d1)(n-1)}
\delta^{-(n-1)}
\big)
\ll 
Q^{1-n} 
\big(
\log (Q/\delta)
+ 
2^{- \ell (\d1)(n-1)}
\delta^{-\frac{n+1}{2}} 
\big).
\end{equation*}
Suppose $\delta^{-\frac{n-1}{2}} 2^{-\ell(n-1) (\d1)}\geq \log (Q/\delta)$.
Then 
\begin{equation*}
2^{\ell \frac{d}{2}(n-1)}2^{\ell(n-1)}  Z 
\ll 
Q^{1-n} 
2^{- \ell \frac{\d1}{2}(n-1)}
\delta^{-\frac{n+1}{2}}
\leq \delta^{\frac{n-1}{2}} 
T_1 (\frac{1}{Q},\frac{1}{\delta})
\end{equation*}
where we used \eqref{eq: T1 things} in the last step.
Next suppose $\delta^{-\frac{n-1}{2}} 
2^{-\ell(n-1) (\d1)}< \log (Q/\delta)$.
Then, in view of $2^{\ell d} \ll 
Q/\delta$, we conclude that
\begin{align*}
2^{\ell (\frac{d}{2}-1)(n-1)}Z
& \ll 2^{\ell (\frac{d}{2}-1)(n-1)} 
\log \left(\frac{Q}{\delta}\right)
Q^{-(n-1)}
\ll 
\Big(\frac{Q}{\delta}\Big)^{(n-1)(\frac{1}{2}-\frac{1}{d}+o(1))} 
Q^{-(n-1)} \\
& \leq Q^{-\frac{n-1}{2}} \delta^{-\frac{n-1}{2}}
\leq \delta^{\frac{n-1}{2}} 
T_1 (\frac{1}{Q},\frac{1}{\delta})
\end{align*}
where we used $n\geq 3$ in the last step.
All in all, we have proved \eqref{eq: Z bound}.

Plugging \eqref{eq: Z bound}
into \eqref{eq: sum N1 over r when d small dual}, we get
\begin{align*}
U  & \ll 
Q^{\frac{n+1}{2}} \delta^{\frac{n+1}{2}}
 \Big(
 Q^{-1} \delta^{-n} +
\delta^{-(\beta_{\kappa+1} + n \varepsilon)}
\Big) +\Big(\frac{\delta}{Q}
\Big)^{\frac{n-1}{d}}Q^{n} 
\delta^{-2\varepsilon(n-1)} + 
B(\delta,Q) \\
& \ll
B(\delta,Q)
+
Q^{\frac{n+1}{2}} 
\delta^{-(\beta_{\kappa+1}-\frac{n+1}{2}
 + n \varepsilon)} +\Big(\frac{\delta}{Q}
\Big)^{\frac{n-1}{d}}Q^{n} 
\delta^{-2\varepsilon(n-1)}.
\end{align*}
Since this is exactly the 
bound \eqref{eq: bound on N1 sum 4},
the proof is complete.
\end{proof}

\section{Proof of Upper Bounds in Theorem \ref{thm: main smoothed}}\label{sec: upper bounds}
In this section, we use the mechanism developed in the last sections to establish upper bounds on the rational point count near hypersurfaces with homogeneous parametrizing functions. 
Recall that throughout 
our discussion from \S\ref{sec: preliminaries} on wards, 
we have fixed a function 
$f=f_{\1}\in 
\mathscr{H}_d^{
\mathbf{0}}(\mathbb{R}^{n-1})$, 
of homogeneity $d\geq 2$. 
We also recall its 
Legendre dual 
$f_{\d1}\in 
\mathscr{H}_{d/(d-1)}^{\mathbf{0}}
(\sD_{\d1})$, 
as defined in 
\S\ref{subsec hom func duality}, 
of homogeneity 
$\frac{d}{\d1}\in (1, 2]$. 
Finally, let $\varepsilon>0$ be a small fixed 
(as in 
Theorem \ref{thm: main smoothed}).
We now establish that there exists 
$k>0$, depending only on $n$, such that 
for  any $Q\geq 1$, we have
\begin{equation}
    \label{eq up bd f}
    N_{f_\1}(\delta,Q) \ll 
   \delta Q^n +
\Big(\frac{\delta}{Q}\Big)^{
\frac{n-1}{d}} Q^{n+k \varepsilon}
    + Q^{\max(n-1,n-\frac{2(n-1)}{d})+k\varepsilon},
\end{equation}
and 
\begin{equation}
\label{eq up bd fd}
 N_{f_{\d1}}(\delta,Q) \ll 
     \delta Q^n +
     Q^{\max(n-1,n-\frac{2(n-1)}{d})+k\varepsilon}
\end{equation}
uniformly in $\delta \in 
(Q^{\varepsilon-1},1/2)$.
The proof of the two bounds above shall proceed hand in hand.
\begin{remark}
The above estimates are enough to establish the desired upper bounds in Theorem \ref{thm: main}.  
This is because of the following. Firstly observe that 
the map $[2, \infty) \ni d\mapsto \frac{d}{d-1} 
\in(1,2]$ is a bijection.  
Next, for any fixed function 
$g_{\d1}\in \mathscr{H}_{d/(d-1)}^{\mathbf{0}}
(\mathbb{R}^{n-1})$, we can proceed as in \S\ref{sec: preliminaries}: first localizing to a small ball around $\bzero$ and then to a domain $\mathscr{D}_{\d1}$ which can be expressed as a union of dyadically scaled balls approaching the origin, such that $g_{\d1}$ is a diffeomorphism on $\mathscr{D}_{\d1}$. One can then define its Legendre dual $g_{\1}$ (on the dual domain $\mathscr{D}_{\1}$) which is homogeneous of degree $d$.
Replacing $f_{\p}$ by $g_{\p}$ for $\p\in \{\1, \d1\}$, we can then conclude the estimates \eqref{eq up bd f} and \eqref{eq up bd fd}; in particular, the latter implies the desired upper bound for $g_{\d1}$ by a compactness argument. 
\end{remark}
We now embark on proving \eqref{eq up bd f} and \eqref{eq up bd fd}. Let $K:=K(\varepsilon)\in \mathbb{Z}_{\geq 1}$ be a fixed parameter depending only on $\varepsilon$ (to be specified at the end). 
Recall the weight $\varrho$ introduced in \S \ref{subsec local}, supported in the ball $\mathscr{B}(\bx_{\1}, \tfrac{\epsilon_{\bx_{\1}}}4)$
and identically equal to one on the ball $\mathscr{B}(\bx_{\1}, \tfrac{\epsilon_{\bx_{\1}}}8)$. For $\p\in\{\1, \d1\}$ and $\kappa\in \{1,\ldots, K\}$, let $P_\kappa^{[\p]}$ be the weights (supported in $\sW_{\p}$), defined recursively in \eqref{eq: def kappa weight} and \eqref{eq: def dual kappa weight}, with $\varrho_0=\varrho$ and $\varrho_\kappa$ as defined in \eqref{eq varrho}.

For $Q\geq 1$, $\delta\in (0, 1/2)$ and $\ell\in \mathbb{Z}_{\geq 0}$, we define
\[\mathfrak{N}_{\kappa}(\delta,Q,\ell, \p):=\mathfrak{N}_f^{P_{\kappa}}(\delta,Q,\ell, \p),\qquad \mathfrak{N}_{\kappa}(\delta,Q,\ell, \d1):=\mathfrak{N}_{f_{\d1}}^{{P_{\kappa}}}(\delta,Q, \ell, \d1).\]
We shall use $k$ to denote a positive constant depending only on $n$, which may change from instance to instance
(e.g. $Q^{2k}\ll Q^{k}$).
We start with the trivial estimate: for each $Q\geq 1$, each $\delta\in (0, 1/2)$ and all $\ell \in \{0, 1, \ldots, \fL_{\1}(\delta, Q)\}$, we have
$$\mathfrak{N}_K(\delta, Q, \ell, \1)\leq C_{K}2^{-\ell(n-1)}Q^n.$$
Here $C_K$ depends on $P_K$ and therefore on $\varepsilon$. We now distinguish between two regimes.

\subsection{The regime $d<n-1$} In this regime, we use Theorem \ref{thm: og to dual sdeg}, with $\kappa=K-1$, $\beta_{\kappa+1}=\beta_K=n$, to conclude that for all $Q\geq 1$ and $\delta\in (0, 1/2]$, we have
$$\mathfrak{N}_K(\delta,Q,\ell, \d1)\leq C_{K-1}  2^{-\frac{d}{2}\ell(n-1)}2^{\ell d}\left(\delta Q^{n}
+Q^{\beta_{K-1}+n\varepsilon}\right), $$   
for all $$\ell\in \{0, 1, \ldots, \fL_{\d1}(\delta, Q)\},$$
where \[\beta_{K-1}=n-\frac{n-1}{n+1}.\]
Here the constant $C_{K-1}$ depends on $\varepsilon$ via $P_K^{[\p]}$ and $P_{K-1}^{[\p]}$.

Next, we feed the above estimate into Theorem \ref{thm: dual to og sdeg} with $\kappa=K-2$. This yields that for all $Q\geq 1$ and $\delta\in (0, 1/2]$, we have
$$
\mathfrak{N}_{K-2}(\delta,Q,\ell, \1)\leq 
C_{K-2}2^{-\ell(n-1-d)}\left(\delta Q^{n}
+Q^{\beta_{K-2}+n\varepsilon}\right),    
$$
for all  \ellponecond \,
with \[\beta_{K-2}:=n-\frac{n-1}{2\beta_{K-1}-n+1}.\]

With repeated applications of Theorems \ref{thm: og to dual sdeg} and \ref{thm: dual to og sdeg}, after $2i$ many steps, for all $Q\geq 1$ and $\delta\in (0, 1/2]$, we arrive at the estimates
\begin{equation}
    \label{eq kap est dual}
    {\mathfrak{N}}_{K-(2i-1)}(\delta,Q,\ell, \d1)\leq C_{K-(2i-1)}  2^{-\frac{d}{2}\ell(n-1)}2^{\ell d}\left(\delta Q^{n}
+Q^{\beta_{K-{(2i-1)}}+n\varepsilon}\right), 
\end{equation}
for all \ellpdcond; and
\begin{equation}
\label{eq kap est og}    
\mathfrak{N}_{K-2i}(\delta,Q,\ell, \1)\leq 
C_{K-2i}2^{-\ell(n-1-d)}\left(\delta Q^{n}
+Q^{\beta_{K-2i}+n\varepsilon}\right),
\end{equation}
for all \ellpdcond.
The constants $C_{K-(2i-1)}$ and $C_{K-2i}$ again depend on $\varepsilon$ (via $P_K^{[\p]}, \ldots, P_{K-(2i-2)}^{[\p]}$, and also $P_{K-(2i-1)}^{[\p]}$ in the case of the latter). The sequence $\{\beta_{\kappa}\}_{\kappa=K-2i}^{K}$ is recursively defined as $\beta_{\kappa}=n-\frac{n-1}{2\beta_{\kappa+1}-n+1}$. Rearranging, we have
$$\beta_{\kappa}-(n-1)=\frac{\beta_{\kappa+1}-(n-1)}{\beta_{\kappa+1}-\frac{n-1}{2}}.$$
Observe that
$$ \beta_{\kappa+1}-\frac{n-1}{2}\geq n-1-\frac{n-1}{2}=\frac{n-1}{2}.$$
and hence for $n\geq 4$, we conclude that
\begin{equation*}
    \beta_{\kappa}-(n-1)\leq \frac{2}{3}(\beta_{\kappa+1}-(n-1))\leq \left(\frac{2}{3}\right)^{\kappa}(\beta_K-(n-1))=\left(\frac{2}{3}\right)^{\kappa}.
\end{equation*}
For $n\geq 4$, we set $$K(\varepsilon):=\lfloor \log \varepsilon^{-1}\rfloor.$$
After $K$ many steps, \eqref{eq kap est dual} and \eqref{eq kap est og} give
\begin{equation}
\label{eq K est og n4}     
{\mathfrak{N}}_{1}(\delta,Q,\ell, \d1)\leq C_{1}(\varepsilon)  2^{-\frac{d}{2}\ell(n-1)}2^{\ell d}\left(\delta Q^{n}
+Q^{n-1}Q^{k\varepsilon}\right), 
\end{equation}
for all \ellpdcond, and
\begin{equation}
\label{eq K est dual n4}  
\mathfrak{N}_{0}(\delta,Q,\ell, \1)\leq 
C_{0}(\varepsilon)2^{-\ell(n-1-d)}\left(\delta Q^{n}
+Q^{n-1}Q^{k\varepsilon}\right)
\end{equation}
for all \ellponecond.
For $n=3$ though, the recursive relation
$$\beta_{\kappa}-2=\frac{\beta_{\kappa+1}-2}{\beta_{\kappa+1}-1}$$ gives
\begin{equation*}
\beta_\kappa=2+\frac{1}{K-{\kappa+1}}.
\end{equation*}
In other words, the sequence $\{\beta_\kappa\}$ converges at a much slower rate for $n=3$, and in this case we set $$K(\varepsilon)=\lfloor\varepsilon^{-1}\rfloor$$ to obtain
\begin{equation}
\label{eq K est og n3}
{\mathfrak{N}}_1(\delta,Q,\ell, \d1)\leq C_{1}(\varepsilon) 2^{-\frac{d}{2}\ell(n-1)}2^{\ell d}\left(\delta Q^{n}
+Q^2Q^{k\varepsilon}\right),     
\end{equation}
for all \ellpdcond, and 
\begin{equation}
\label{eq K est dual n3}
\mathfrak{N}_0(\delta,Q,\ell, \1)\leq 
C_{0}(\varepsilon)2^{-\ell(n-1-d)}\left(\delta Q^{n}
+Q^2Q^{k\varepsilon}\right),    
\end{equation}
for all \ellponecond.
Recall that
$$\mathfrak{N}_0(\delta,Q,\ell, \1):=\mathfrak{N}^{\varrho_{0}}(\delta,Q,\ell, \1)=\mathfrak{N}^{\varrho}(\delta,Q,\ell, \1),$$ 
and
$$\mathfrak{N}^{{\varrho}}(\delta,Q, \ell, \d1)=\mathfrak{N}_{0}(\delta,Q, \ell, \d1)\ll {\mathfrak{N}}_{1}(\delta,Q,\ell, \d1)=\mathfrak{N}^{{P_{1}}}(\delta,Q, \ell, \d1).$$
Thus, summing up in $\ell$ and using estimates 
\eqref{eq K est og n4} and 
\eqref{eq K est dual n4} 
(or \eqref{eq K est og n3} 
and \eqref{eq K est dual n3} for $n=3$) gives
$$\sum_{0\leq \ell\leq  \fL_{\1}(\delta, Q)}\mathfrak{N}^{\varrho}(\delta,Q,\ell, \1)\ll \delta Q^n+Q^{n-1+k\varepsilon} $$
and
$$\sum_{0\leq \ell\leq  \fL_{\d1}
(\delta, Q)}\mathfrak{N}^{{\varrho}}(\delta,Q, \ell, \d1)
\ll \delta Q^n+Q^{n-1+k\varepsilon}.$$
Again, the implicit constants above depend on $\varepsilon$.
As a final step, we combine the above bounds 
with the tail estimates in Lemma \ref{lem: tail terms} 
to get
\[
N_{f_1}(\delta,Q)
\ll
\delta Q^n+Q^{n-1+k\varepsilon} +\left(\frac{\delta}{Q}\right)^{\frac{(n-1)}{d}}Q^{n},
\]
and
\[
{N}_{f_{\d1}}(\delta,Q)
\ll
\delta Q^n+Q^{n-1+k\varepsilon} +\left(\frac{\delta}{Q}\right)^{\frac{(\d1)(n-1)}{d}}Q^{n},
\]
for all $Q\geq 1$ and $\delta\in (0, 1/2)$. Since 
$$\max\left[\left(\frac{\delta}{Q}\right)^{\frac{(n-1)}{d}},\left(\frac{\delta}{Q}\right)^{\frac{(\d1)(n-1)}{d}}\right]Q^n\leq \delta Q^n$$
for $n\geq 3$ and $2\leq d<n-1$, this proves \eqref{eq up bd f} and \eqref{eq up bd fd} in the current regime.

\subsection{The regime $d \geq n-1$} Like before, we set $$K(\varepsilon):=\begin{cases}
2\lfloor \log \varepsilon^{-1}\rfloor,\,& n\geq 4.\\
2\lfloor \varepsilon^{-1}\rfloor,\,& n=3.
\end{cases}$$
Further, let $\{\beta_\kappa\}_{\kappa=0}^K$ be the sequence defined recursively, with $\beta_K=n$, and
$$\beta_{\kappa}=\begin{cases}
n-\frac{n-1}{2\beta_{\kappa+1}-n+1},\, & \kappa \text{ is odd}.\\
\beta_{\kappa+1},\, & \kappa \text{ is even}.\\
\end{cases}$$

We now use Theorems \ref{thm: og to dual ldeg} and \ref{thm: dual to og ldeg} successively; and our bootstrapping procedure differs from the case $d < n-1$ in two key respects. Firstly, the improvement in the value of $\beta_\kappa$ only happens in alternate steps. Secondly, while the relations 
$$   \beta_{\kappa}-(n-1)\leq \frac{2}{3}(\beta_{\kappa+2}-(n-1))$$
for $n\geq 4$, and
$$\beta_\kappa=\begin{cases}
2+\frac{1}{K-{\kappa+1}},\, & \kappa \text{ is odd}.\\
2+\frac{1}{K-{\kappa+2}},\, & \kappa \text{ is even}.\\
\end{cases}$$
are true (so that the $\{\beta_\kappa\}_{\kappa=0}^K$ still converge to $n-1$), the condition $\beta_\kappa\geq n-\frac{2(n-1)}{d}$ (as opposed to $\beta_\kappa\geq n-1$) is baked into our argument and determines the number of steps in our induction procedure when $d>2(n-1)$. 

In the case when $(n-1)<d\leq 2(n-1)$, we have $$n-\frac{2(n-1)}{d}\leq n-1\leq \beta_\kappa,\, \qquad  0\leq \kappa\leq K.$$ The first application (of Theorem \ref{thm: og to dual ldeg}  with $\kappa=K-1$, $\beta_{\kappa+1}=\beta_K=n$, gives us that for all $Q\geq 1$ and $\delta\in (0, 1/2]$, we have
$${\mathfrak{N}}_{K-1}(\delta,Q,\ell, \d1)\leq C_{K-1}(\varepsilon) 2^{-\frac{d}{2}\ell(n-1)}2^{\ell (n-1)}\left(\delta Q^{n}
+Q^{\beta_{K-1}+n\varepsilon}\right), $$   
for all $$\ell\in \{0, 1, \ldots, \fL_{\d1}(\delta, Q)\},$$
with \[\beta_{K-1}=\max \left(n-\frac{n-1}{n+1}, n-\frac{2(n-1)}{d}\right).\]
As in the previous regime, $C_{K-1}$ depends on $\varepsilon$ via $P_K^{[\p]}$ and $P_{K-1}^{[\p]}$.
Feeding the above estimate into Theorems \ref{thm: dual to og ldeg} with $\kappa=K-2$, we derive the conclusion that for all $Q\geq 1$ and $\delta\in (0, 1/2]$, we have
$$
\mathfrak{N}_{K-2}(\delta,Q,\ell, \1)\leq C_{K-2}(\varepsilon)\left(2^{-\ell(n-1)}\delta Q^{n}
+Q^{\beta_{K-2}+3n\varepsilon}+\left(\frac{\delta}{Q}\right)^{\frac{n-1}{d}}Q^{n(1+\varepsilon)}\delta^{-2\varepsilon(n-1)}\right).  $$
whenever \ellponecond.
Thus, arguing as before, after $K$ many steps, we obtain
\begin{equation}
{\mathfrak{N}}_{1}(\delta,Q,\ell, \d1)\leq C_{1}(\varepsilon)  2^{-\frac{d}{2}\ell(n-1)}2^{\ell (n-1)}\left(\delta Q^{n}
+Q^{k\varepsilon}\right),
\label{eq int 1}
\end{equation}
for all \ellpdcond, and
\begin{align*}
\mathfrak{N}^{\varrho}
(\delta,Q,\ell, \1) & 
= \mathfrak{N}_{0}(\delta,Q,\ell, \1) \\ 
&\leq C_{0}(\varepsilon)\left(2^{-\ell(n-1)}\delta Q^{n}
+Q^{n-1+k\varepsilon}+\left(\frac{\delta}{Q}\right)^{\frac{n-1}{d}}Q^{n(1+\varepsilon)}\delta^{-2\varepsilon(n-1)}\right).   
\end{align*}
whenever\ellponecond.
Summing up the above equation in $\ell$ gives
$$\sum_{0\leq \ell\leq  
\fL_{\1}(\delta, Q)}\mathfrak{N}^{\varrho}(\delta,Q,\ell, \1)
\ll \delta Q^n+Q^{n-1+k\varepsilon}+\left(\frac{\delta}
{Q}\right)^{\frac{n-1}{d}}Q^{n}
(\delta^{-1}Q)^{k\varepsilon},$$
for all $Q\geq 1$ and $\delta\in (0, 1/2)$. 
Since we
operate in the range 
$\delta \in (Q^{\varepsilon-1},1/2)$,
the bound 
$$
(\delta^{-1}Q)^{k\varepsilon}
\ll Q^{k\varepsilon}
$$
is true. 
Hence, we conclude from the above and
Lemma \ref{lem: tail terms}, that
\begin{equation}
    \label{eq mid d flat}
    {N}_{f_{\1}}(\delta,Q) \ll
\delta Q^n+Q^{n-1+k\varepsilon} +\left(\frac{\delta}{Q}\right)^{\frac{n-1}{d}}Q^{n+k\varepsilon}
\end{equation}

This establishes \eqref{eq up bd f} for $n-1 \leq d\leq 2(n-1)$. 

For the rough case, we sum \eqref{eq int 1} in $\ell$ and use the estimate $\mathfrak{N}^{{\varrho}}(\delta,Q, \ell, \d1)=\mathfrak{N}_{0}(\delta,Q, \ell, \d1)\ll {\mathfrak{N}}_{1}(\delta,Q,\ell, \d1)$ to get 
$$\sum_{0\leq \ell\leq  \fL_{\d1}
(\delta, Q)}
\mathfrak{N}^{{\rho}}(\delta,Q, \ell, \d1)
\ll \delta Q^n+Q^{n-1+k\varepsilon} 
$$
for all $Q\geq 1$ and $\delta\in (0, 1/2)$, whenever $d>2$. Combining the above with the estimate for tail terms in Lemma \ref{lem: tail terms}, we conclude that
\[
{N}_{f_{\d1}}^{{\rho}}(\delta,Q)
\ll
\delta Q^n+Q^{n-1+k\varepsilon} +\left(\frac{\delta}{Q}\right)^{\frac{(\d1)(n-1)}{d}}Q^{n}.
\] 
Since  
$$\left(\frac{\delta}{Q}\right)^{\frac{(\d1)(n-1)}{d}}Q^n\leq \delta Q^n$$
for $n\geq 3$ and $d>2$, this establishes \eqref{eq up bd fd} in the case when $n-1\leq d\leq 2(n-1)$ and $d\neq 2$.

In the case when $d=2$ (and hence $n=3$), the negative power of $2^\ell$ in front of the first term in \eqref{eq int 1} vanishes. Thus summing it up in $\ell$ would give a logarithmic loss in $Q$. However, this is no issue as when $d=2$, we have $f,\widetilde{f}\in \sH_{2}^{\bzero}(\bRn)$. Therefore, replacing $f$ by $\tilde{f}$ in \eqref{eq mid d flat} gives the desired estimate in this case. (Or one can simply refer to \cite{Huang rational points} since this is the case of non-vanishing curvature!)

Finally, for $d>2(n-1)$, there exists an odd number $\kappa(d)$ with $0<\kappa(d)<K$ such that
$$\beta_{\kappa(d)}-(n-1)<n-\frac{2(n-1)}{d}<\beta_{\kappa(d)+1}-(n-1).$$ This means that at the $(K-\kappa(d))$-th step, Theorem \ref{thm: og to dual ldeg} with $\kappa=\kappa(d)$ gives
$$
{\mathfrak{N}}_{\kappa(d)}(\delta,Q,\ell, \d1)\leq 
C_{\kappa(d)}(\varepsilon)2^{-\frac{d}{2}\ell(n-1)}2^{\ell(n-1)}\left(\delta Q^{n}
+Q^{n-\frac{2(n-1)}{d}+n\varepsilon}\right),
$$
whenever $$\ell\in \{0, 1, \ldots, \fL_{\d1}(\delta, Q)\}.$$
The next application of Theorem \ref{thm: dual to og ldeg} yields
\begin{align*}
&\mathfrak{N}_{\kappa(d)-1}(\delta,Q,\ell, \1)\\&\leq C_{\kappa(d)-1}(\varepsilon)\left(2^{-\ell(n-1)}\delta Q^{n}
+Q^{n-\frac{2(n-1)}{d}+3n\varepsilon}+\left(\frac{\delta}{Q}\right)^{\frac{n-1}{d}}Q^{n(1+\varepsilon)}\delta^{-2\varepsilon(n-1)}\right)    
\end{align*}
for all \ellponecond.
Summing up in $\ell$, we obtain
$$\sum_{0\leq \ell\leq  
\fL_{\1}(\delta, Q)}\mathfrak{N}^{\varrho}(\delta,Q,\ell, \1)
\ll \delta Q^n+Q^{n-\frac{2(n-1)}{d}+k\varepsilon}+
\left(\frac{\delta}{Q}\right)^{\frac{n-1}{d}}Q^{n}
(\delta^{-1}Q)^{k\varepsilon}$$
and
$$\sum_{0\leq \ell\leq  
\fL_{\d1}(\delta, Q)}\mathfrak{N}^{{\varrho}}
(\delta,Q, \ell, \d1)\ll 
\delta Q^n+Q^{n-\frac{2(n-1)}{d}+k\varepsilon},
$$
for all $Q\geq 1$ and $\delta\in (0, 1/2)$, with the implicit constants depending on $\varepsilon$.
Combining the above with Lemma \ref{lem: tail terms}, we get
\[
{N}_{f_\1}(\delta,Q)
\ll
\delta Q^n+Q^{n-\frac{2(n-1)}{d}
+k\varepsilon} +\left(\frac{\delta}{Q}
\right)^{\frac{n-1}{d}-k\varepsilon}Q^{n},
\]
and
\[
{N}_{f_{d-1}}(\delta,Q)
\ll
\delta Q^n+Q^{n-\frac{2(n-1)}{d}
+k\varepsilon} +\left(\frac{\delta}{Q}
\right)^{\frac{(\d1)(n-1)}{d}}Q^{n}.
\]
Since  
$$
\left(\frac{\delta}{Q}\right)^{\frac{(\d1)(n-1)}{d}}
Q^n\leq \delta Q^n$$
for $n\geq 3$ and $d\geq 2$ and $\delta \in (Q^{\varepsilon-1},1/2)$, we derive the desired estimates \eqref{eq up bd f} and \eqref{eq up bd fd} also in the case $d>2(n-1)$.

\section{Proof of Lower Bounds in Theorem \ref{thm: main smoothed}}\label{sec: lower bounds}
Now we establish the lower bounds for 
Theorem \ref{thm: main smoothed}.
More concretely, we are concerned with
demonstrating 
$$
N_f(\delta,Q) \gg \delta Q^n + 
Q^n \Big(\frac{\delta}{Q}
    \Big)^{\frac{n-1}{d}}
$$
for $\delta \in 
(Q^{\varepsilon -1},1/2)$.
This bound follows 
at once from the following
two simple auxiliary results.
\begin{prop}\label{prop: fourier lower bound}
Let $n\geq 3$. 
Suppose $f\in \sH_d^{\bzero}(\bRn)$,
where $d>1$.
If $\varepsilon>0$, then
\begin{equation*}
    N_{f}(\delta,Q) \gg \delta Q^n.
\end{equation*}
uniformly in $Q\geq 1$, and 
$\delta\in 
(Q^{\varepsilon -1},1/2)$.
\end{prop}
Moreover, we require a lower bound 
that incorporates the `flat'-term.
The next proposition 
provides such a lower bound. 
In contrast to the previous proposition, 
it is based on the geometry of numbers
(which is why it is valid in a far wider 
range of $\delta$). In fact, the range 
for $\delta$ arising naturally from computing 
when the ($Q$-renormalized) 
volume of the Knapp Caps,
see \S\ref{sec: outline}, tends to infinity.
\begin{prop}\label{prop: geometry lower bound}
Let $n\geq 2$, and $1\leq c < n$ be integers.
Suppose 
$f_1,\ldots,f_c\in \sH_d^{\bzero}(\bR^{n-c})$,
where $d>1$.
Put $\bf:=(f_1,\ldots,f_c)$,
and $m:=n-c$.
If $\varepsilon>0$, then
    \begin{equation*}
    N_{\bf}(\delta,Q) \gg 
    Q^{m+1} \Big(\frac{\delta}{Q}
    \Big)^{\frac{m}{d}}
\end{equation*}
uniformly in $Q\geq 1$
and $\delta\in(Q^{\varepsilon-(d-1)},1/2)$.
\end{prop}

We begin with proving 
the last proposition first.
To this end, recall 
that $\sU_m$ denotes the unit ball 
in $\bR^m$. Let $V_m$ be 
the $m$-dimensional volume 
of $\sU_m$. Centering 
a box of unit-side length 
around every integer point shows
\begin{equation}\label{eq: integer points in ball}
    \# (\lambda \sU_m \cap \bZ^m)=
    \lambda^m V_m + O_m(\lambda^{m-1})
\end{equation}
as $\lambda \rightarrow \infty$.

\begin{proof}[Proof of Proposition
\ref{prop: geometry lower bound}]
Define the (Knapp cap inspired) set
$$
\sS:=
\Big(
\frac{\delta}{Z Q}\Big)^\frac{1}{d}
\cdot
\sU_{m}
\quad \mathrm{where}
\quad Z:=
\max_{i\leq c}
\max_{\bx \in \sU_m} 
\vert f_i(\bx)\vert.
$$
Let $i\leq c$. Notice 
$f_i(\bx)= \Vert \bx \Vert_2^d 
f_i(\Vert \bx \Vert_2^{-1}\bx)$
for non-zero $\bx\in\bR^m$.
So if $\bx \in \sS\setminus \{\bzero\}$
and $q/Q\in \supp(\om)$, then 
$$
\vert f_i(\bx) \vert \leq 
\frac{\delta}{Z Q}
\big\vert 
f_i\big(\frac{\bx}{\Vert \bx 
\Vert_2}\big) \big\vert 
\leq 
\frac{\delta}{Q}
\leq 
\frac{\delta}{q}.
$$
If $\bx=\bzero$ then
$\vert f_i(\bx) \vert= 0 
\leq \delta/q$ is also true.
Consequently, 
any $\ba/q\in \sS$
satisfies 
$ \vert qf_i(\ba/q) \vert \leq \delta $.
Therefore
$$
N_{\bf}(\delta,Q)\geq 
\sum_{q\in \bZ}
\om(\frac{q}{Q})
\#(\sS\cap q^{-1}\bZ^m).
$$
Observe that 
$\#(\sS\cap q^{-1}\bZ^m)=
\#(q \sS\cap \bZ^m)$
and $q \sS = \lambda \sU_m$,
where 
$$
\lambda:= q \delta^\frac{1}{d}
(Z Q)^{-1/d}=\delta^\frac{1}{d}
(Z q^{-d} Q)^{-1/d}. 
$$ 
In view of 
$\delta \geq Q^{\varepsilon-(d-1)}$
and $q\asymp Q$,
we conclude that 
$\lambda \gg Q^{\varepsilon/d}$. 
By taking \eqref{eq: integer points in ball}
into account we infer 
$ \#(\sS\cap q^{-1}\bZ^m)\gg \lambda^m$.
The implied constant is uniform 
in $q/Q\in \supp(\om)$. 
As $\lambda^m \asymp Q^m 
(\delta/Q)^{\frac{m}{d}}$,
summing the previous inequality over 
$q/Q\in \supp(\om)$ completes the proof.
\end{proof}
We proceed to prove Proposition \ref{prop: fourier lower bound}.
\begin{proof}[Proof of Proposition \ref{prop: fourier lower bound}]
Let $f \in \sH_d^{\bzero}(\bRn)$
where $d>1$.
For demonstrating
lower bounds one may 
remove any part
of the hypersurface 
as long as one can work
with the remaining part. 
In particular, $N_{f}(\delta,Q)\geq 
\fN_{f}(\delta,Q,0)$.
Note that $\ell=0$ corresponds to the piece of the hypersurface where the Gaussian curvature is uniformly bounded. Due to Theorem \ref{thm: Huang},
we know $\fN_{f}(\delta,Q,0)
\gg \delta Q^n$, for any 
$\delta\in (Q^{\varepsilon-1},1/2)$,
as required.
\end{proof}

\section{Closing Remarks, Open Questions, and Acknowledgements}\label{sec: final remarks}
In view of the results of this paper, 
studying rational points near 
hypersurfaces immersed 
by a homogeneous function
$f:\bR^{n-1} \rightarrow \bR$ 
becomes a topical question.
To explain further details
we write, as before, the hypersurfaces
in the normalised Monge form
\eqref{def: normalised Monge}.
Now a natural problem is the following.
\begin{problem}\label{prob: homogeneous}
Let $n\geq 3$ be an integer, and  
$f\in \sH_d^{\bzero}(\bRn)$.
For which $d\in (0,1]$
does \eqref{eq: main bounds 1}
hold true?
\end{problem}
While Problem \ref{prob: homogeneous}
asks to extend Theorem \ref{thm: main}
in the $d$-aspect, one can also wonder at what happens when one 
widens the class $\sH_d^{\bzero}(\bRn)$.
In particular, what happens when
we drop the assumption that $$\det\, H_f(\bx)=0
\,\text{ implies }\, \bx= \bzero$$ and consider all
$f\in \sH_d(\bRn)$? 
We formulate this as a separate problem.
\begin{problem}\label{prob: homogeneous general}
Let $n\geq 3$ be an integer, and $d>1$.
For which $f\in \sH_d(\bRn)$ 
does \eqref{eq: main bounds 1}
hold true?
\end{problem}

Let us briefly indicate what 
kind of difficulties 
arise when approaching
the above problems with our methods. 
For Problem \ref{prob: homogeneous}---
since for $d\in(0,1]$, the H\"older dual $d'$ 
is in $(-\infty,0)$ --- one is forced
to consider $\widetilde{f}\in 
\mathscr{H}_{d'}^{\mathbf{0}}(\bRn)$.
The singularity of $\widetilde{f}$
at the origin needs a different kind of treatment
than the pruning arguments provided 
in Section \ref{sec: pruning}. \\
Regarding 
Problem \ref{prob: homogeneous general},
the curvature-informed 
decomposition 
(via $\rho$ and its dyadic 
dilations)
underlying our method 
is rendered insufficient in the scenario
that 
$f\in  
\mathscr{H}_{d}(\bRn)$
because $\det\, H_f(\by)$ may vanish for $\by$
on a higher dimensional sub-variety. 
Instead one would need to work 
with a decomposition that is adapted
to the distance from
$$ 
\{\by \in \bRn: \det\, H_f(\by)=0\}.
$$
Returning to the situation when $f\in \sH_d^{\bzero}(\bRn)$,
we point out a few other interesting directions:
\begin{itemize}
    \item obtaining an asymptotic formula
    for $N_f(\delta,Q)$ 
    valid (uniformly) in the range 
    $\delta \in (Q^{\varepsilon-1},1)$, and
    \item establishing the size of
    $N_f(\delta,Q)$ (uniformly) in 
    the range $(Q^{\varepsilon(Q)-1},1)$
    so that $\varepsilon(Q) \rightarrow 0$
    as $Q\rightarrow \infty$.
\end{itemize}
Accomplishing either of these tasks involves
several related technical difficulties. 
For the second one, a notable hurdle is as follows:
because $\varepsilon$ is fixed
in the current argument, 
the stationary phase expansion 
could be cut-off 
at the point $t$ given by \eqref{def: t}.
In contrast, 
when $\varepsilon$ is allowed to vary
with $Q$ then
$t$ can no longer be uniformly bounded.
Aside from the obvious ramification
that the stationary phase analysis is far more 
intricate in this scenario, 
we observe that
the enveloping machinery in 
Section \ref{sec: envelope} 
(and the boot-strapping) 
will also become significantly
more involved. The underlying reason is 
that we now need good control over the $C^{A}$-norms, with $A$ being unbounded.
As a result, 
we expect the growth of,   
$(\Vert \rho \Vert_{C^A})_{A\geq 1}$ 
to play a crucial role in the analysis
when bounding the decay of 
Fourier transforms.
In other words, the exact
choice of the smoothing
functions $\omega, \rho$ will affect our
argument now.\\
The difficulty pertaining 
to the first problem listed
above is as follows. 
One might want to tighten
the constants from the bootstrapping process, 
by letting the support of $\rho$
shrink (with a certain rate depending on $Q$
and $f$). This is in contrast to just 
using finitely many functions $\rho$,
as we did in Section \ref{subsec local}. \\
To proceed in either of these directions,
a substantial part of our methods 
have to be reworked, extended,
and be made more elaborate. \\
\\
Since the present manuscript has already reached
a certain length, we decided 
against extending this 
work further. \\
\\
Of course, all the questions 
above make sense and are interesting 
also in the planar case $n=2$.
We will address this situation
in a forthcoming paper.

\subsubsection*{Acknowledgements}
R. S. was
supported by the Deutsche 
Forschungsgemeinschaft 
(DFG, German Research Foundation) 
under Germany’s Excellence Strategy 
- EXC-2047/1 -
390685813 as well as SFB 1060, and also in part by NSF grants DMS 1764295 and DMS 2054220.
N. T. was supported by
a Schr\"{o}dinger Fellowship 
of the Austrian Science Fund (FWF):
project J 4464-N, and thanks 
Maksym Radziwi{\l\l} for helpful discussions 
about analysis in general.
Both the authors are grateful to Andreas Seeger for 
enlightening conversations about 
oscillatory integrals, and to him, 
Victor Beresnevich and Damaris Schindler 
for their encouragement.
Furthermore, they thank 
Shuntaro Yamagishi for a careful reading of an earlier version of this manuscript and several helpful comments.
They also acknowledge the Hausdorff Research Institute of Mathematics and the organisers of the trimester program “Harmonic Analysis
and Analytic Number Theory” for a productive stay in the summer of 2021, when this project began. 
\begin{appendices}   
\section{Poisson summation, Smoothing and Partition of Unity} \label{app: smoothing}
We gather several standard results here. The first lemma is well-known to some
but perhaps not known to the extent 
that a general reader knows 
readily what a `smooth partition of unity' is.
In view of this, we provide:
\begin{lem}[Smooth partition of Unity]\label{lem: partition of unity}
There exists a smooth function $\omega: \bR \rightarrow [0,1]$
so that 
$$
\sum_{i\in \bZ} \omega(\frac{x}{2^i}) 
= 1, \quad 
\mathrm{for\, all}\, x\neq 0,
$$
while $\omega$ is supported in $[-4,-1]\cup [1,4]$.
\end{lem}
\begin{proof}
We follow the explicit construction 
from a work of Skriganov \cite{Skriganov}.
The function 
\[
\beta(x):=\Biggl(1-\exp\biggl(-\frac{1}{1-\frac{\vert x\vert}{2}}\biggr)\Biggr)\exp\biggl(\frac{1-\frac{\vert x\vert}{2}}{1-\vert x\vert}\biggr)
\]
is even, smooth and vanishes at $\vert x\vert=2$ to an infinite order.
Hence, 
\[
\omega (x):=\begin{cases}
0, & \mathrm{if}\,\,0\leq\vert x\vert<1,\\
\beta(x), & \mathrm{if}\,\,1<\vert x\vert<2,\\
1, & \mathrm{if}\,\,\vert x\vert=2,\\
1-\beta(\frac{x}{2}) & \mathrm{if}\,\,2<\vert x\vert<4,\\
0, & \mathrm{if}\,\,4\leq\vert x\vert
\end{cases}
\]
is smooth, and supported in 
$[-4,-1]\cup[1,4]$. 
Further, it is easily seen that
$0\leq \omega \leq 1$
Since $\omega $ is
even, we can without loss of generality verify $1=\sum_{j\in\mathbb{Z}}\omega (\frac{x}{2^{j}})$
for $x>0$. Notice that 
$x\mapsto\omega (\frac{x}{2^{j}})$ is supported
in 
\[
2^{j}[-4,-1]\cup[1,4]=[-2^{j+2},-2^{j}]\cup[2^{j},2^{j+2}].
\]
As a result, $2^{j_{0}}<x<2^{j_{0}+1}$ implies
\[
\sum_{j\in\mathbb{Z}}\omega 
(\frac{x}{2^{j}})=
\omega (\frac{x}{2^{j_{0}-1}})
+\psi(\frac{x}{2^{j_{0}}})
=1-
\beta(\frac{x}{2\cdot2^{j_{0}-1}})
+\beta(\frac{x}{2^{j_{0}}})=1.
\]
The case that $2^{j_{0}}=x$ can be argued in the same way.
\end{proof}
To we explain how 
Theorem \ref{thm: main smoothed}
implies Theorem \ref{thm: main}
we need the following basic result.

\begin{proof}[Proof that 
Theorem \ref{thm: main smoothed}
implies Theorem \ref{thm: main}]
At first, we are concerned with showing 
\begin{equation}\label{eq: upper bound count}
\rN_f(\delta,Q) \ll 
\delta Q^n
+ 
\Big(\frac{\delta}{Q}\Big)^{
\frac{n-1}{d}}Q^{n+k\varepsilon}
\,\,\mathrm{for\,any}\,\, Q\geq 1 \,\,
\mathrm{and}\,\, \delta \in 
(Q^{\varepsilon-1}, 1/2).
\end{equation}
Choose smooth and compactly supported 
functions $\om, b:\bR \rightarrow [0,1]$
so that 
$$
\ind_{[1,2)}(x) \leq \om(x), 
\,\, \mathrm{and} \,\,
\ind_{[0,1]}(y) \leq b(y)
\,\,\mathrm{for\,all\,} x,y\in \bR.
$$
By the compactness arguments from 
Subsection \ref{subsec local}, 
we see that there is finite collection
$\rho_1,\ldots,\rho_I$ 
and (strictly) 
positive numbers $\eta_1\ldots,\eta_I$
so that the following holds true. 
We have
\begin{equation}\label{eq: rho partition unity}
\sum_{\ell \geq 0 } \rho_1(2^\ell \bx) 
\leq \ind_{\sU_{n-1}\setminus\{\bzero\}}(\bx) 
\leq \sum_{\ell \geq 0 } \sum_{i\leq I} 
\rho_i(2^\ell \bx)
\,\, \mathrm{for\,any\,}\bx \in \bRn,
\end{equation}
while each $\rho_i$, for each $i\leq I$,
is $(f,\eta_i)$ admissible. 
Notice if $\ba/q=\bzero$ is the origin,
then
$$
\ind_{[1,2)}(\frac{q}{Q})
\ind_{\sU_{n-1}} (\bzero)
\ind_{[0,1]}
(\delta^{-1}\Vert q f(\bzero)\Vert) = 1
$$
for any $q\in [Q,2Q)$.
As a result, 
$$
\ind_{[1,2)}(\frac{q}{Q})
\ind_{\sU_{n-1}} (\frac{\ba}{q})
\ind_{[0,1]}
(\delta^{-1}\Vert q f(\frac{\ba}{q})\Vert)
\leq 1+ \om(\frac{q}{Q}) 
\Big( \sum_{\ell \geq 0 } \sum_{i\leq I} 
\rho_i(2^\ell \frac{\ba}{q})\Big)
b(\delta^{-1}\Vert q f(\frac{\ba}{q})\Vert)
$$
for any $\ba \in \bZn$, $q\in \bZ$, 
and $\delta\in (0,1/2)$.
By recalling
\eqref{eq; counting function hypersurfaces}
and \eqref{def: counting function},
we see that
$$
\rN_{f}(\delta,Q) 
\leq 
Q+ 
\sum_{i\leq I} N_f^{\omega,\rho_i,b}(\delta,Q).
$$
By applying Theorem \ref{thm: main smoothed}
the estimate \eqref{eq: upper bound count}
follows. Next, we 
prove the correspond lower bound 
\begin{equation}\label{eq: lower bound count}
\rN_f(\delta,Q) \gg
\delta Q^n
+ 
\Big(\frac{\delta}{Q}\Big)^{
\frac{n-1}{d}}Q^{n}
\,\,\mathrm{for\,any}\,\, Q\geq 1 \,\,
\mathrm{and}\,\, \delta \in 
(Q^{\varepsilon-1}, 1/2)
\end{equation}
to \eqref{eq: upper bound count}.
Now we choose smooth and compactly supported 
functions $\om_{-}, b_{-}:\bR \rightarrow [0,1]$
so that 
$$
\ind_{[1,2)}(x) \geq \om_{-}(x), 
\,\, \mathrm{and} \,\,
\ind_{[0,1]}(y) \geq b_{-}(y)
\,\,\mathrm{for\,all\,} x,y\in \bR .
$$
By the left most inequality 
in \eqref{eq: rho partition unity},
we conclude 
$\rN_{f}(\delta,Q) 
\geq 
\rN_f^{\omega_{-},\rho_1,b_{-}}(\delta,Q)$.
Theorem \ref{thm: main smoothed} 
implies \eqref{eq: lower bound count}.
Combining \eqref{eq: upper bound count}
and \eqref{eq: lower bound count},
completes the proof.
\end{proof}
Finally, we frequently
use multi-dimensional Poisson summation.
For the convenience of the reader,
we state:
\begin{lem}[Poisson Summation]
Let $m\in \bZ_{\geq 1}$. If $f:\bR^m \rightarrow \bR$
is a compactly supported $C^2$ function, then
$$
\sum_{\bn \in \bZ^m} f(\bn) = \sum_{\bk \in \bZ^m} \widehat{f}(\bk).
$$
\end{lem}
Further, we need fine-scale smoothing
which is a basic maneuver (e.g.
in the theory of correlation functions).
\begin{lem}\label{lem: truncated Poisson}
Let $b: \bR \rightarrow \bR$ 
be smooth and $\supp(b)\subseteq(-1/3,4/3)$.
If $\varepsilon,A>0$,
then 
$$
\bump\Big(\frac{\Vert x\Vert}
{\delta}\Big)
= 
\delta \widehat{\bump}(0)
+ \delta  
\sum_{1 \leq \vert j \vert 
\leq Q^{\varepsilon}/\delta} 
\widehat{\bump}(\delta j) e(jx)
+ O_{b,\varepsilon,A}(Q^{-A})
$$
uniformly in $\delta \in (0,1/2)$ 
and $x\in \bR$.
\end{lem}
\begin{proof}
Observe that
for any $x\in \bR$
there exists at most one
$m=m(x)$ so that $\Vert x\Vert 
= x - m \leq \delta <1/2$.
As $\supp(b)$ is contained in an 
interval of length at most 
$5/3<1/\delta$ we see that
$$
\bump\Big(\frac{\Vert x\Vert}
{\delta}\Big) = 
\sum_{m\in \bZ} b\Big(\frac{x+m}{\delta}\Big)
$$
for any $x\in \bR$.
By Poisson summation 
the right hand side equals
$$
\sum_{j\in \bZ} \int_{\bR} 
 b\left(\frac{x+y}{\delta}\right) e(-jy) \rd y.
$$
For each $j\in \mathbb{Z}$, using two changes of variables 
($y\mapsto y-x$ and $y\mapsto \delta y$), we conclude 
$$
\int_{\bR} b\left(\frac{x+y}{\delta}\right) 
e(-jy) \rd x = 
\delta e(jx)\int_{\bR} b(y) 
e(-j\delta y) \rd y = 
\delta e(jx) \widehat{b}(j\delta).
$$
By partial summation and a straightforward 
dyadic decomposition, 
the tail $\vert j \vert >
Q^{\varepsilon}/\delta$
contributes $O(Q^{-100})$.
Since the zero mode $j=0$ evaluates to 
$\delta \widehat{b}(0)$
the proof is complete.
\end{proof}

\section{Proof of 
Theorem \ref{thm: dio app low degree} 
and Theorem \ref{thm: dio app high degree}}\label{app: Dio}
This reduction is well-known,
and we include it for the sake of completeness
claiming no novelty here.
We follow the exposition
of \cite[Section 9]{Huang rational points} closely
modifying it when necessary for our setting.

It suffices to prove 
Theorem \ref{thm: dio app low degree} 
and Theorem \ref{thm: dio app high degree}
locally on a given chart. 
We begin with proving the former theorem.
Once this is done, establishing the latter theorem
will be a simple change in the exponents 
in a few places.

\subsection{Proof of Theorem \ref{thm: dio app low degree}}
Recall that $\sU_{n-1}$
is the (Euclidean) unit ball in 
$\bR^{n-1}$.
Without loss of
generality, 
we assume that 
$\sM$ is given 
in the normalised 
 Monge parametrisation 
 \eqref{def: normalised Monge}.
Let $\Omega_n(f, \psi)
\subset \bRn$ denote the projection 
of $\mathcal{S}\cap 
\mathscr{S}_n(\psi)\subset \bR^n$ 
onto $\sC$
defined by forgetting the last coordinate.
Since this projection is bi-Lipschitz, 
$\mathcal{H}^s(\mathcal{S}\cap \mathscr{S}_n(\psi)) =0$ 
if and only if
$\mathcal{H}^s(\Omega_n(f, \psi))=0$.
Next, we show that if $s>\frac{n-1}2$ then
$\mathcal{H}^s(\Omega_n(f,\psi))=0$ under the condition 
\begin{equation}\label{eq: Hausdorff convergence low degre}
\sum_{q\geq 1}\psi(q)^{s+1}q^{n-1-s}<\infty.
\end{equation}
First of all, $\mathcal{H}^s(\Omega_n(f,\psi))=0$ automatically for all $s>n-1$. So we may assume that $s\le n-1$. Next choose $\eta>0$ such that $\eta<(2s-n+1)/(s+1)$. It is easily verified that there is no loss of generality in assuming that 
\begin{equation}\label{eq: above conv}
\psi(q)\ge q^{-1+\eta} \quad\text{for all }q\geq 1,
\end{equation}
because if \eqref{eq: above conv} 
fails, then we replace $\psi$ by 
$\hat\psi(q):=\max(\psi(q), q^{-1+\eta})$ 
which satisfies \eqref{eq: above conv}.
We observe that $\Omega_n(f,\psi)$ 
consists of those points 
$\mathbf{x}\in \sC$ so that
\[
\max_{1\leq i\leq n-1}
|x_i-\frac{a_i}q|<\frac{\psi(q)}q,
\quad \mathrm{as\,well\,as}\quad 
|f(\mathbf{x})
-\frac{\fa}q| < \frac{\psi(q)}q
\]
is satisfied for infinitely many 
$(q, \mathbf{a}, \fa)\in\bN\times\bZ^{n-1}\times \bZ$. 
For $\frac{\mathbf{p}}q\in\bQ^n$,
where 
$\mathbf{p}=(\mathbf{a},\fa)\in\bZ^{n-1}\times\bZ$, 
let $\sigma(\mathbf{p}/q)$ denote the set of 
$\mathbf{x}\in \sC$ satisfying the above inequality. 
If we show that
\begin{equation}\label{eq: Hausdorff measures converge}
\sum_{q\geq 1}
\sum_{\mathbf{p}\in\bZ^{n}} 
\text{diam}\left(
\sigma\left(\frac{\mathbf{p}}
q\right)\right)^s
<\infty,
\end{equation}
then 
the Hausdorff-Cantelli Lemma 
\cite[p. 68]{BeD} 
implies that 
$\mathcal{H}^s(\Omega_n(f,\psi))=0$ 
and thus the proof would be complete.
In order to verify
\eqref{eq: Hausdorff measures converge}
we use the trivial estimate
\begin{equation}\label{eq: diameter simple bound}
\text{diam}
\left(\sigma\left(\frac{\textbf{p}}q\right)\right)
\ll \frac{\psi(q)}q
\end{equation}
where the implied 
constant only depends on $n$. Furthermore, the dyadic decomposition
$$
\sum_{q\geq 1}
\sum_{\mathbf{p}\in\bZ^{n}} 
\text{diam}\left(
\sigma\left(\frac{\mathbf{p}}
q\right)\right)^s
=
\sum_{i\geq 0}
\sum_{\substack{
2^i\leq q <2^{i+1}\\
\mathbf{p}\in\bZ^{n}\\
\sigma\left(\frac{\mathbf{p}}
q\right) \neq \emptyset} }
\text{diam}\left(
\sigma\left(\frac{\mathbf{p}}
q\right)\right)^s
\ll 
\sum_{i\geq 0}
\Big(\frac{\psi(2^i)}{2^i}\Big)^s \#\sD_i
$$
where $\sD_i:=\{\mathbf{p}/q\in\bQ^{n}: 
2^i\le q<2^{i+1}, 
\sigma(\mathbf{p}/q)\not=\emptyset\}$
will be useful. 
Suppose
$\sigma(\mathbf{p}/q)\not=\emptyset$ 
and
$\textbf{x}\in\sigma(\mathbf{p}/q)$. 
Being a smooth function,
$f$ is Lipschitz-continuous
(on any fixed compact interval).
Hence,
\begin{align*}
\left|f\left(\frac{\textbf{a}}q\right)
-\frac{\fa}q\right|
&\le\left|f(\mathbf{x})
-\frac{\fa}q\right|+\left|
f\left(\frac{\textbf{a}}q\right)
-f(\mathbf{x})\right|
\le \frac{\psi(q)}q+c_3\left|
\mathbf{x}-\frac{\mathbf{a}}q\right|
\le c_4\frac{\psi(q)}q,
\end{align*}
where $c_3$ and $c_4$ 
are constants that depend only on $f$. 
Thus, for $i\ge0$, we have
\begin{align*}
\#\sD_i
\le
\#\{(\ba, \fa)/q\in\bQ^{n}: 
2^i\le q<2^{i+1}, \mathbf{a}/q\in\sU_{n-1}, 
|f(\mathbf{a}/q)-\fa/q|\le 
c_4\psi(q)/q\}.
\end{align*}
Because $\psi$ is non-increasing,
we are guaranteed $\psi(q)\leq \psi(2^i)$
for $2^i\leq q <2^{i+1}$. As a result,
$$
\#\sD_i
\leq 
\#
\{\mathbf{a}/q\in
\bQ^{n-1}: 2^i\leq q<2^{i+1}, 
\mathbf{a}/q\in \sU_{n-1}, 
\|qf (\mathbf{a}/q)\|
\le c_4\psi(2^i)
\}.
$$ In view of \eqref{eq: above conv}, 
Theorem \ref{thm: main} implies
\begin{equation}\label{eq: bound on sets}
\#\sD_i\ll \psi(2^i)2^{ni}.
\end{equation}
Recall that $\Omega_n(f,\psi)$ 
consists of $\mathbf{x}\in \sU_{n-1}$ 
lying in infinitely 
many sets $\sigma(\mathbf{p}/q)$. 
Utilizing \eqref{eq: diameter simple bound}
and \eqref{eq: bound on sets}, we have
\begin{align*}
\sum_{q\geq 1}
\sum_{\mathbf{p}\in\bZ^{n}} 
\text{diam}\left(
\sigma\left(\frac{\mathbf{p}}
q\right)\right)^s
& \ll
 \sum_{i\geq 0} 
 \#\{\mathbf{p}/q\in\bQ^{n}: 2^i
\le q<2^{i+1}, 
\sigma(\mathbf{p}/q)\not=\emptyset\}
(\psi(2^i))^s
2^{-is}\\
& \ll 
\sum_{i\geq 0} 
(\psi(2^i))^{1+s}2^{i(n-s)}
\ll \sum_{q\geq 1} 
(\psi(q))^{1+s}q^{n-s-1}.
\end{align*}
Due to \eqref{eq: Hausdorff convergence low degre},
the right hand side converges,
completing the proof of 
Theorem \ref{thm: dio app low degree}.

\subsection{Proof of Theorem \ref{thm: dio app high degree}}
As mentioned earlier, 
we argue very similarly as
in the proof 
Theorem \ref{thm: dio app low degree}. 
Therefore we recycle the notation from that proof,
and proceed to describe 
the necessary modifications.
First of all, instead of 
\eqref{eq: Hausdorff convergence low degre}
we show now that  
\begin{equation}\label{eq: convergence assumption large degree}
    \sum_{q\geq 1}(\psi(q))^{s+1}q^{n-1-s}<\infty,
    \,\,\, 
    \mathrm{and}
    \,\,\,
    \sum_{q\geq 1}(\psi(q))^{s+\frac{n-1}{d}}
    q^{n+k\varepsilon-s-\frac{n-1}{d}-1}<\infty
\end{equation}
implies $\mathcal{H}^s(\Omega_n(f,\psi))=0$. 
Theorem \ref{thm: main smoothed} 
implies in the present situation
$$
\# \sD_i \ll \psi(2^i) 2^{in} 
+ \Big( \frac{\psi(2^i)}{2^i}\Big)^{\frac{n-1}{d}}
2^{i (n+k\varepsilon)}
$$
in place of \eqref{eq: bound on sets}.
Arguing as before, we deduce 
$$
\sum_{q\geq 1}
\sum_{\mathbf{p}\in\bZ^{n}} 
\text{diam}\left(
\sigma\left(\frac{\mathbf{p}}
q\right)\right)^s
\ll \sum_{q\geq 1} 
(\psi(q))^{1+s}q^{n-s-1}
+ 
\sum_{q\geq 1}(\psi(q))^{s+\frac{n-1}{d}}
    q^{n+k\varepsilon-s-\frac{n-1}{d}-1}
$$
By \eqref{eq: convergence assumption large degree}
the right hand side converges. 
Thus the proof of 
Theorem \ref{thm: dio app high degree} is complete.
\section{Proof of 
Theorem \ref{thm: dimension growth app}}\label{app: dimension growth}
The following bi-rational projection
argument is standard in the context 
of the dimension growth conjecture. 
Recall the notation from 
Definition \ref{def: class of addmissible manifolds}.
In view of the compactness of the manifold,
it suffices to prove the theorem 
when for the manifold the first 
$(n-\rc)$-coordinates of the manifold
in question are restricted to 
a small open ball $\sB(\bx_0,r)$. 
Denote this local piece of $\sM$
by $\sM_{\bx_0,r}$.
For the ease of exposition we detail
only the case $\bx_0 = \bzero$
as the general case can be done similarly
by shifting appropriately by $\bx_0$
throughout. Define 
$P: \sB(\bzero,r) \rightarrow \bR^{n-\rc +1}$
via
$ P(\bx):= (\bx, f(\bx))$
where $f(\bx) := \langle 
\bf(\bx),\bs/ Q_0\rangle$.
There are two useful observations to make now. 
Firstly, if $(\bx,\bf(\bx)) 
\in \sM_{\bzero,r} \cap q^{-1} \bZ^n$
then 
$\bx = \ba/q$  
for some integer vector $\ba$ 
and $ \Vert  Q_0 q f(\ba/q) \Vert =0$.
By slackening this equality to
the inequality $ \Vert  Q_0 q f(\ba/q) \Vert 
\leq Q^{\varepsilon-1}$ we infer 
$$
\rN_{\sM_{\bzero,r}}(B) 
\leq 
N_{Q_0f}(B,0)
\leq 
N_{Q_0f}(B,Q^{\varepsilon-1}).
$$
Applying Theorem \ref{thm: main}
implies the claim.
\end{appendices}

\newpage
\pagestyle{plain}
\printindex
\end{document}